\documentclass[11pt]{article}
\usepackage[letterpaper,margin=1.00in]{geometry}
\usepackage{float}

\usepackage{amscd}
\usepackage{mathrsfs}
\usepackage[IL2]{fontenc}

\usepackage{comment} % \begin{comment}\end{comment} to comment stuff that may be uncommented later on

\usepackage{color}
\usepackage{amssymb,amsmath,amsthm,amsfonts,latexsym}
\usepackage[utf8x]{inputenc}
\usepackage{bbm} %this is so we can use mathbb with numbers, but using the command bbm

\usepackage{graphicx}
\usepackage[disable]{todonotes}
\usepackage{amsmath,amsfonts,amssymb,amsthm} 
\usepackage{tikz}
\usepackage{hyperref}
\usepackage{thmtools}
\usepackage{thm-restate}
\newtheorem{theorem}{Theorem}[section]

\newtheorem*{claim*}{Claim}
\newtheorem*{theorem*}{Theorem}
\newtheorem*{definition*}{Definition}

\newtheorem{corollary}[theorem]{Corollary}
\newtheorem{lemma}[theorem]{Lemma}
\newtheorem{remark}[theorem]{Remark}
\newtheorem{claim}[theorem]{Claim}

\newtheorem{proposition}[theorem]{Proposition}

\newtheorem{definition}[theorem]{Definition}

\newcommand{\comm}[1]{}
\usepackage[capitalize, nameinlink]{cleveref}
\crefname{theorem}{Theorem}{Theorems}
\crefname{proposition}{Proposition}{Propositions}
\crefname{observation}{Observation}{Observations}
\crefname{lemma}{Lemma}{Lemmas}
\crefname{claim}{Claim}{Claims}
\crefname{problem}{Problem}{Problems}
\crefname{conjecture}{Conjecture}{Conjectures}
\crefname{question}{Question}{Questions}
\crefname{example}{Example}{Examples}
\crefname{fact}{Fact}{Facts}

\newcommand{\rake}{{\sf Rake}}
\newcommand{\compress}{{\sf Compress}}
\newcommand{\GC}[1]{{G}_{#1}^\textsf{C}}
\newcommand{\GR}[1]{{G}_{#1}^\textsf{R}}
\newcommand{\VC}[1]{{V}_{#1}^\textsf{C}}
\newcommand{\VR}[1]{{V}_{#1}^\textsf{R}}

\newcommand{\yes}{\textsf{YES}}
\newcommand{\no}{\textsf{NO}}

\newcommand{\Lpump}{ {\ell_{\operatorname{pump}}} }
\newcommand{\simm}{\overset{\star}{\sim}}
\newcommand{\type}{\textsf{Class}_2}
\newcommand{\class}{\textsf{Class}_1}
\renewcommand{\deg}{\operatorname{deg}}
\newcommand{\dom}{\operatorname{dom}}

\newcommand{\local}{\mathsf{LOCAL}}
\newcommand{\localo}{\mathsf{OLOCAL}}
\newcommand{\olocal}{\mathsf{OLOCAL}}

\newcommand{\LOCAL}{\mathsf{LOCAL}}

\newcommand{\dlocal}{\mathsf{DLOCAL}}

\newcommand{\rlocal}{\mathsf{RLOCAL}}

\newcommand{\CONT}{\mathsf{CONTINUOUS}}
\newcommand{\cont}{\mathsf{CONTINUOUS}}

\newcommand{\TOAST}{\mathsf{TOAST}}

\newcommand{\toast}{\TOAST}

\newcommand{\BOREL}{\mathsf{BOREL}}

\newcommand{\MEASURE}{\mathsf{MEASURE}}

\newcommand{\fiid}{\mathsf{fiid}}
\newcommand{\ffiid}{\mathsf{ffiid}}
\newcommand{\borel}{\mathsf{BOREL}}

\newcommand{\continuous}{\mathsf{CONTINUOUS}}
\newcommand{\baire}{\mathsf{BAIRE}}
\newcommand{\measure}{\mathsf{MEASURE}}

\renewcommand{\P}{\textrm{P}}
\newcommand{\E}{\textrm{\textbf{E}}}

\newcommand{\poly}{\operatorname{poly}}

\newcommand{\eps}{\varepsilon}

\newcommand{\fA}{\mathcal{A}}

\newcommand{\fD}{\mathcal{D}}
\newcommand{\fE}{\mathcal{E}}

\newcommand{\fG}{\mathcal{G}}
\newcommand{\fH}{\mathcal{H}}
\newcommand{\fI}{\mathcal{I}}

\newcommand{\fM}{\mathcal{M}}

\newcommand{\fO}{\mathcal{O}}

\newcommand{\fT}{\mathcal{T}}

\newcommand{\fV}{\mathcal{V}}

\newcommand{\ta}{\mathtt{a}}
\newcommand{\tb}{\mathtt{b}}
\newcommand{\tc}{\mathtt{c}}
\newcommand{\td}{\mathtt{d}}

\newcommand{\tx}{\mathtt{x}}

\newcommand{\mP}{\mathbb{P}}

\newcommand{\idgraphcolored}[3]{\mathbf{H}_{#1,#2,#3}}
\newcommand{\idgraph}[3]{H_{#1,#2,#3}}
\newcommand{\idgraphingcolored}[1]{\mathbf{H}_{#1}}
\newcommand{\idgraphing}[1]{\mathcal{H}_{#1}}

\def\leukfrac#1/#2{\leavevmode
               \kern.1em
                \raise.9ex\hbox{\the\scriptfont0 ${}_#1$}
                \hskip -1pt\kern-.1em
                /\kern-.15em\lower.10ex\hbox{\the\scriptfont0 ${}_#2$}}

\makeatletter
\def\diam{\mathop{\operator@font diam}\nolimits}
\makeatother

\newcommand{\Aut}{\operatorname{Aut}}
\newcommand{\perfmatch}{\Pi_{\text{pm}}}
\newcommand{\twomatch}{\Pi^2_{\text{pm}}}
\newcommand{\edgegrab}{\Pi_{\text{edgegrab}}}
\newcommand{\deltacol}{\Pi_{\chi,\Delta}}
\newcommand{\vizing}{\Pi_{\chi',\Delta+1}}
\newcommand{\PI}{\operatorname{Alice}}
\newcommand{\PK}{\operatorname{Bob}}

\title{\vspace{-1cm} Local Problems on Trees from the Perspectives of Distributed Algorithms, Finitary Factors, and Descriptive Combinatorics}

\begin{document}

\newcommand*\samethanks[1][\value{footnote}]{\footnotemark[#1]}

\author{
Sebastian Brandt \\
\small ETH Z\"{u}rich \\
\small \texttt{brandts@ethz.ch}\\
\and 
Yi-Jun Chang\thanks{Supported by  Dr.~Max R\"{o}ssler, by the Walter Haefner Foundation, and by the ETH Z\"{u}rich Foundation.} \\
\small ETH Z\"{u}rich \\
\small \texttt{yi-jun.chang@eth-its.ethz.ch} \\
\and 
Jan Greb\'{i}k\thanks{Supported by Leverhulme Research Project Grant RPG-2018-424.} 
\\
\small University of Warwick\\
\small \texttt{Jan.Grebik@warwick.ac.uk}\\
\and 
Christoph Grunau\thanks{Supported by the European Research Council (ERC) under the European
Unions Horizon 2020 research and innovation programme (grant agreement No.~853109). 
}\\
\small ETH Z\"{u}rich \\
\small \texttt{cgrunau@inf.ethz.ch}\\
\and 
V\'{a}clav Rozho\v{n}\samethanks \\
\small ETH Z\"{u}rich \\
\small \texttt{rozhonv@ethz.ch}\\
\and 
Zolt\'{a}n Vidny\'{a}nszky\thanks{Partially supported by the
		National Research, Development and Innovation Office
		-- NKFIH, grants no.~113047, no.~129211 and FWF Grant M2779.}\\
\small California Institute of Technology\\
\small \texttt{vidnyanz@caltech.edu}\\}

\date{}

\maketitle 
\pagenumbering{gobble}   

\vspace{-1cm}
\begin{abstract}
    %In this work, 
    We study connections between three different fields: distributed local
    algorithms, finitary factors of iid processes, and descriptive combinatorics.
    We focus on two central questions:
    Can we apply techniques from one of the areas to obtain results in another?
    Can we show that complexity classes coming from different areas contain precisely the same problems?
    We give an affirmative answer to both questions in the context of local problems on regular trees:
    %More precisely, we show the following.
    
    \begin{enumerate}
        \item We extend the Borel determinacy technique of Marks [Marks -- J. Am. Math. Soc. 2016] coming from descriptive combinatorics and adapt it to the area of distributed computing, thereby obtaining a more generally applicable lower bound technique in descriptive combinatorics and an entirely new lower bound technique for distributed algorithms.
        Using our new technique, we prove deterministic distributed $\Omega(\log n)$-round lower bounds for problems from a natural class of homomorphism problems.
        Interestingly, these lower bounds seem beyond the current reach of the powerful round elimination technique [Brandt -- PODC 2019] responsible for all substantial locality lower bounds of the last years.
        Our key technical ingredient is a novel \emph{ID graph} technique that we expect to be of independent interest; in fact, it has already played an important role in a new  (follow-up)
        lower bound for the Lovász local lemma in the Local Computation Algorithms model from sequential computing [Brandt, Grunau, Rozhoň -- PODC 2021]. 

        \item We prove that a local problem admits a Baire measurable coloring if and only if it admits a local algorithm with local complexity $O(\log n)$, extending the classification of Baire measurable colorings of Bernshteyn [Bernshteyn -- personal communication].
        A key ingredient of the proof is a new and simple characterization of local problems that can be solved in $O(\log n)$ rounds.
        We complement this result by showing separations between complexity classes from distributed computing, finitary factors, and descriptive combinatorics. 
        Most notably, the class of problems that allow a distributed algorithm with sublogarithmic randomized local complexity is incomparable with the class of problems with a Borel solution. % and the complexity classes coming from finitary factors gives finer scale to measure complexity classes than distributed computing\todo{smthing like this :)}. 
    \end{enumerate}
    %In general, 
    We hope that our treatment will help to view all three perspectives as part of a common theory of locality, in which we follow the insightful paper of [Bernshteyn -- arXiv 2004.04905].
\end{abstract}

\newpage
\tableofcontents

\newpage
\pagenumbering{arabic}   

\section{Introduction}
\label{sec:intro}
In this work, we study local problems on regular trees from three different perspectives. 

\textit{First,} we consider the perspective of distributed algorithms. In distributed computing, the studied setup is a network of computers where each computer can only communicate with its neighbors. Roughly speaking, the question of interest in this area is which problems can be solved with only a few rounds of communication in the underlying network. 

\textit{Second}, we consider the perspective of (finitary) factors of iid processes.
In probability, random processes model systems that appear to vary in a random manner.
%where the randomness develops over time.
These include Bernoulli processes, Random walks etc.
A particular, well-studied, example is the Ising model.

\textit{Third}, we investigate the perspective of descriptive combinatorics. The goal of this area is to understand which constructions on infinite graphs can be performed without using the so-called axiom of choice.

Although many of the questions of interest asked in these three areas are quite similar to each other, no systematic connections were known until an insightful paper of Bernshteyn~\cite{Bernshteyn2021LLL} who showed that results from distributed computing can automatically imply results in descriptive combinatorics.
In this work, we show that the connections between the three areas run much deeper than previously known, both in terms of techniques and in terms of complexity classes.
In fact, our work suggests that it is quite useful to consider all three perspectives as part of a common theory, and we will attempt to present our results accordingly.
We refer the reader to \cref{fig:big_picture_trees} for a partial overview of the rich connections between the three perspectives, some of which are proven in this paper. 

In this work, we focus on the case where the graph under consideration is a regular tree.
Despite its simplistic appearance, regular trees play an important role in each of the three areas, as we will substantiate at the end of this section.
To already provide an example, in the area of distributed algorithms, almost all locality lower bounds are achieved on regular trees.
Moreover, when regarding lower bounds, the property that they already apply on regular trees actually \emph{strengthens} the result---a fact that is quite relevant for our work as our main contribution regarding the transfer of techniques between the areas is a new lower bound technique in the area of distributed computation that is an adaptation and generalization of a technique from descriptive combinatorics.
Regarding our results about the relations between complexity classes from the three areas, we note that such connections are also studied in the context of paths and grids in two parallel papers \cite{grebik_rozhon2021LCL_on_paths,grebik_rozhon2021toasts_and_tails}.

In the remainder of this section, we give a high-level overview of the three areas that we study.
The purpose of these overviews is to provide the reader with a comprehensive picture of the studied settings that can also serve as a starting point for delving deeper into selected topics in those areas---in order to follow our paper, it is not necessary to obtain a detailed understanding of the results and connections presented in the overviews. Necessary technical details will be provided in \cref{sec:preliminaries}.
Moreover, in \cref{sec:contribution}, we present our contributions in detail.

\paragraph{Distributed Computing}
The definition of the $\local$ model of distributed computing by Linial~\cite{linial92LOCAL} was motivated by the desire to understand distributed algorithms in huge networks. 
As an example, consider a huge network of wifi routers. Let us think of two routers as connected by an edge if they are close enough to exchange messages. 
It is desirable that such close-by routers communicate with user devices on different channels to avoid interference.
In graph-theoretic language, we want to properly color the underlying network. 
Even if we are allowed a color palette with $\Delta+1$ colors where $\Delta$ denotes the maximum degree of the graph (which would admit a simple greedy algorithm in a sequential setting), the problem remains highly interesting in the distributed setting, as, ideally, each vertex decides on its output color after only a few rounds of communication with its neighbors, which does not allow for a simple greedy solution. 

The $\local$ model of distributed computing formalizes this setup: we have a large network, where each vertex knows the network's size, $n$, and perhaps some other parameters like the maximum degree $\Delta$. In the case of randomized algorithms, each vertex has access to a private random bit string, while in the case of deterministic algorithms, each vertex is equipped with a unique identifier from a range polynomial in the size $n$ of the network. 
In one round, each vertex can exchange any message with its neighbors and can perform an arbitrary computation. The goal is to find a solution to a given problem in as few communication rounds as possible. 
As the allowed message size is unbounded, a $t$-round $\local$ algorithm can be equivalently described as a function that maps $t$-hop neighborhoods to outputs---the output of a vertex is then simply the output its $t$-hop neighborhood %is
mapped to by this function.
An algorithm is correct if and only if the collection of outputs at all vertices constitutes a correct solution to the problem. 

There is a rich theory of distributed algorithms and the local complexity of many problems is understood. 
The case of trees is a highlight of the theory: it is known that any local problem (a class of natural problems we will define later) belongs to one of only very few complexity classes.
More precisely, for any local problem, its randomized local complexity is either $O(1), \Theta(\log^* n), \Theta(\log\log n), \Theta(\log n)$, or $\Theta(n^{1/k})$ for some $k \in \mathbb{N}$. Moreover, the deterministic complexity is always the same as the randomized one, except for the case $\Theta(\log\log n)$, for which the corresponding deterministic complexity is $\Theta(\log n)$ (see \cref{fig:big_picture_trees}).

\begin{figure}
    \centering
    \includegraphics[width=\textwidth]{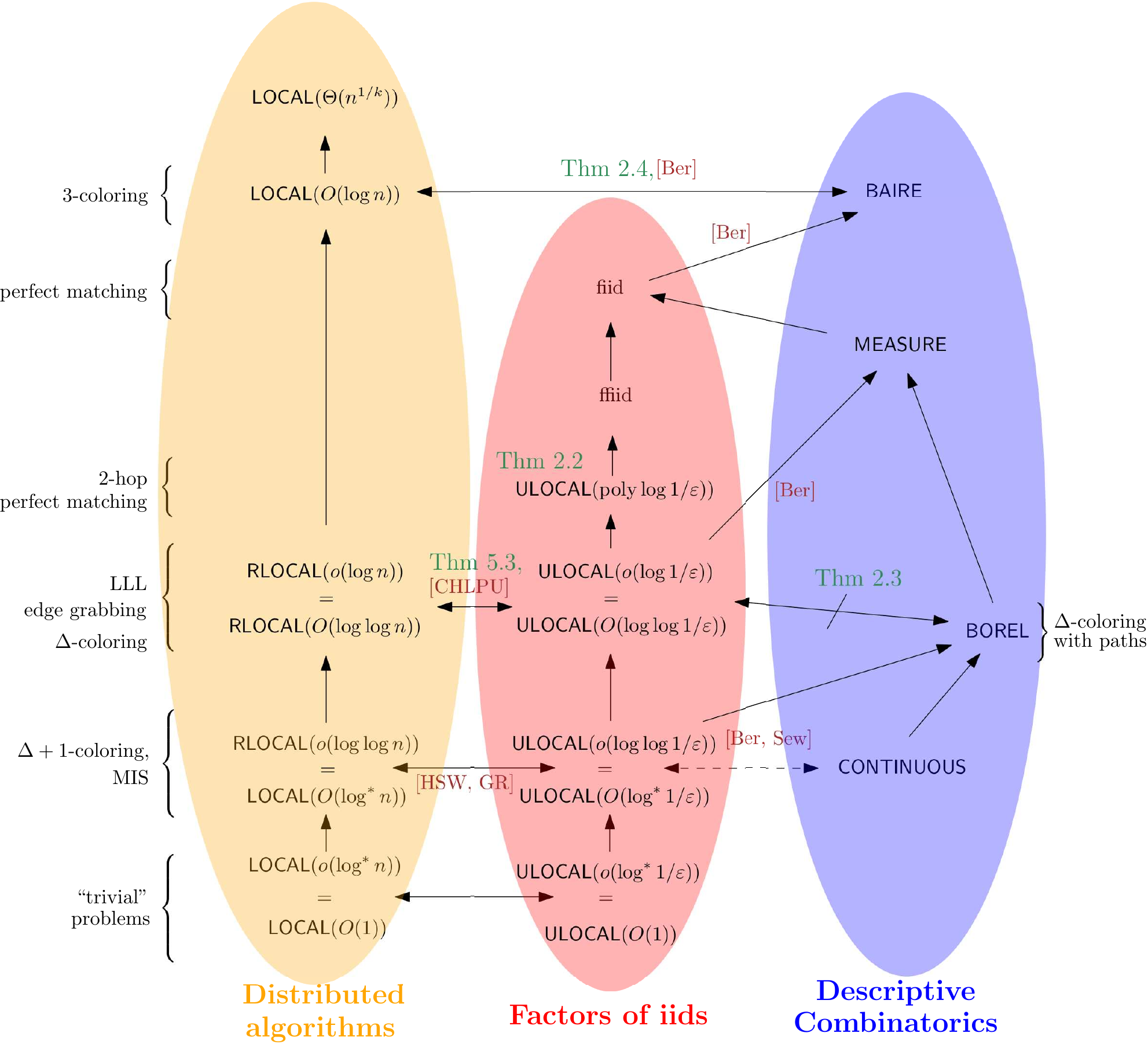}
    \caption{Complexity classes on regular trees considered in the three areas of distributed computing, factors of iid processes/uniform algorithms, and descriptive combinatorics. The left part shows complexity classes of distributed computing. We use the shorthand $\local$ if it does not matter whether we talk about the deterministic or randomized complexity. These two notions differ only for the class of problems of randomized local complexity $O(\log\log n)$, which have deterministic complexity $O(\log n)$. \\
    The uniform complexity classes of sublogarithmic complexity are in correspondence to appropriate classes in the randomized local complexity model, as proven in \cref{sec:separating_examples}. On the other hand, the class fiid is very similar to the class $\measure$ from descriptive combinatorics. The equivalence of the class $\cont$ and $\local(O(\log^* n)) = \olocal(O(\log^* 1/\eps))$ is marked with a dashed arrow as it was proven in case the tree is generated by a group action (think of the tree being equipped with an additional $\Delta$-edge coloring). The inclusion $\local(O(\log^* n)) \subseteq \borel$ however clearly holds also in our setting. The class $\borel$ is incomparable with $\rlocal(O(\log\log n))$, as proven in \cref{sec:separating_examples}.   }
    \label{fig:big_picture_trees}
\end{figure}

\todo{thm 5.1 - 5.3}

\paragraph{(Finitary) Factors of iid Processes and Uniform Algorithms}

In recent years, \emph{factors of iid (fiid) processes} on trees attracted a lot of attention in combinatorics, probability, ergodic theory and statistical physics \cite{ABGGMP,ArankaFeriLaci,BackSzeg,BackVir,backhausz,backhausz2,BGHV,Ball,Bowen1,Bowen2,CPRR,csoka2015invariant,GabLyo,GamarnikSudan,GGP,HarangiVirag,HolroydPoisson,HLS,HPPS,HolroydPeres,HoppenWormald,Kun,RahmanPercolation,RahVir,LyonsTrees,Mester,schmidt2016circular,Timar1,Timar2}.
Intuitively, factors of iid processes are randomized algorithms on, e.g., infinite $\Delta$-regular trees, where each vertex outputs a solution to a problem after it explores random strings on vertices of the whole tree.
As an example, consider the \emph{perfect matching} problem.
An easy parity argument shows that perfect matching cannot be solved by any local randomized algorithm on finite trees.
However, if we allow a small fraction of vertices not to be matched, then, by a result of Nguyen and Onak \cite{nguyen_onak2008matching} (see also \cite{ElekLippner}), there is a constant-round randomized algorithm that produces such a matching on high-girth graphs (where the constant depends on the fraction of unmatched vertices that we allow).
This result can also be deduced from a result of Lyons and Nazarov \cite{lyons2011perfect}, who showed that perfect matching can be described as a factor of iid process on an infinite $\Delta$-regular tree.
The high-level idea behind this connection is that high-girth graphs approximate the infinite $\Delta$-regular tree and constant-round local algorithms approximate factors of iid processes.
This correspondence is formalized in the notion of \emph{Benjamini-Schramm} or \emph{local-global} convergence \cite{benjamini2011recurrence, hatamilovaszszegedy}.
We note that getting only ``approximate'' solutions, that is, solutions where a small fraction of vertices does not have to satisfy the constraints of a given problem, is intrinsic in this correspondence.
Regardless, there are many techniques, such as entropy inequality \cite{backhausz} or correlation decay \cite{backhausz2}, and particular results such as the aforementioned perfect matching problem \cite{lyons2011perfect} that provide lower and upper bounds, respectively, in our setting as well.
We refer the reader to \cite{LyonsTrees,BGHV} for a comprehensive summary of the field.

In this paper, we mostly consider a stronger condition than fiid, namely so-called \emph{finitary} factors of iid (ffiid) processes that are studied in the same context as fiid \cite{holroyd2017one_dependent_coloring,HolroydSchrammWilson2017FinitaryColoring,Spinka}. 
Perhaps surprisingly, the notion of ffiid is identical to the notion of so-called uniform distributed randomized algorithms \cite{Korman_Sereni_Viennot2012Pruning_algorithms_+_oblivious_coloring,grebik_rozhon2021toasts_and_tails} that we now describe. 
We define an uniform local algorithm as a randomized local algorithm that does not know the size of the graph $n$ -- this enables us to run such an algorithm on infinite graphs, where there is no $n$. 
More precisely, we require that each vertex eventually outputs a solution that is compatible with the output in its neighborhood, but the time until the vertex finishes is a potentially unbounded random variable. 
As in the case of classical randomized algorithms, we can now measure the \emph{uniform complexity} of an uniform local algorithm (known as the tail decay of ffiid \cite{HolroydSchrammWilson2017FinitaryColoring}). 
The uniform complexity of an algorithm is defined as the function $t(\eps)$ such that the probability that the algorithm run on a specific vertex needs to see outside its $t(\eps)$-hop neighborhood is at most $\eps$. 
As in the case of classical local complexity, there is a whole hierarchy of possible uniform complexities (see \cref{fig:big_picture_trees}). 

We remark that uniform distributed local algorithms can be regarded as Las Vegas algorithms. The output will always be correct; there is however no fixed guarantee at what point all vertices have computed their final output. On the other hand, a randomized distributed local algorithm can be viewed as a Monte Carlo algorithm as it needs to produce an output after a fixed number of rounds, though the produced output might be incorrect.

\paragraph{Descriptive Combinatorics}
    %If one generalizes finite graph-theoretic notions to infinite graphs in the most straightforward way, thanks to the axiom of choice, the latter often exhibit counterintuitive, almost paradoxical behavior.  %(see, e.g., the Banach-Tarski paradox below). 
    %The natural way to eliminate this type of behavior is to impose \textit{definability constraints} on the infinite side. 
	
	The Banach-Tarski paradox states that a three-dimensional ball of unit volume can be decomposed into finitely many pieces that can be moved by isometries (distance preserving transformations such as rotations and translations) to form two three-dimensional balls each of them with unit volume(!).
    The graph theoretic problem lurking behind this paradox is the following: fix finitely many isometries of $\mathbb{R}^3$ and then consider a graph where $x$ and $y$ are connected if there is an isometry that sends $x$ to $y$.
	Then our task becomes to find a perfect matching in the appropriate subgraph of this graph -- namely, the bipartite subgraph where one partition contains points of the first ball and the other contains points of the other two balls. 
	Banach and Tarski have shown that, with a suitably chosen set of isometries, the axiom of choice implies the existence of such a matching. 
	In contrast, since isometries preserve the Lebesgue measure, the pieces in the decomposition cannot be Lebesgue measurable. 
	Surprisingly, Dougherty and Foreman \cite{doughertyforeman} proved that the pieces in the Banach-Tarski paradox can have the \emph{Baire property}. 
	The Baire property is a topological analogue of being Lebesgue measurable; a subset of $\mathbb{R}^3$ is said to have the Baire property if its difference from some open set is topologically negligible.
	
        %Recently, this result has been strengthened by Marks and Unger \cite{marksunger}, who proved, under some mild topological assumption, that if there is a paradoxical decomposition, then there is one, where the pieces have the Baire property.

    Recently, results similar to the Banach-Tarski paradox that lie on the border of combinatorics, logic, group theory, and ergodic theory led to an emergence of a new field often called \textit{descriptive} or \textit{measurable combinatorics}. 
    The field focuses on the connection between the discrete and continuous and is largely concerned with the investigation of graph-theoretic concepts. 
    The usual setup in descriptive combinatorics is that we have a graph with uncountably many connected components, each being a countable graph of bounded degree. 
    For example, in case of the Banach-Tarski paradox, the vertices of the underlying graph are the points of the three balls, edges correspond to isometries, and the degree of each vertex is bounded by the number of chosen isometries. 
    Some of the most beautiful results related to the field include  \cite{laczk, marksunger,measurablesquare,doughertyforeman,marks2016baire,gaboriau,KST,DetMarks,millerreducibility,conley2020borel,csokagrabowski,Bernshteyn2021LLL}, see  \cite{kechris_marks2016descriptive_comb_survey,pikhurko2021descriptive_comb_survey} for recent surveys. 
    
    Importantly, in many of these results, including the Banach-Tarski paradox, graphs where each component is an infinite $\Delta$-regular tree appear naturally. 
    Oftentimes, questions considered in descriptive combinatorics lead to constructing a solution to a local problem in the underlying uncountable graph (in the case of Banach-Tarski, the local problem is perfect matching).  
    The construction needs to be such that the solution of the problem has some additional regularity properties. For example in the case of Banach-Tarski, a solution is possible when the regularity condition is the Baire property, but not if it is Lebesgue measurability.
    In fact, together with Borel measurability these are the most prominent regularity conditions studied in descriptive combinatorics.
    %These are the two main regularity conditions studied, the third most popular one being the Borel measurability\todo{true?}. 
    The corresponding complexity classes of local problems that always admit a solution with the respective regularity property are $\borel,\measure,\baire$ (See \cref{fig:big_picture_trees}). 
    In this paper, we moreover consider the setting where each connected component of the underlying graph is a $\Delta$-regular tree. 

	    The connection between distributed computing and descriptive combinatorics arises from the fact that in descriptive combinatorics we care about constructions that do not use the axiom of choice. 
	    In the distributed language, the axiom of choice corresponds to leader election, that is, the constructions in descriptive combinatorics do not allow picking exactly one point in every component.  
	    %The deterministic local model in descriptive combinatorics is partially captured by the complexity class $\borel$ (see also Remark \ref{rem:cont}). 
	    To get some intuition about the power of the complexity class $\borel$, we note that Borel constructions allow us to alternate countably many times the following two operations. 
	    First, any local algorithm with  constant local complexity can be run. 
	    Second, we have an oracle that provides a maximal independent set (MIS) on any graph that can be constructed locally from the information computed so far \cite{KST}. 
        Note that from the speedup result of \cite{chang_kopelowitz_pettie2019exp_separation} we get that every local problem with local complexity $O(\log^* n)$ can be solved by constant local constructions and \emph{one} call to such an MIS oracle. This implies the inclusion $\local(O(\log^* n)) \subseteq \borel$ in \cref{fig:big_picture_trees} proven in the insightful paper of Bernshteyn \cite{Bernshteyn2021LLL}.  
        The relationship of the class $\measure$ (and of the class $\fiid$ from the discussion of factors) to the class $\borel$ is analogous to the relationship of randomized distributed algorithms to deterministic distributed algorithms.

\paragraph{Local Problems on Regular Trees}

After introducing the three areas of interest in this work, we conclude the section by briefly discussing the kinds of problems we focus on, which are local problems on regular trees.
More precisely, we study \emph{locally checkable labeling (LCL)} problems, which are a class of problems, where the correctness of the solution can be checked locally.
%We also refer to them as just local problems. 
Examples include classical problems from combinatorics such as proper vertex coloring, proper edge coloring, perfect matching, and maximal independent set. % etc.
One main goal of this paper is to understand possible complexity classes of LCLs without inputs on infinite $\Delta$-regular trees and their finite analogues.
We refer the reader to \cref{sec:preliminaries} for a precise definition of a finite $\Delta$-regular tree.

The motivation for studying regular trees in this work stems from different sources: (a) infinite $\Delta$-regular trees are studied in the area of ergodic theory \cite{Bowen1,Bowen2}, random processes \cite{backhausz,backhausz2,lyons2011perfect} and descriptive combinatorics \cite{DetMarks,BrooksMeas}, (b) many lower bounds in distributed computing are proven in regular trees \cite{balliu_et_al:binary_lcl,balliu2019LB,brandt_etal2016LLL,brandt19automatic,Brandt20tightLBmatching,ChangHLPU20,goeoes14_PODC}, and (c) connecting and comparing the techniques of the three areas in this simple setting reveals already deep connections, see \cref{sec:contribution}.

\section{Our Contributions}
\label{sec:contribution}

We believe that our main contribution is presenting all three perspectives as part of a common theory. Our technical contribution is split into three main parts.

\subsection{Generalization of Marks' Technique}\label{sec:intromarks}
In Section \ref{sec:marks} we extend the Borel determinacy technique of Marks \cite{DetMarks}, which was used to prove the nonexistence of Borel $\Delta$-colorings and perfect matchings, to a  broader class of problems, and 
adapt the extended technique to the distributed setting, thereby obtaining a simple method for proving distributed lower bounds.
This method is the first lower bound technique for distributed computing using ideas coming from descriptive combinatorics (see \cite{BernshteynVizing} for a distributed computing upper bound motivated by descriptive combinatorics).
Moreover, we show how to use the developed techniques to obtain both $\borel$ and $\LOCAL$ lower bounds for local problems from a natural class, called homomorphism problems.
Our key technical ingredient for obtaining the mentioned techniques and results is a novel technique based on the notion of an \emph{ID graph}.
We note that a very similar concept to the ID graph was independently discovered by \cite{id_graph}. 

\paragraph{Marks' technique}
In the following we give an introduction to Marks' technique by going through a variant of his proof~\cite{DetMarks,Marks_Coloring} that shows that $\Delta$-coloring has deterministic local complexity $\Omega(\log n)$. The proof already works in the case where the considered regular tree comes with an input $\Delta$-edge coloring. In this case, the output color of a given vertex $u$ can be interpreted as that $u$ ``grabs'' the incident edge of that color. The problem is hence equivalent to the \emph{edge grabbing} problem where every vertex is required to grab an incident edge such that no two vertices grab the same edge. 

We first show the lower bound in the case that vertices do not have unique identifiers but instead are properly colored with $L > \Delta$ colors. 
Suppose there is an algorithm $\fA$ solving the edge grabbing problem with local complexity $t(n) = o(\log n)$, and consider a tree rooted at vertex $u$ of depth $t(n)$; such a tree has less than $n$ vertices, for large enough $n$.
Assume that $u$ has input color $\sigma \in [L]$, and consider, for some fixed edge color $\alpha$, the edge $e$ that is incident to $u$ and has color $\alpha$.
%Color vertex $u$ with any color $\sigma \in [L]$ and pick a special edge incident to $u$ of color $\alpha$.
Two players, Alice and Bob, are playing the following game. In the $i$-th round, Alice colors the vertices at distance $i$ from $u$ in the subtree reachable via edge $e$ with colors from $[L]$. Then, Bob colors all other vertices at distance $i$ from $u$ with colors from $[L]$ (see \cref{fig:game}).
Consider the output of $u$ when executing $\fA$ on the obtained colored tree.
Bob wins the game if $u$ grabs the edge $e$, and Alice wins otherwise. 

Note that either Alice or Bob has a winning strategy. 
Given the color $\sigma$ of $u$, if, for each edge color $\alpha$, Alice has a winning strategy in the game corresponding to the pair $(\sigma, \alpha)$, then we can create $\Delta$ copies of Alice and let them play their strategy on each subtree of $u$, telling them that the colors chosen by the other Alices are what Bob played. The result is a coloring of the input tree such that $u$, by definition, does not pick any edge, contradicting the fact that $\fA$ provides a valid solution!
So for every $\sigma$ there is at least one $\alpha$ such that Bob has a winning strategy for the game corresponding to $(\sigma, \alpha)$. By the pigeonhole principle, there are two colors $\sigma_1, \sigma_2$, such that Bob has a winning strategy for both pairs $(\sigma_1, \alpha)$ and $(\sigma_2, \alpha)$. 
But now we can imagine a tree rooted in an edge between vertices $u_1, u_2$ that are colored with colors $\sigma_1, \sigma_2$. We can now take two copies of Bob, one playing at $u_1$ and the other playing at $u_2$ and let them battle it out, telling each copy that the other color from $\{\sigma_1, \sigma_2\}$ and whatever the other copy plays are the moves of Alice. The resulting coloring has the property that both $u_1$ and $u_2$, when executing $\fA$ on the obtained colored tree, grab the edge between them, a contradiction that finishes the proof!

\paragraph{The ID graph}
The downside of the proof is that it does not work in the model with unique identifiers (where the players' moves consist in assigning identifiers instead of colors), since gluing copies of the same player could result in an identifier assignment where the identifiers are not unique. 
One possible remedy is to conduct the whole proof in the context of Borel graphs as was done by Marks. This proves an even stronger statement, namely that $\Delta$-coloring is not in the class $\borel$, but requires additional ad-hoc tricks and a pretty heavy set theoretic tool---Martin's celebrated Borel determinacy theorem~\cite{martin} stating that even for infinite two-player games one of the players has to have a winning strategy if the payoff set is Borel. 
The ID graph enables us to adapt the proof (and its generalization that we develop in \cref{sec:marks}) to the distributed setting, where the fact that one of the players has a winning strategy is obvious. 
Moreover, we use an infinite version of the ID graph to generalize Marks' technique also in the Borel setting. 

Here is how it works: 
The ID graph is a specific graph whose vertices are the possible unique input identifiers (for input graphs of size $n$), that is, numbers from $[n^{O(1)}]$. Its edges are colored with colors from $[\Delta]$ and its girth is $\Omega(\log n)$. When we define the game between Alice and Bob, we require them to label vertices with identifiers in such a way that whenever a new vertex is labeled with identifier $i$, and its already labeled neighbor has identifier $j$, then $ij$ is an edge in the ID graph.
Moreover, the color of edge $ij$ in the ID graph is required to be the same as the color of the edge between the vertices labeled $i$ and $j$ in the tree where the game is played. 
It is straightforward to check that these conditions enforce that even if we let several copies of the same player play, the resulting tree is labeled with unique identifiers. Hence, the same argument as above now finally proves that the deterministic local complexity of $\Delta$-coloring is  $\Omega (\log n)$. 

We note that our ID graph technique is of independent interest and may have applications in many different contexts. 
To give an example from distributed computing, consider the proof of the deterministic $\Omega(\log n)$-round lower bound for $\Delta$-coloring developed by the distributed community~\cite{brandt_etal2016LLL,chang_kopelowitz_pettie2019exp_separation}, which is based on the celebrated round elimination technique. 
Even though the result is deterministic, the proof is quite technical due to the fact that it relies on examining randomized algorithms, for reasons similar to the reasons why Marks' proof does not apply directly to the setting with unique identifiers. 
Fortunately, it can be again streamlined with the use of the ID graph technique. 
Moreover, in a follow-up work, the ID graph technique has already led to a new lower bound for the Lovász local lemma~\cite{brandt_grunau_rozhon2021LLL_in_LCA} in the area of Local Computation Algorithms (which is part of the realm of sequential computation), thereby giving further evidence for the versatility of the technique. 

%and to the distributed context using a new ID graph technique (which later found other uses in an LCA model lower bound, see \cite{brandt_grunau_rozhon2021LLL_in_LCA})\todo{cite the published version, i.e., PODC}. Albeit Marks' method exhibits a surprising degree of similarity to the round elimination technique developed in \cite{brandt19automatic}, our adaptation yields new lower bounds for problems that were resistant to the latter technique.\todo{add the question which is answered} We believe that our approach might also shed some new light on the long searched for  discretization of Borel determinacy (see, e.g., Gowers's Polymath 9 project). 

\paragraph{Marks vs. Round Elimination}
It is quite insightful to compare Marks' technique (and our generalization of it) with the powerful round elimination technique~\cite{brandt19automatic}, which has been responsible for all locality lower bounds of the last years~\cite{brandt_etal2016LLL, balliu2019LB, brandt19automatic, balliu2019hardness_homogeneous, balliu_et_al:binary_lcl, Brandt20tightLBmatching, balliu20rulingset, ChangHLPU20, balliu21dominatingset}. 
While, on the surface, Marks' approach developed for the Borel world may seem quite different from the round elimination technique, there are actually striking similarities between the two methods.
On a high level, in the round elimination technique, the following argument is used to prove lower bounds in the LOCAL model:
If a $T$-round algorithm exists for a problem $\Pi_0$ of interest, then there exists a $(T-1)$-round algorithm for some problem $\Pi_1$ that can be obtained from $\Pi_0$ in a mechanical manner.
By applying this step iteratively, we obtain a problem $\Pi_t$ that can be solved in $0$ rounds; by showing that there is no $0$-algorithm for $\Pi_t$ (which is easy to do if $\Pi_t$ is known), a $(T+1)$-round lower bound for $\Pi_0$ is obtained.

The interesting part regarding the relation to Marks' technique is how the ($T-i-1$)-round algorithms $\fA'$ are obtained from the ($T-i$)-round algorithms $\fA$ in the round elimination framework: in order to execute $\fA'$, each vertex $v$, being aware of its ($T-i-1$)-hop neighborhood, essentially asks whether, for all possible extensions of its view by one hop along a chosen incident edge, there exists some extension of its view by one hop along all other incident edges such that $\fA$, executed on the obtained ($T-i$)-hop neighborhood, returns a certain output at $v$, and then bases its output on the obtained answer.
It turns out that the vertex sets corresponding to these two extensions correspond precisely to two moves of the two players in the game(s) played in Marks' approach: more precisely, in round $T-i$ of a game corresponding to the considered vertex $v$ and the chosen incident edge, the move of Alice consists in labeling the vertices corresponding to the first extension, and the move of Bob consists in labeling the vertices corresponding to the second extension.

However, despite the similarities, the two techniques (at least in their current forms) have their own strengths and weaknesses and are interestingly different in that there are local problems that we know how to obtain good lower bounds for with one technique but not the other, and vice versa.
Finding provable connections between the two techniques is an exciting research direction that we expect to lead to a better understanding of the possibilities and limitations of both techniques.

In \cref{sec:marks} we use our generalized and adapted version of Marks' technique to prove new lower bounds for so-called homomorphism problems. 
Homomorphism problems are a class of local problems that generalizes coloring problems---each vertex is to be colored with some color and there are constraints on which colors are allowed to be adjacent. 
The constraints can be viewed as a graph---in the case of coloring this graph is a clique. 
In general, whenever the underlying graph of the homomorphism problem is $\Delta$-colorable, its deterministic local complexity is $\Omega(\log n)$, because solving the problem would imply that we can solve $\Delta$-coloring too (in the same runtime). 
It seems plausible that homomorphism problems of this kind are the only hard, i.e., $\Omega(\log n)$, homomorphism problems. However, our generalization of Marks' technique asserts that this is not true. 

\begin{theorem}
There are homomorphism problems whose deterministic local complexity on trees is $\Omega(\log n)$ such that the chromatic number of the underlying graph is $2\Delta-2$.
\end{theorem}

It is not known how to prove the same lower bounds using round elimination\footnote{Indeed, the descriptions of the problems have comparably large numbers of labels and do not behave like so-called ``fixed points'' (i.e., nicely) under round elimination, which suggests that it is hard to find a round elimination proof with the currently known approaches.}; in fact, as far as we know, these problems are the only known examples of problems on $\Delta$-regular trees for which a lower bound is known to hold but currently not achievable by round elimination. 
Proving the same lower bounds via round elimination is an exciting open problem.

%the two techniques exhibit the following provable connection: if the $0$-round algorithm obtained for the edge grabbing problem\footnote{We note that the lower bound for edge grabbing implies the lower bounds for the other mentioned problems.} in the round elimination approach selects an incident edge for some vertex $v$, then the game corresponding to $v$ and the selected incident edge admits a winning strategy for the second player.
%
%However,\todo{add stuff about the differences, i.e., that we can prove stuff with our extension of Marks' that we cannot do with round elimination (hype that :-)), and perhaps also that this is true the other way round. Maybe come back to the similarities at some point to argue that it is quite plausible that understanding the new approach (or/and its relation to round elimination) better leads to a better understanding of round elimination, e.g., for the task of finding good relaxations which is (one of) the main obstacle to get tons of lower bounds for all kinds of problems.}

%To our knowledge, this is the first time when nontrivial results were obtained in the finite world, using techniques inspired by Borel determinacy. Therefore, we believe that our approach might also shed some new light on the long searched for  discretization of Borel determinacy (see, e.g., Gowers's Polymath 9 project). 

%\todo{only time somebody gets something in local out of descriptive, but sheding light on gowers seems too far fetched}

\subsection{Separation of Various Complexity Classes}
\label{sec:introseparation}
\paragraph{Uniform Complexity Landscape}
We investigate the connection between randomized and uniform distributed local algorithms, where uniform algorithms are equivalent to the studied notion of finitary factors of iid. 
First, it is simple to observe that local problems with uniform complexity $t(\eps)$ have randomized complexity $t(1 / n^{O(1)})$ -- by definition, every vertex knows its local output after that many rounds with probability $1 - 1/n^{O(1)}$. The result thus follows by a union bound over the $n$ vertices of the input graph.

On the other hand, we observe that on $\Delta$-regular trees the implication also goes in the opposite direction in the following sense. Every problem that has a randomized complexity of $t(n) = o(\log n)$ has an uniform complexity of $O(t(1/\eps))$. 

One could naively assume that this equivalence also holds for higher complexities, but this is not the case. Consider for example the $3$-coloring problem.
It is well-known in the distributed community that $3$-coloring a tree can be solved deterministically in $O(\log n)$ rounds  using the rake-and-compress decomposition~\cite{chang_pettie2019time_hierarchy_trees_rand_speedup,MillerR89}.
On the other hand, there is no uniform algorithm for $3$-coloring a tree. 
If there were such an uniform algorithm, we could run it on any graph with large enough girth and color $99\%$ of its vertices with three colors. This in turn would imply that the high-girth graph has an independent set of size at least $0.99 \cdot n/3$. This is a contradiction with the fact that there exists high-girth graphs with a much smaller independence number~\cite{bollobas}. 

Interestingly, the characterization of Bernshteyn \cite{Bernshteyn_work_in_progress} implies that \emph{any} uniform distributed algorithm can be ``sped up'' to a deterministic local $O(\log n)$ complexity, as we prove in~\cref{thm:MainBaireLog}. 

We show that there are local problems that can be solved by an uniform algorithm but only with a complexity of $\Omega(\log 1/\eps)$. Namely, the problem of constructing a $2$-hop perfect matching on infinite $\Delta$-regular trees for $\Delta \ge 3$ has an uniform local complexity between $\Omega(\log 1/\eps)$ and $O(\poly\log 1/\eps)$. 
Formally, this proves the following theorem. 

\begin{theorem}
$\olocal(O(\log\log 1/\eps)) \subsetneq \olocal(O(\poly\log 1/\eps))$. 
\end{theorem}

The uniform  algorithm for this problem is based on a so-called one-ended forest decomposition introduced in \cite{BrooksMeas} in the descriptive combinatorics context. In a one-ended forest decomposition, each vertex selects exactly one of its neighbors as its parent by orienting the corresponding edge outwards. This defines a decomposition of the vertices into infinite trees. We refer to such a decomposition as a one-ended forest decomposition if the subtree rooted at each vertex only contains finitely many vertices. 
Having computed such a decomposition, $2$-hop perfect matching can be solved inductively starting from the leaf vertices of each tree.

We leave the understanding of the uniform complexity landscape in the regime $\Omega(\log 1/\eps)$ as an exciting open problem. In particular, does there exist a function $g(\eps)$ such that each local problem that can be solved by an uniform algorithm has an uniform complexity of $O(g(\eps))$?

\paragraph{Relationship of Distributed Classes with Descriptive Combinatorics}

Bernshteyn recently proved that $\local(\log^* n) \subseteq \borel$ \cite{Bernshteyn2021LLL}. That is, each local problem with a deterministic $\local$ complexity of $O(\log^* n)$ also admits a Borel-measurable solution. A natural question to ask is whether the converse also holds. 
Indeed, it is known that $\local(\log^* n) = \borel$ on paths with no additional input \cite{grebik_rozhon2021LCL_on_paths}. 
We show that on regular trees the situation is different. 
On one hand, a characterization of Bernshteyn \cite{Bernshteyn_work_in_progress} implies that $\borel \subseteq \baire \subseteq \LOCAL(O(\log n))$. 
On the other hand, we show that this result cannot be strengthened by proving the following result. 

\begin{theorem}
$\borel \subsetneq \rlocal(o(\log n))$.
\end{theorem}

That is, there exists a local problem that admits a Borel-measurable solution but cannot be solved with a (randomized) $\local$ algorithm running in a sublogarithmic number of rounds. 

Let us sketch a weaker separation, namely that $\borel \setminus \local(O(\log^* n)) \neq \emptyset$. Consider a version of $\Delta$-coloring where a subset of vertices can be left uncolored. However, the subgraph induced by the uncolored vertices needs to be a collection of doubly-infinite paths (in finite trees, this means each path needs to end in a leaf vertex). 
The nonexistence of a fast distributed algorithm for this problem essentially follows from the celebrated $\Omega(\log n)$ deterministic lower bound for $\Delta$-coloring of~\cite{brandt_etal2016LLL}. 
On the other hand, the problem allows a Borel solution. First, sequentially compute $\Delta-2$ maximal independent sets, each time coloring all vertices in the MIS with the same color, followed by removing all the
colored vertices from the graph. In that way, a total of $\Delta-2$ colors are used. Moreover, each
uncolored vertex has at most $2$ uncolored neighbors. 
This implies that the set of uncolored vertices forms a disjoint union of finite paths, one ended
infinite paths and doubly infinite paths. The first two classes can be colored inductively with two
additional colors, starting at one endpoint of each path in a Borel way (namely it can be done by making use of the countably many MISes in larger and larger powers of the input graph). 
Hence, in the end only doubly infinite paths are left uncolored, as desired.  

To show the stronger separation between the classes $\borel$ and $\rlocal(o(\log n))$ we use a variation of the
$2$-hop perfect matching problem. In this variation, some of the vertices can be left unmatched, but similar as in the variant of the $\Delta$-coloring problem described above, the graph induced by all the unmatched vertices needs to satisfy some additional constraints.

We conclude the paragraph by noting that the separation between the classes $\borel$ and $\local(O(\log^* n))$  is not as simple as it may look in the following sense. 
This is because problems typically studied in the $\local$ model with a $\local$ complexity of $\omega(\log^* n)$ like $\Delta$-coloring and perfect matching also do not admit a Borel-measurable solution due to the technique of Marks \cite{DetMarks} that we discussed in \cref{sec:intromarks}.

\subsection{$\local(O(\log n)) = \baire$} 

We already discussed that one of complexity classes studied in descriptive combinatorics is the class $\baire$.
Recently, Bernshteyn proved \cite{Bernshteyn_work_in_progress} that all local problems that are in the complexity class $\baire, \measure$ or $\fiid$ have to satisfy a simple combinatorial condition which we call being $\ell$-full. 
On the other hand, all $\ell$-full problems allow a $\baire$ solution \cite{Bernshteyn_work_in_progress}.
This implies a complete combinatorial characterization of the class $\baire$. 
We defer the formal definition of $\ell$-fullness to Section \ref{sec:baire} as it requires a formal definition of a local problem. 
Informally speaking, in the context of vertex labeling problems, a problem is $\ell$-full if we can choose a subset $S$ of the labels with the following property.
Whenever we label two endpoints of a path of at least $\ell$ vertices with two labels from $S$, we can extend the labeling with labels from $S$ to the whole path such that the overall labeling is valid.
For example, proper $3$-coloring  is $3$-full with $S=\{1,2,3\}$ because for any path of three vertices such that its both endpoints are colored arbitrarily, we can color the middle vertex so that the overall coloring is proper. 
On the other hand, proper $2$-coloring is not $\ell$-full for any $\ell$. 

We complement this result as follows. 
First, we prove that any $\ell$-full problem has local complexity $O(\log n)$, thus proving that all complexity classes considered in the areas of factors of iids and descriptive combinatorics from \cref{fig:big_picture_trees} are contained in $\local(O(\log n))$. 
In particular, this implies that the existence of \emph{any} uniform algorithm implies a local distributed algorithm for the same problem of local complexity $O(\log n)$. 
We obtain this result via the well-known rake-and-compress decomposition~\cite{MillerR89}.

On the other hand, we prove that any problem in the class $\local(O(\log n))$ satisfies the $\ell$-full condition. 
The proof combines a machinery  developed by Chang and Pettie~\cite{chang_pettie2019time_hierarchy_trees_rand_speedup}
%, which involves a pumping lemma for trees, 
with additional nontrivial ideas.
In this proof we construct recursively a sequence of sets of rooted, layered, and partially labeled trees, where the partial labeling is computed by simulating any given $O(\log n)$-round distributed algorithm, 
and then the set  $S$ meeting the $\ell$-full condition is constructed by considering all possible extensions of the partial labeling to complete correct labeling of these trees.
%\todo{does it make sense to add something from yi-jun's sketch?}

This result implies the following equality:
\begin{theorem}
$\local(O(\log n)) = \baire$. 
\end{theorem}

This equality is surprising in that the definitions of the two classes do not seem to have much in common on the first glance!
Moreover, the proof of the equality relies on nontrivial results in both distributed algorithms (the technique of Chang and Pettie~\cite{chang_pettie2019time_hierarchy_trees_rand_speedup}) and descriptive combinatorics (the fact that a hierarchical decomposition, so-called toast, can be constructed in $\baire$, \cite{conleymillerbound}, see Proposition \ref{pr:BaireToast}).

%The existence of \emph{gaps}~   in the  complexity landscape of  LCL problems  classify the distributed problems into natural complexity classes. 

The combinatorial characterization of the local complexity class $\local(O(\log n))$ on $\Delta$-regular trees is interesting from the perspective of distributed computing alone. 
This result can be seen as a part of a large research program aiming at classification of possible local complexities on various graph classes~\cite{balliu2020almost_global_problems,brandt_etal2016LLL,chang_kopelowitz_pettie2019exp_separation,chang_pettie2019time_hierarchy_trees_rand_speedup,chang2020n1k_speedups,Chang_paths_cycles,balliu2021_rooted_trees,balliu_et_al:binary_lcl,balliu2018new_classes-loglog*-log*}. 
%\gtodo{more citations?} 
That is, we wish not only to understand possible complexity classes (see the left part of \cref{fig:big_picture_trees} for possible local complexity classes on regular trees), but also to find combinatorial characterizations of problems in those classes that allow us to  efficiently decide for a given problem which class it belongs to.  
Unfortunately, even for grids with input labels, it is \emph{undecidable} whether a given local problem can be solved in $O(1)$ rounds~\cite{naorstockmeyer,brandt_grids}, since
local problems on grids can be used to simulate a Turing machine.
This undecidability result does not apply to paths and trees, hence for these graph classes it is still hopeful that we can find simple and useful characterizations for different classes of distributed problems. 

In particular, on paths it is decidable what classes a given local problem belongs to, for all classes coming from the three areas considered here, and this holds even if we allow inputs~\cite{Chang_paths_cycles,grebik_rozhon2021LCL_on_paths}. 
The situation becomes much more complicated when we consider trees. 
Recently, classification results on trees were obtained for so-called binary-labeling problems \cite{balliu_et_al:binary_lcl}. 
More recently, a complete classification was obtained in the case of \emph{rooted regular trees}~\cite{balliu2021_rooted_trees}. 
Although their algorithm takes exponential time in the worst case, the authors provided a practical implementation fast enough to classify many typical problems of interest. 

Much less is known for general, \emph{unoriented} trees, with an arbitrary number of labels. 
In general, deciding the optimal distributed complexity for a local problem on  bounded-degree trees is $\mathsf{EXPTIME}$-hard~\cite{chang2020n1k_speedups},  such a hardness result does not rule out the possibility for having a simple and polynomial-time characterization for the case of \emph{regular trees}, where there is no input and the constraints are placed only on degree-$\Delta$ vertices. 
Indeed, it was stated in~\cite{balliu2021_rooted_trees} as an open question to find such a characterization. Our characterization of $\local(O(\log n)) = \baire$ by $\ell$-full problems makes progress in better understanding the distributed complexity classes on trees and towards answering this open question.

\paragraph{Roadmap}
\hypertarget{par:roadmap}{}
In \cref{sec:preliminaries}, we define formally all the three setups we consider in the paper. 
In \cref{sec:marks} we discuss the lower bound technique of Marks and the new concept of an ID graph.
Next, in \cref{sec:separating_examples} we prove some basic results about the uniform complexity classes and give examples of problems separating some classes from \cref{fig:big_picture_trees}. 
Finally, in \cref{sec:baire} we prove that a problem admits a Baire measurable solution if and only if it admits a distributed algorithm of local complexity $O(\log n)$.

The individual sections can be read largely independently of each other.
Moreover, most of our results that are specific to only one of the three areas can be understood without reading the parts of the paper that concern the other areas.
We encourage the reader interested mainly in one of the areas to skip the respective parts.

%Similarly, we hope that a nonexpert in some of the areas can skip some proofs and still understand what is happening.
%for instance, a reader only interested in our results in descriptive combinatorics may proceed to \cref{subsec:descriptivecombinatorics} and then to\todo{4, 4.1, and 4.3 sounds weird as 4 contains the other two} \cref{sec:marks,subsec:homlcls,subsec:local_to_borel}. 
%\todo{write down the statements/parts that can be understood separately by local people, etc. The current way looks weird}

\section{Preliminaries}
\label{sec:preliminaries}

In this section, we explain the setup we work with, the main definitions and results. 
The class of graphs that we consider in this work are either infinite $\Delta$-regular trees, or their finite analogue that we define formally in \cref{subsec:local_problems}.
%, but, intuitively, imagine a tree of maximum degree $\Delta$ such that each vertex of degree $d$ has additional virtual $\Delta-d$ edges that do not have other endpoint. 
%If our problem is e.g. edge coloring, these virtual edges need to be labelled, too. 

We sometimes explicitly assume $\Delta > 2$. 
The case $\Delta = 2$, that is, studying paths, behaves differently and seems much easier to understand \cite{grebik_rozhon2021LCL_on_paths}. 
Unless stated otherwise, we do not consider any additional structure on the graphs, but sometimes it is natural to work with trees with an input $\Delta$-edge-coloring. 
%Another view of such trees is that they are the Cayley graphs of the free product of $\Delta$ $Z_2$ groups (think of finite strings over alphabet $[\Delta]$, together with a rule that two identical letters cannot be next to each other -- these strings represent the vertices of the underlying tree). \todo{hm not needed?} 

\subsection{Local Problems on $\Delta$-regular trees}
\label{subsec:local_problems}

The problems we study in this work are locally checkable labeling (LCL) problems, which, roughly speaking, are problems that can be described via local constraints that have to be satisfied in a suitable neighborhood of each vertex.
In the context of distributed algorithms, these problems were introduced in the seminal work by Naor and Stockmeyer~\cite{naorstockmeyer}, and have been studied extensively since.
In the modern formulation introduced in~\cite{brandt19automatic}, instead of labeling vertices or edges, LCL problems are described by labeling half-edges, i.e., pairs of a vertex and an incident edge.
This formulation is very general in that it not only captures vertex and edge labeling problems, but also others such as orientation problems, or combinations of all of these types.
Before we can provide this general definition of an LCL, we need to introduce some definitions.
We start by formalizing the notion of a half-edge.

\begin{definition}[Half-edge]
\label{def:half-edge}
    A \emph{half-edge} is a pair $(v, e)$ where $v$ is a vertex, and $e$ an edge incident to $v$.
    We say that a half-edge $(v, e)$ is \emph{incident} to a vertex $w$ if $w = v$, we say that a vertex $w$ is \emph{contained} in a half edge $(v,e)$ if $w=v$, and we say that $(v, e)$ \emph{belongs} to an edge $e'$ if $e' = e$.
    We denote the set of all half-edges of a graph $G$ by $H(G)$.
    A \emph{half-edge labeling} is a function $c \colon H(G) \to \Sigma$ that assigns to each half-edge an element from some label set $\Sigma$.
\end{definition}

%In order to be able to study LCL problems in all three areas of interest, we will essentially need to restrict attention to $\Delta$-regular trees, for some constant $\Delta$.
In order to speak about finite $\Delta$-regular trees, we need to consider slightly modified definition of a graph.
We think of each vertex to be contained in $\Delta$-many half-edges, however, not every half edge belongs to an actual edge of the graph.
That is half-edges are pairs $(v,e)$, but $e$ is formally not a pair of vertices.
Sometimes we refer to these half-edges as \emph{virtual} half-edges.
We include a formal definition to avoid confusions. See also \cref{fig:Delta_reg_tree}. 

%While this restriction can be formalized in different equivalent ways, we will mainly consider the following version of $\Delta$-regular trees.

\begin{definition}[$\Delta$-regular trees]
    A tree $T$, finite or infinite is a \emph{$\Delta$-regular tree} if either it is infinite and $T=T_\Delta$, where $T_\Delta$ is the unique infinite $\Delta$-regular tree, that is each vertex has exactly $\Delta$-many neighbors, or it is finite of maximum degree $\Delta$ and each vertex $v\in T$ of degree $d\le \Delta$ is contained in $(\Delta-d)$-many virtual half-edges.
    %These half-edges have only one endpoint, namely the respective vertex of degree $1$.
    
    Formally, we can view $T$ as a triplet $(V(T), E(T), H(T))$, where $(V(T),E(T))$ is a tree of maximum degree $\Delta$ and $H(T)$ consists of real half-edges, that is pairs $(v,e)$, where $v\in V(T)$, $e\in E(T)$ and $e$ is incident to $v$, together with some virtual edges, in the case when $T$ is finite, such that each vertex $v\in V(T)$ is contained in exactly $\Delta$-many half-edges (real or virtual).
    %The usual definitions for graphs and half-edges carry over to this definition in the natural way.
    %For simplicity, we may omit the explicit mentioning of half-edges and write $T = (V(T), E(T))$.
\end{definition}

\begin{figure}
    \centering
    \includegraphics{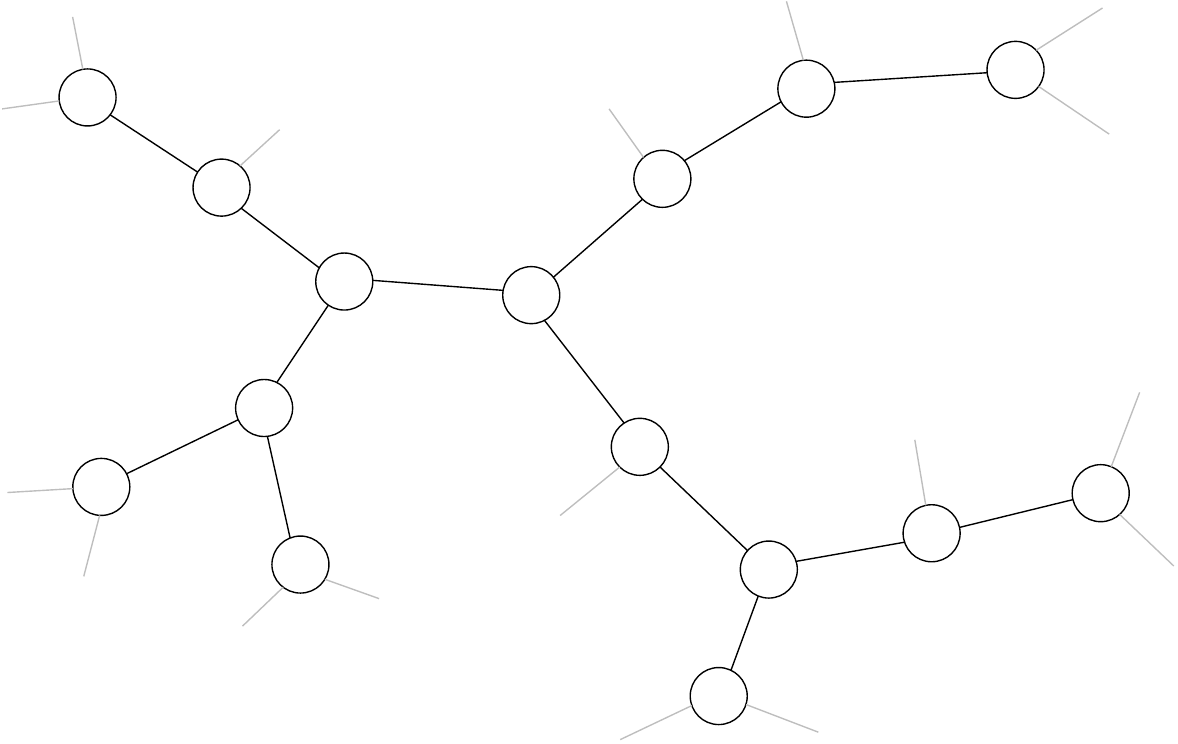}
    \caption{A $3$-regular tree on 15 vertices. Every vertex has $3$ half-edges but not all half-edges lead to a different vertex. }
    \label{fig:Delta_reg_tree}
\end{figure}

As we are considering trees in this work, each LCL problem can be described in a specific form that provides two lists, one describing all label combinations that are allowed on the half-edges incident to a vertex, and the other describing all label combinations that are allowed on the two half-edges belonging to an edge.\footnote{Every problem that can be described in the form given by Naor and Stockmeyer~\cite{naorstockmeyer} can be equivalently described as an LCL problem in this list form, by simply requiring each output label on some half-edge $h$ to encode all output labels in a suitably large (constant) neighborhood of $h$ in the form given in~\cite{naorstockmeyer}.}
We arrive at the following definition for LCLs on $\Delta$-regular trees.\footnote{Note that the defined LCL problems do not allow the correctness of the output to depend on input labels.}

\begin{definition}[LCLs on $\Delta$-regular trees]\label{def:LCL_trees}
A \emph{locally checkable labeling problem}, or \emph{LCL} for short, is a triple $\Pi=(\Sigma,\fV,\fE)$, where $\Sigma$ is a finite set of labels, $\fV$ is a subset of unordered cardinality-$\Delta$ multisets\footnote{Recall that a multiset is a modification of the concept of sets, where repetition is allowed.} of labels from $\Sigma$ , and $\fE$ is a subset of unordered cardinality-$2$ multisets of labels from $\Sigma$.

We call $\fV$ and $\fE$ the \emph{vertex constraint} and \emph{edge constraint} of $\Pi$, respectively.
Moreover, we call each multiset contained in $\fV$ a \emph{vertex configuration} of $\Pi$, and each multiset contained in $\fE$ an \emph{edge configuration} of $\Pi$.

Let $T$ be a $\Delta$-regular tree and $c:H(T)\to \Sigma$ a half-edge labeling of $T$ with labels from $\Sigma$.
We say that $c$ is a \emph{$\Pi$-coloring}, or, equivalently, a correct solution for $\Pi$, if, for each vertex $v$ of $T$, the multiset of labels assigned to the half-edges incident to $v$ is contained in $\fV$, and, for each edge $e$ of $T$, the cardinality-$2$ multiset of labels assigned to the half-edges belonging to $e$ is an element of $\fE$.
\end{definition}

An equivalent way to define our setting would be to consider $\Delta$-regular trees as commonly defined, that is, there are vertices of degree $\Delta$ and vertices of degree $1$, i.e., leaves.
In the corresponding definition of LCL one would consider leaves as unconstrained w.r.t. the vertex constraint, i.e., in the above definition of a correct solution the condition ``for each vertex  $v$'' is replaced by ``for each non-leaf vertex $v$''.
Equivalently, we could allow arbitrary trees of maximum degree $\Delta$ as input graphs, but, for vertices of degree $< \Delta$, we require the multiset of labels assigned to the half-edges to be extendable to some cardinality-$\Delta$ multiset in $\fV$.
When it helps the exposition of our ideas and is clear from the context, we may make use of these different but equivalent perspectives.

We illustrate the difference between our setting and ``standard setting'' without virtual half-edges on the perfect matching problem.
A standard definition of the perfect matching problem is that some edges are picked in such a way that each vertex is covered by exactly one edge.
It is easy to see that there is no local algorithm to solve this problem on the class of finite trees (without virtual half-edges), this is a simple parity argument.
However, in our setting, every vertex needs to pick exactly one half-edge (real or virtual) in such a way that both endpoints of each edge are either picked or not picked.
We remark that in our setting it is not difficult to see that (if $\Delta>2$), then this problem can be solved by a local deterministic algorithm of local complexity $O(\log(n))$.

\subsection{The $\local$ model}

In this section, we define local algorithms and local complexity.
We discuss the general relation between the classical local complexity and the uniform local complexity. Recall that when we talk about distributed algorithm on $\Delta$-regular trees, the algorithm has access to $n$, the size of the tree. The measure of complexity is the classical local complexity. On the other hand, when we talk about uniform distributed algorithms on $\Delta$-regular trees, we talk about an infinite $\Delta$-regular tree and the measure of complexity is the uniform local complexity. 
We start with the classical notions.
Recall that $B(v,t)$ is the $t$-hop neighborhood of a vertex $v$ in an underlying graph $G$.

\begin{definition}[Local algorithm]
\label{def:local_algorithm}
A \emph{distributed local algorithm} $\fA$ of local complexity $t(n)$ is a function defined on all possible $t(n)$-hop neighborhoods of a vertex. Applying an algorithm $\fA$ on an input graph $G$ means that the function is applied to a $t(n)$-hop neighborhood $B(u, t(n))$ of each vertex $u$ of $G$. The output of the function is a labeling of the half-edges around the given vertex. The algorithm also takes as input the size of the input graph $n$. 
\end{definition}

\begin{definition}[Local complexity]
\label{def:local_complexity}
We say that an LCL problem $\Pi$ has a deterministic local complexity $t(n)$ if there is a local algorithm $\fA$ of local complexity $t(n)$ such that when run on the input graph $G$, with each of its vertices having a unique identifier from $[n^{O(1)}]$, $\fA$ always returns a valid solution to $\Pi$.  
We also say $\Pi \in \dlocal( O(t(n))$. 

The randomized complexity is define analogously, but instead of unique identifiers, each vertex of $G$ has an infinite random string. The solution of $\fA$ needs to be a valid solution with probability $1 - 1/n$. 
We also say $\Pi \in \rlocal( O(t(n))$. 
\end{definition}

Whenever we talk about \emph{local} complexity on $\Delta$-regular trees, we always tacitly think about the class of finite $\Delta$-regular trees. We use the notation $\local(O(t(n)))$ whenever the deterministic and the randomized complexity of the problems are the same up to constant factor (cf. \cref{thm:basicLOCAL}). 

\paragraph{Classification of Local Problems on $\Delta$-regular Trees}

There is a lot of work aiming to classify all possible local complexities of LCL problems on bounded degree trees. This classification was recently finished.
Even though the results in the literature are stated for the classical notion of finite trees, we note that it is easy to check that the same picture emerges if we restrict ourselves to $\Delta$-regular trees according to our definition.

\begin{theorem}[Classification of local complexities of LCLs on $\Delta$-regular trees \cite{naorstockmeyer,chang_kopelowitz_pettie2019exp_separation,chang_pettie2019time_hierarchy_trees_rand_speedup,chang2020n1k_speedups,balliu2020almost_global_problems,balliu2019hardness_homogeneous,brandt_grunau_rozhon2021classification_of_lcls_trees_and_grids}]
\label{thm:basicLOCAL}
Let $\Pi$ be an LCL problem. Then the deterministic/randomized local complexity of $\Pi$ on $\Delta$-regular trees is one of the following:
\begin{enumerate}
    \item $O(1)$,
    \item $\Theta(\log^*n)$,
    \item $\Theta(\log n)$ deterministic and $\Theta(\log\log n)$ randomized,
    \item $\Theta(\log n)$,
    \item $\Theta(n^{1/k})$ for $k \in \mathbb{N}$.
\end{enumerate}
\end{theorem}

We do not know whether the complexity classes $\Theta(n^{1/k})$ are non-empty on \emph{$\Delta$-regular} trees. The $2 \frac{1}{2}$ coloring problems $\mathcal{P}_k$  are known to have complexity $\Theta(n^{1/k})$ on \emph{bounded-degree} trees~\cite{chang_pettie2019time_hierarchy_trees_rand_speedup}. However, as the correctness criterion of $\mathcal{P}_k$ depends on the degree of vertices, it does not fit into the class of LCL problems that we consider in \cref{def:LCL_trees}. 
Nevertheless, as can be seen in \cref{fig:big_picture_trees}, all classes of factors of iid and from descriptive combinatorics are already contained in the class $\local(O(\log n))$.

\paragraph{Uniform Algorithms}
As we mentioned before, when talking about local complexities, we always have in mind that the underlying graph is finite.
In particular, the corresponding algorithm knows the size of the graph.
On infinite $\Delta$-regular trees, or infinite graphs in general, we use the following notion \cite{HolroydSchrammWilson2017FinitaryColoring,Korman_Sereni_Viennot2012Pruning_algorithms_+_oblivious_coloring}. 

\begin{definition}
An \emph{uniform local algorithm} $\fA$ is a function that is defined on all possible (finite) neighborhoods of a vertex.
For some neighborhoods it outputs a special symbol $\emptyset$ instead of a labeling of the half-edges around the central vertex.
Applying $\fA$ on a graph $G$ means that for each vertex $u$ of $G$ the function is applied to $B(u,t)$, where $t$ is the minimal number such that $\fA(u,t)\not=\emptyset$.
We call $t$ the coding radius of $\fA$, and denote it, as a function on vertices, as $R_\fA$.
\end{definition}

We define the corresponding notion of \emph{uniform local complexity} for infinite $\Delta$-regular trees where each vertex is assigned an infinite random string.

\begin{definition}[Uniform local complexity \cite{HolroydSchrammWilson2017FinitaryColoring}]
\label{def:uniform_complexity}
We say that the \emph{uniform local (randomized) complexity} of an LCL problem $\Pi$ is $t(\eps)$ if there is an uniform local algorithm $\fA$ such that the following hold on the infinite $\Delta$-regular tree.
Recall that $R_\fA$ is the random variable measuring the coding radius of a vertex $u$, that is, the distance $\fA$ needs to look at to decide the answer for $u$. Then, for any $0 < \eps < 1$:
\[
\P(R_\fA \ge t(\eps)) \leq \eps.
\]
We also say $\Pi \in \localo(O(t(\eps))$.
\end{definition}

We finish by stating the following lemma that bounds the uniform complexity of concatenation of two uniform algorithm (we need to be little bit more careful and cannot just add the complexites up). 

\begin{lemma}[Sequential Composition]
\label{lem:sequential_composition}
Let $A_1$ and $A_2$ be two distributed uniform algorithms with an uniform local complexity of $t_1(\eps)$ and $t_2(\eps)$, respectively. 
Let $A$ be the sequential composition of $A_1$ and $A_2$. That is, $A$ needs to know the output that $A_2$ computes when the local input at every vertex is equal to the output of the algorithm $A_1$ at that vertex. Then, $t_A(\eps) \leq t_{A_1} \left( \frac{\eps/2}{\Delta^{t_{A_2}(\eps/2) + 1}}\right) + t_{A_2}(\eps/2)$.
\end{lemma}
\begin{proof}
Consider some arbitrary vertex $u$. Let $E_1$ denote the event that the coding radius of $\mathcal{A}_1$ is at most $t_{A_1} \left( \frac{\eps/2}{\Delta^{t_{A_2}(\eps/2) + 1}}\right)$ for all vertices in the $t_{A_2}(\eps/2)$-hop neighborhood around $u$. As the $t_{A_2}(\eps/2)$-hop neighborhood around $u$ contains at most $\Delta^{t_{A_2}(\eps/2) + 1}$ vertices, a union bound implies that $P(E_1) \geq 1 - \eps/2$. Moreover, let $E_2$ denote the event that the coding radius of algorithm $\mathcal{A}_2$ at vertex $u$ is at most $t_{A_2}(\eps/2)$. By definition, $P(E_2) \geq 1 - \eps/2$. Moreover, if both events $E_1$ and $E_2$ occur, which happens with  probability at least $1 - \eps$, then the coding radius of algorithm $\mathcal{A}$ is at most $t_{A_1} \left( \frac{\eps/2}{\Delta^{t_{A_2}(\eps/2) + 1}}\right) + t_{A_2}(\eps/2)$, thus finishing the proof.
 \end{proof}

\subsection{Descriptive combinatorics}
\label{subsec:descriptivecombinatorics}

Before we define formally the descriptive combinatorics complexity classes, we give a high-level overview on their connection to distributing computing for the readers more familiar with the latter.

The complexity class that partially captures deterministic local complexity classes is called $\borel$ (see also Remark \ref{rem:cont}).
First note that by a result of Kechris, Solecki and Todor\v{c}evi\'c \cite{KST} the \emph{maximal independent set problem} (with any parameter $r\in \mathbb{N}$) is in this class for any bounded degree graph.\footnote{That is, it is possible to find a Borel maximal independent set, i.e., a maximal independent set which is, moreover, a Borel subset of the vertex set.}
In particular, this yields that $\borel$ contains the class $\local(O(\log^* n))$ by the characterization of \cite{chang_kopelowitz_pettie2019exp_separation}, see \cite{Bernshteyn2021LLL}. Moreover, as mentioned before, $\borel$ is closed under countably many iterations of the operations of finding maximal independent set (for some parameter that might grow) and of applying a constant local rule that takes into account what has been constructed previously.\footnote{It is in fact an open problem, whether this captures fully the class $\borel$. However, note that an affirmative answer to this question would yield that problems can be solved in an ``effective" manner in the Borel context, which is known not to be the case in unbounded degree graphs \cite{todorvcevic2021complexity}.}

To get a better grasp of what this means, consider for example the proper vertex $2$-coloring problem on half-lines. It is clear that no local algorithm can solve this problem.
However, as it is possible to determine the starting vertex after countably many iterations of the maximal independent set operation, we conclude that this problem is in the class $\borel$.
The idea that $\borel$ can compute some unbounded, global, information will be implicitly used in all the construction in \cref{sec:separating_examples} that separate $\borel$ from local classes.

The intuition behind the class $\measure$ is that it relates in the same way to the class $\borel$, as randomized local algorithms relate to deterministic ones.
In particular, the operations that are allowed in the class $\measure$ are the same as in the class $\borel$ but the solution of a given LCL can be incorrect on a measure zero set.

The class $\baire$ can be considered as a topological equivalent of the measure theoretic class $\measure$, that is, a solution can be incorrect on a topologically negligible set.
The main difference between the classes $\measure$ and $\baire$ is that in the later there is a hierarchical decomposition that is called \emph{toast}.
(Note that this phenomenon is present in the case of $\measure$ exactly on so-called amenable graphs. It is also tightly connected with the notion of hyperfiniteness \cite{conleymillerbound,gao2015forcing}.)
The independence of colorings on a tree together with this structure allows for a combinatorial characterization of the class $\baire$, which was proven by Bernshteyn \cite{Bernshteyn_work_in_progress}, see also \cref{sec:baire}.

Next we formulate the precise definitions.
We refer the reader to \cite{pikhurko2021descriptive_comb_survey,kechris_marks2016descriptive_comb_survey,Bernshteyn2021LLL,kechrisclassical}, or to \cite[Introduction,~Section~4.1]{grebik_rozhon2021LCL_on_paths} and \cite[Section~7.1,~7.2]{grebik_rozhon2021toasts_and_tails} for intuitive examples and standard facts of descriptive set theory.
In particular, we do not define here the notion standard Borel/probability space, a Polish topology, a Borel probability measure, Baire property etc.

Let $\fG$ be a Borel graph of bounded maximum degree on a standard Borel space $X$.
In this paper we consider exclusively acyclic $\Delta$-regular Borel graphs and we refer to them as \emph{$\Delta$-regular Borel forests}.
It is easy to see that the set of half-edges (see Definition \ref{def:half-edge}) is naturally a standard Borel space, we denote this set by $H(\fG)$.
Thus, it makes sense to consider Borel labelings of $H(\fG)$.
Moreover, if $\mathcal{G}$ is a $\Delta$-regular Borel forest and $\Pi=(\Sigma,\fV,\mathcal{E})$ is an LCL, we can also decide whether a coloring $f:\mathcal{H}(\fG)\to \Sigma$ is a solution to $\Pi$ as in Definition \ref{def:LCL_trees}. 
Similarly, we say that the coloring $f$ solves $\Pi$, e.g., on a $\mu$-conull set for some Borel probability measure $\mu$ on $X$ if there is a Borel set $C\subseteq X$ such that $\mu(C)=1$, the vertex constraints are satisfied around every $x\in C$ and the edge constraints are satisfied for every $x,y\in C$ that form an edge in $\fG$.

\begin{definition}[Descriptive classes]
\label{def:descriptive}
Let $\Pi=(\Sigma,\fV,\fE)$ be an LCL.
We say that $\Pi$ is in the class $\borel$ if for every acyclic $\Delta$-regular Borel graph $\fG$ on a standard Borel space $X$, there is a Borel function $f:H(\fG)\to \Sigma$ that is a $\Pi$-coloring of $\fG$.

We say that $\Pi$ is in the class $\baire$ if for every acyclic $\Delta$-regular Borel graph $\fG$ on a standard Borel space $X$ and every compatible Polish topology $\tau$ on $X$, there is a Borel function $f:H(\fG)\to \Sigma$ that is a $\Pi$-coloring of $\fG$ on a $\tau$-comeager set.

We say that $\Pi$ is in the class $\measure$ if for every acyclic $\Delta$-regular Borel graph $\fG$ on a standard Borel space $X$ and every Borel probability measure $\mu$ on $X$, there is a Borel function $f:H(\fG)\to \Sigma$ that is a $\Pi$-coloring of $\fG$ on a $\mu$-conull set.
\end{definition}

The following claim follows directly from the definition.

\begin{claim}\label{cl:basicDC}
We have $\borel\subseteq \measure, \baire$.
%In fact, this inclusion holds for any graph.
\end{claim}

Recently Bernshteyn \cite{Bernshteyn2021LLL,Bernshteyn2021local=cont} proved several results that connect distributed computing with descriptive combinatorics.
Using the language of complexity classes we can formulate some of the result as inclusions in \cref{fig:big_picture_trees}.

\begin{theorem}[\cite{Bernshteyn2021LLL}]
We have
\begin{itemize}
    \item $\local(O(\log^*(n)))\subseteq \borel$,
    \item $\rlocal(o(\log(n))) \subseteq \measure, \baire$.
\end{itemize}
In fact, these inclusions hold for any reasonable class of bounded degree graphs.
\end{theorem}

\begin{remark}
\label{rem:cont}
The ``truly'' local class in descriptive combinatorics is the class $\CONT$.
Even though we do not define this class here\footnote{To define the class $\continuous$, rather than asking for a continuous solution on all possible $\Delta$-regular Borel graphs, one has to restrict to a smaller family of graphs, otherwise the class trivializes. To define precisely this family is somewhat inconvenient, and not necessary for our arguments.}, and we refer the reader to \cite{Bernshteyn2021local=cont,GJKS,grebik_rozhon2021toasts_and_tails} for the definition and discussions in various cases, we mention that the inclusion \ \[\local(O(\log^* n))\subseteq \CONT \tag{*}\] holds in most reasonable classes of bounded degree graphs, see \cite{Bernshteyn2021LLL}.
This also applies to our situation
Recently, it was shown by Bernshteyn \cite{Bernshteyn2021local=cont}, and independently by Seward \cite{Seward_personal}, that $(*)$ can be reversed for Cayley graphs of finitely generated groups.
This includes, e.g., natural $\Delta$-regular trees with proper edge $\Delta$-coloring, as this is a Cayely graph of free product of $\Delta$-many $\mathbb{Z}_2$ induced by the standard generating set.
It is, however, not clear whether it $(*)$ can be reversed in our situation, i.e., $\Delta$-regular trees without any additional labels.

%This partially applies also to our situation, since most ``natural" $\Delta$-regular forests fall under this rule (e.g., when the function assigning to a vertex the isomorphism type of its $R$-neighbourhood, for each $R$). For example our lower bounds in Section \ref{sec:marks} still hold for the more restricted class.

%In our situation these result do not seem to apply in a straightforward way, that is, we do not know if $\CONT=\local(O(\log^* n))$ for $\Delta$-regular trees without additional structure.
\end{remark}

\subsection{Random processes}

We start with an intutitve description of fiid processes. Let $T_\Delta$ be the infinite $\Delta$-regular tree.
Informally, \emph{factors of iid processes (fiid)} on $T_\Delta$ are infinite analogues of local randomized algorithms in the following way.
Let $\Pi$ be an LCL and $u\in T_\Delta$.
In order to solve $\Pi$, we are allowed to explore the whole graph, and the random strings that are assigned to vertices, and then output a labeling of half-edges around $u$.
If such an assignment is a measurable function and produces a $\Pi$-coloring almost surely, then we say that $\Pi$ is in the class $\fiid$.
Moreover, if every vertex needs to explore only finite neighborhood to output a solution, then we say that $\Pi$ is in the class $\ffiid$.
Such processes are called \emph{finitary} fiid (\emph{ffiid}). There is also an intimate connection between ffiid and uniform algorithms.
This is explained in \cite[Section 2.2]{grebik_rozhon2021toasts_and_tails}.
Informally, an ffiid process that almost surely solves $\Pi$ is, in the language of distributed computing, an uniform local algorithm that solves $\Pi$.
This allows us to talk about uniform local complexity of an ffiid.
In the rest of the paper we interchange both notions freely with slight preference for the distributed computing language.

Now we define formally these classes, using the language of probability.
We denote by $\Aut(T_\Delta)$ the automorphism group of $T_\Delta$.
An \emph{iid process} on $T_\Delta$ is a collection of iid random variables $Y=\{Y_v\}_{v\in V(T_\Delta)}$ indexed by vertices, or edges, half-edges etc, of $T_\Delta$ such that their joint distribution is invariant under $\Aut(T_\Delta)$.
We say that $X$ is a \emph{factor of iid process (fiid)} if $X=F(Y)$, where $F$ is a measurable $\Aut(T_\Delta)$-invariant map and $Y$ is an iid process on $T_\Delta$.\footnote{We always assume $Y=(2^\mathbb{N})^{T_\Delta}$ endowed with the product Lebesgue measure.}
Moreover, we say that $X$ is a \emph{finitary factor if iid process (ffiid)} if $F$ depends with probability $1$ on a finite (but random) radius around each vertex.
We denote as $R_F$ the random variable that assigns minimal such radius to a given vertex, and call it the \emph{coding radius} of $F$.
We denote the space of all $\Pi$-colorings of $T_\Delta$ as $X_\Pi$.
This is a subspace of $\Sigma^{H(T_\Delta)}$ that is invariant under $\Aut(T_\Delta)$.

\begin{definition}
We say that an LCL $\Pi$ is in the class $\fiid$ ($\ffiid$) if there is an fiid (ffiid) process $X$ that almost surely produces elements of $X_\Pi$.

Equivalently, we can define $\ffiid = \bigcup_f \localo(f(\eps))$ where $f$ ranges over all functions. That is, $\ffiid$ is the class of problems solvable by any uniform distributed algorithm. 

%Given a function $f:\mathbb{N}\to \mathbb{N}$, we say that $\Pi\in \tail(f)$ if there is an ffiid process $X$ that almost surely produces elements of $X_\Pi$ and satisfies
%$$\P(R_F>r)\le \frac{1}{f(r)}$$
%for every $r\in \mathbb{N}$ large enough.
%We say that $\Pi\in \tail$ if it is in $\tail(f)$ for some $f$.
\end{definition}

%We say that a ffiid $F$ is a \emph{block factor} if $R_F$ is dominated by some deterministic $R<+\infty$ almost surely.

It is obvious that $\ffiid\subseteq \fiid$.
The natural connection between descriptive combinatorics and random processes is formulated by the inclusion $\measure\subseteq \fiid$.
While this inclusion is trivially satisfied, e.g., in the case of $\Delta$-regular trees with proper edge $\Delta$-coloring, in our situation we need a routine argument that we include for the sake of completeness in \cref{app:ktulu}.

\begin{restatable}{lemma}{FIIDinMEASURE}\label{thm:FIIDinMEASURE}
Let $\Pi$ be an LCL such that $\Pi\in \measure$.
Then $\Pi\in \fiid$.
\end{restatable}

\subsection{Specific Local Problems}

Here we list some well-known local problems that were studied in all three considered areas. We describe the results known about these problems with respect to classes from \cref{fig:big_picture_trees}. 

\paragraph{Edge Grabbing}
We start by recalling the well-known problem of edge grabbing (a close variant of the problem is known as sinkless orientation\cite{brandt_etal2016LLL}). 
In this problem, every vertex should mark one of its half-edges (that is, grab an edge) in such a way that no edge can be marked by two vertices. 
%\begin{definition}[Edge Grabbing -- $\edgegrab$]
%Every vertex should mark one of its half-edges (that is, grab an edge). No edge can be marked by two vertices. 
%\end{definition}
It is known that $\edgegrab  \not\in \local(O(\log^* n))$ \cite{brandt_etal2016LLL} but $\edgegrab \in \rlocal(O(\log\log n))$.

Similarly, $\edgegrab \not \in \borel$ by \cite{DetMarks}, but $\edgegrab\in \measure$ by \cite{BrooksMeas}: to see the former, just note that if a $\Delta$-regular tree admits proper $\Delta$-colorings of both edges and vertices, every vertex can grab an edge of the same color as the color of the vertex. Thus, $\edgegrab\in \borel$ would yield that $\Delta$-regular Borel forests with Borel proper edge-colorings admit a Borel proper $\Delta$-coloring, contradicting \cite{DetMarks}.

This completes the complexity characterization of $\edgegrab$ as well as the proper vertex $\Delta$-coloring with respect to classes in \cref{fig:big_picture_trees}.

\paragraph{Perfect Matching}
Another notorious LCL problem, whose position in \cref{fig:big_picture_trees} is, however, not completely understood, is the perfect matching problem $\perfmatch$.
Recall that the perfect matching problem $\perfmatch$ asks for a matching that covers all vertices of the input tree.\footnote{In our formalism this means that around each vertex exactly one half-edge is picked.}
It is known that $\perfmatch \in \fiid$ \cite{lyons2011perfect}, and it is easy to see that $\perfmatch \not\in \rlocal(O(\log\log(n)))$ (we will prove a stronger result in \cref{pr:AddingLineLLL}).
Marks proved \cite{DetMarks} that it is not in $\borel$, even when the underlying tree admits a Borel bipartition.
It is not clear if $\perfmatch$ is in $\ffiid$, nor whether it is in $\measure$.

\paragraph{Graph Homomorphism}
We end our discussion with LCLs that correspond to graph homomorphisms (see also the discussion in Section \ref{sec:marks}).
These are also known as edge constraint LCLs.
Let $G$ be a finite graph.
Then we define $\Pi_G$ to be the LCL that asks for a homomorphism from the input tree to $G$, that is, vertices are labeled with vertices of $G$ in such a way that edge relations are preserved.
There are not many positive results except for the class $\baire$.
It follows from the result of Bernshteyn \cite{Bernshteyn_work_in_progress} (see \cref{sec:baire}) that $\Pi_G\in \baire$ if and only if $G$ is not bipartite.
The main negative results can be summarized as follows.
An easy consequence of the result of Marks \cite{DetMarks} is that if $\chi(G)\le \Delta$, then $\Pi_G\not \in \borel$.
In this paper, we describe examples of graphs of chromatic number up to $2\Delta-2$ such that the corresponding homomorphism problem is not in $\borel$, see \cref{sec:marks}.
The theory of \emph{entropy inequalities} see \cite{backhausz} implies that if $G$ is a cycle on more than $9$ vertices, then $\Pi_G\not\in \fiid$.

\section{Generalization of Marks' technique}
\label{sec:marks}

In this section, we first develop a new way of proving lower bounds in the $\local$ model based on a generalization of a technique of Marks \cite{DetMarks}. Then, we use ideas arising in the finitary setting---connected to the adaptation of Marks' technique---to obtain new results back in the Borel context.
For an introduction to Marks' technique and a high-level discussion about the challenges in adapting the technique to the standard distributed setting as well as our solution via the new notion of an ID graph, we refer the reader to \cref{sec:intromarks}.

The setting we consider in this section is $\Delta$-regular trees that come with a proper $\Delta$-edge coloring with colors from $[\Delta]$.
All lower bounds discussed already hold under these restrictions (and therefore also in any more general setting).

Recall that an LCL $\Pi=(\Sigma,\fV,\fE)$ is given by specifying a list of allowed vertex configurations and a list of allowed edge configurations (see Definition \ref{def:LCL_trees}).
To make our lower bounds as widely applicable as possible, we replace the latter list by a separate list for each of the $\Delta$ edge colors; in other words, we consider LCLs where the correctness of the output is allowed to depend on the input that is provided by the edge colors.
Hence, formally, in this section, an LCL is more generally defined: it is a tuple $\Pi=(\Sigma,\fV,\fE)$ where $\Sigma$ and $\fV$ are as before, while $\fE = (\fE_{\alpha})_{\alpha \in [\Delta]}$ is now a $\Delta$-tuple of sets $\fE_{\alpha}$ consisting of cardinality-$2$ multisets.
Similarly as before, a half-edge labeling (see Definition \ref{def:descriptive} for the Borel context) with labels from $\Sigma$ is a correct solution for $\Pi$ if, for each vertex $v$, the multiset of labels assigned to the half-edges incident to $v$ is contained in $\fV$, and, for each edge $e$, the multiset of labels assigned to the half-edges belonging to $e$ is contained in $\fE_{\alpha}$, where $\alpha$ is the color of the edge $e$.

The idea of our approach is to identify a condition that an LCL $\Pi$ necessarily has to satisfy if it is solvable in $O(\log^* n)$ rounds in the $\local$ model.
Showing that $\Pi$ does not satisfy this condition then immediately implies an $\Omega(\log n)$ deterministic and $\Omega(\log \log n)$ randomized lower bound, by Theorem~\ref{thm:basicLOCAL}.

In order to define our condition we need to introduce the notion of a configuration graph: a \emph{configuration graph} is a $\Delta$-tuple $\mathbb{P}=(\mathbb{P}_\alpha)_{\alpha\in \Delta}$ of graphs, where the vertex set of each of the graphs $\mathbb{P}_\alpha$ is the set of subsets of $\Sigma$, and there is an edge in $\mathbb{P}_\alpha$ connecting two vertices $S,T$ if and only if there are $\ta\in S$ and $\tb\in T$ such that $\{\ta,\tb\}\in {\fE}_\alpha$, note that loops are allowed.
(Naturally, we will consider any two vertices of different $\mathbb{P}_\alpha$ to be distinct, even if they correspond to the same subset of $\Sigma$.)

Now we are set to define the aforementioned condition. The intuition behind the playability condition is the following: assume that there exists a local algorithm $\mathcal{A}$ that solves $\Pi$ using the $t$ neighbourhood of a given vertex. We are going to define a family of two player games. The game will depend on some $S \subseteq \Sigma$. $\PI$ and $\PK$  will assign labels (or IDs) to the vertices in the $t$-neighbourhood (in some way specified later on, depending on $\alpha \in [\Delta]$). When the assignment is complete, we evaluate $\mathcal{A}$ on the root of the obtained labelled graph, this way obtaining an element $s$ of $\Sigma$, and decide who is the winner based on $s 
\in S$ or not. Naturally, it depends on $S$ and $\alpha$, which player has a winning strategy. This gives rise to a two coloring of vertices of $\mathbb{P}_\alpha$  by colors $\PI$ and $\PK$. The failure of the playability condition will guarantee that using a strategy stealing argument one can derive a contradiction.

\begin{definition}[Playability condition]\label{def:Playability}
We say that an LCL $\Pi$ is \emph{playable} if for every $\alpha\in [\Delta]$ there is a coloring $\Lambda_{\alpha}$ of the vertices of $\mathbb{P}_{\alpha}$ with two colors $\{\PI,\PK\}$ such that the following conditions are satisfied:
\begin{enumerate}
    \item [(A)] For any tuple $(S_{\alpha})_{\alpha \in [\Delta]} \in V(\mP_1) \times \dots \times V(\mP_{\Delta})$ satisfying $\Lambda_\alpha(S_\alpha)=\PI$ for each $\alpha \in [\Delta]$, there exist $\ta_\alpha\not\in S_\alpha$ such that $\{\ta_\alpha\}_{\alpha\in \Delta}\in {\fV}$, and
    \item [(B)] for any $\alpha \in [\Delta]$, and any tuple $(S,T) \in V(\mP_{\alpha}) \times V(\mP_{\alpha})$ satisfying $\Lambda_\alpha(S)=\Lambda_\alpha(T)=\PK$, we have that $(S,T)$ is an edge of $\mathbb{P}_{\alpha}$.
\end{enumerate}

\end{definition}

Our aim in this section is to show the following general results.

\begin{theorem}\label{th:MainLOCAL}
Let $\Pi$ be an LCL that is not playable.
Then $\Pi$ is not in the class $\local(O(\log^* n))$.
\end{theorem}

Using ideas from the proof of this result, we can formulate an analogous theorem in the Borel context.
Let us mention that while \cref{th:MainLOCAL} is a consequence of \cref{thm:MainBorel} by \cite{Bernshteyn2021LLL}, we prefer to state the theorems separately.
This is because the proof of the Borel version of the theorem uses heavily complex results such as the Borel determinacy theorem, theory of local-global limits and some set-theoretical considerations about measurable sets.
This is in a stark contrast with the proof of \cref{th:MainLOCAL} that uses `merely' the existence of large girth graphs and determinacy of finite games.
However, the ideas surrounding these concepts are the same in both cases.

\begin{theorem}\label{thm:MainBorel}
Let $\Pi$ be an LCL that is not playable.%\todo{now that the main result in this section is the non-BOREL one, where to put all the BOREL stuff (i.e., everything below before Section~\ref{subsec:homlcls}?}
Then $\Pi$ is not in the class $\borel$.
\end{theorem}

The main application of these abstract results is to find lower bounds for \emph{(graph) homomorphism LCLs} aka \emph{edge constraint LCLs}.
We already defined this class in \cref{subsec:local_problems} but we recall the definition in the formalism of this section to avoid any confusion.
Let $G$ be a finite graph.
Then $\Pi_G=(\Sigma,\fV,\mathcal{E})$ is defined by letting
\begin{enumerate}
\item \label{p:vertex} ${\fV}=\{({\ta},\dots,{\ta}):\ta\in \Sigma\}$,
\item $\Sigma=V(G)$,
\item $\forall \alpha \in [\Delta] \ (\mathcal{E}_\alpha=E(G)).$
\end{enumerate}
An LCL for which \eqref{p:vertex} holds is called a \emph{vertex-LCL}. There is a one-to-one correspondence between vertex-LCLs for which $\forall 
\alpha,\beta \ (\mathcal{E}_\alpha=\mathcal{E}_\beta)$ and LCLs of the form $\Pi_G$. (Indeed, to vertex-LCLs with this property one can associate a graph $G$ whose vertices are the labels in $\Sigma$ and where two vertices $\ell, \ell'$ are connected by an edge if and only if $\{ \ell, \ell' \} \in \fE$.)
%Such LCLs are called \emph{(graph) homomorphism LCLs}.
Note that if $\Pi$ is a vertex-LCL, then condition (A) in Definition~\ref{def:Playability} is equivalent to the statement that no tuple $\{S_\alpha\}_{\alpha\in \Delta}$ that satisfies the assumption of (A) can cover $\Sigma$, i.e., that $\Sigma \nsubseteq S_1 \cup \dots \cup S_{\Delta}$ for such a tuple.
Moreover, if $\Pi_G$ is a homomorphism LCL, then $\mathbb{P}_\alpha=\mathbb{P}_{\beta}$ for every $\alpha,\beta\in \Delta$.

The simplest example of a graph homomorphism LCL is given by setting $G$ to be the clique on $k$ vertices; the associated LCL $\Pi_G$ is simply the problem of finding a proper $k$-coloring of the vertices of the input graph.
Now we can easily derive Marks' original result from our general theorem.

\begin{corollary}[Marks \cite{DetMarks}]\label{cor:detmarks}
Let $G$ be a finite graph that has chromatic number at most $\Delta$. Then $\Pi_G$ is not playable. In particular, there is a $\Delta$-regular Borel forest that have Borel chromatic number $\Delta+1$. 
\end{corollary}
\begin{proof}
Let $A_1,\dots, A_\Delta$ be some independent sets that cover $G$, and $\Lambda_1, \dots, \Lambda_{\Delta}$ arbitrary colorings of the vertices of $\mP_1, \dots, \mP_{\Delta}$, respectively, with colors from $\{\PI,\PK\}$.
It follows that $\Lambda_\alpha(A_\alpha)=\PI$ for every $\alpha\in \Delta$, since otherwise condition (B) in Definition~\ref{def:Playability} is violated with $S=T=A_\alpha$.
But then condition (A) does not hold.
\end{proof}

As our main application we describe graphs with chromatic number larger than $\Delta$ such that $\Pi_G$ is not playable.
This rules out the hope that the complexity of $\Pi_G$ is connected with chromatic number being larger than $\Delta$.
%One could naively hope that the chromatic number of $G$ is the only obstacle to solving $\Pi_G$, that is, $\Pi_G$ admits a solution in $\local(O(\log^* n))$/$\borel$ if and only if $\chi(G)>\Delta$. This turns out not to be the case.
In Section \ref{subsec:homlcls} we show the following.
Note that the case $k=\Delta$ is \cref{cor:detmarks}.

\begin{theorem}
\label{th:homomorphism}
Let $\Delta > 2$ and $\Delta < k \leq 2\Delta-2$. There exists a graph $G_k$ with $\chi(G_k)= k$, such that $\Pi_{G_k}$ is not playable and $\Pi_{G_k} \in \rlocal(O(\log\log n ))$. In particular, $\Pi_{G_k}\not\in \borel$ and $\Pi_{G_k}\not\in\local(O(\log^*(n)))$.
\end{theorem}

Interestingly, recent results connected to counterexamples to Hedetniemi's conjecture yield the same statement asymptotically, as $\Delta \to \infty$ (see Remark \ref{r:hedet}).

\begin{remark}
It can be shown that for $\Delta=3$ both the Chv\' atal and Gr\" otsch graphs are suitable examples for $k=4$.
\end{remark}

\begin{remark}
As another application of his technique Marks showed in \cite{DetMarks} that there is a Borel $\Delta$-regular forest that does not admit Borel perfect matching.
This holds even when we assume that the forest is Borel bipartite, i.e., it has Borel chromatic number $2$.
In order to show this result Marks works with free joins of colored hyperedges, that is, Cayley graphs of free products of cyclic groups.
One should think of two types of triangles ($3$-hyperedges) that are joined in such a way that every vertex is contained in both types and there are no cycles.
We remark that the playability condition can be formulated in this setting.
Similarly, one can derive a version of \cref{thm:MainBorel}.
However, we do not have any application of this generalization.
\end{remark}

\subsection{Applications of playability to homomorphism LCLs}\label{subsec:homlcls}

In this section we find the graph examples from \cref{th:homomorphism}.
First we introduce a condition $\Delta$-(*) that is a weaker condition than having chromatic number at most $\Delta$, but still implies that the homomorphism LCL is not playable. Then, we will show that--similarly to the way the complete graph on $\Delta$-many vertices, $K_\Delta$, is maximal among graphs of chromatic number $\leq \Delta$--there exists a maximal graph (under homorphisms) with property $\Delta$-(*). Recall that we assume $\Delta >2$.

\begin{figure}
    \centering
    \includegraphics[width=.6\textwidth]{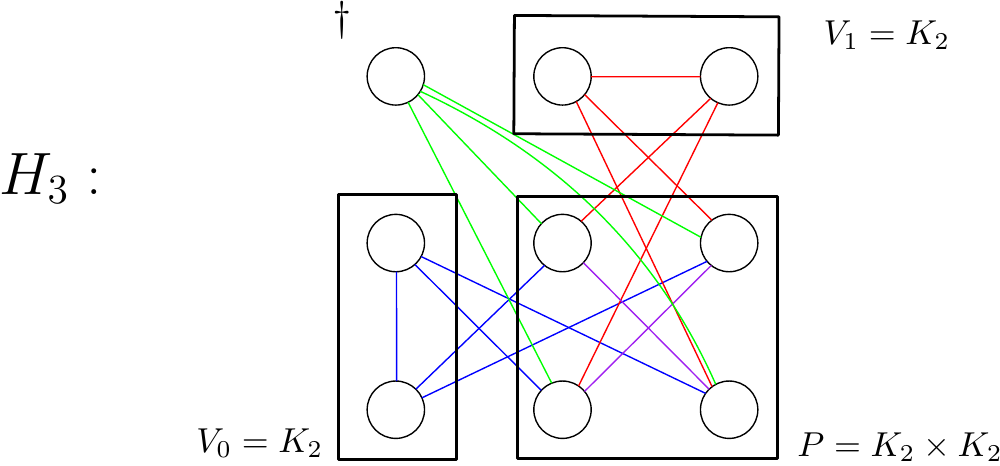}
    \caption{The maximal graph with the property $\Delta$-(*) for $\Delta=3$.}% $H_3$ from \cref{pr:MinimalExample}. }
    \label{fig:h3}
\end{figure}

\begin{definition}[Property $\Delta$-(*)]
Let $\Delta>2$ and $G=(V,E)$ be a finite graph.
We say that $G$ satisfies property $\Delta$-(*) if there are sets $S_0,S_1\subseteq V$ such that $G$ restricted to $V \setminus S_i$ has chromatic number at most $(\Delta-1)$ for $i\in \{0,1\}$, and there is no edge between $S_0$ and $S_1$.
\end{definition}

Note that $\chi(G) \leq \Delta$ implies $\Delta$-(*): indeed, if $A_1,\dots,A_\Delta$ are independent sets that cover $V(G)$, we can set $S_0=S_1=A_1$.

On the other hand, we claim that if $G$ satisfies $\Delta$-(*) then $\chi(G)\le 2\Delta-2$. In order to see this, take $S_0,S_1 \subseteq V(G)$ witnessing $\Delta$-(*). Then, as there is no edge between $S_0$ and $S_1$, so in particular, between $S_0 \setminus S_1$ and $S_0 \cap S_1$, it follows that the chromatic number of $G$'s restriction to $S_0$ is $\leq \Delta-1$. But then we can construct proper $\Delta-1$-colorings of $S_0$ and $V(G) \setminus S_0$, which shows our claim.

\begin{proposition}\label{pr:CONdition*}
Let $G$ be a graph satisfying $\Delta$-(*).
Then $\Pi_G$ is not playable.
\end{proposition}
\begin{proof}
Fix $S_0,S_1$ as in the definition of $\Delta$-(*) and assume for a contradiction that colorings $\Lambda_1, \dots, \Lambda_\Delta$ as described in Definition~\ref{def:Playability} exist.
By Property $\Delta$-(*), there exist independent sets $A_1, \dots, A_{\Delta-1}$ such that $S_0$ together with the $A_i$ covers $G$, i.e., such that $S_0 \cup A_1 \cup \dots \cup A_{\Delta-1} = V(G)$.
For each $\alpha \in [\Delta-1]$, we must have $\Lambda_\alpha(A_{\alpha})=\PI$, since otherwise condition (B) in Definition~\ref{def:Playability} is violated.
Consequently $\Lambda_{\Delta}(S_0)=\PK$, otherwise condition (A) is violated.
Similarly $\Lambda_\Delta(S_1)=\PK$.
This shows that $\Lambda_\Delta$ does not satisfy condition (B) with $S = S_0$ and $T = S_1$.
%Write $\gamma=\Delta\setminus (\Delta-1)$
%There are independent sets $(A^0_\alpha)_{\alpha\in \Delta-1}$ such that $S_0$ together with $(A^0_\alpha)_{\alpha\in \Delta-1}$ cover $G$.
%We must have $\Lambda_\alpha(A^0_\alpha)=\PI$, since otherwise the condition (B) is violated.
%Consequently $\Lambda_{\gamma}(S_0)=\PK$, otherwise (A) is violated.
%Similarly $\Lambda_\gamma(S_1)=\PK$.
%This shows that $\Lambda_\gamma$ does not satisfy (B) with $S_0,S_1$.
\end{proof}

Next we describe maximal examples of graphs that satisfy the condition $\Delta$-(*). That is, we define a graph $H_\Delta$ that satisfies $\Delta$-(*), its chromatic number is $2\Delta-2$ and every other graph that satisfies $\Delta$-(*) admits a homomorphism in $H_\Delta$.
The graph $H_{3}$ is depicted in \cref{fig:h3}.

Recall that the (categorical) product $G \times H$ of graphs $G,H$ is the graph on $V(G) \times V(H)$, such that $((g,h),(g',h')) \in E(G \times H)$ iff $(g,g') \in E(G)$ and $(h,h') \in E(H)$.
  
  Write $P$ for the product $K_{\Delta-1}\times K_{\Delta-1}$. Let $V_0$ and $V_1$ be vertex disjoint copies of $K_{\Delta-1}$.
We think of vertices in $V_i$ and $P$ as having labels from $[\Delta-1]$ and $[\Delta-1] \times [\Delta-1]$, respectively.
The graph $H_\Delta$ is the disjoint union of $V_0$, $V_1$, $P$ and an extra vertex $\dagger$ that is connected by an edge to every vertex in $P$, and additionally, if $v$ is a vertex in $V_0$ with label $i\in [\Delta-1]$, then we connect it by an edge with $(i',j)\in P$ for every $i'\not=i$ and $j\in [\Delta-1]$, and if $v$ is a vertex in $V_1$ with label $j\in \Delta-1$, then we connect it by an edge with $(i,j')\in P$ for every $j'\not=j$ and $i \in [\Delta-1]$.

\begin{proposition}\label{pr:MinimalExample}
\begin{enumerate}
    \item \label{prop:hdelta}$H_\Delta$ satisfies $\Delta$-(*).
    \item \label{prop:hdeltachrom} $\chi(H_\Delta)=2\Delta-2$.
    \item \label{prop:hdeltamax} A graph $G$ satisfies $\Delta$-(*) if and only if it admits a homomorphism to $H_\Delta$.
\end{enumerate}

\end{proposition}
\begin{proof}
\eqref{prop:hdelta} Set $S_0=V(V_0)\cup \{\dagger\}$ and $S_1=V(V_1)\cup \{\dagger\}$.
By the definition there are no edges between $S_0$ and $S_1$.
Consider now, e.g., $V(H_\Delta)\setminus S_0$.
Let $A_j$ consist of all elements in $P$ that have second coordinate equal to $j$ together with the vertex in $V_1$ that has the label $j$.
By the definition, the set $A_i$ is independent and $\bigcup_{i\in [\Delta-1]}A_i$ covers $H_\Delta\setminus S_0$, and similarly for $S_1$.

\eqref{prop:hdeltachrom} By \eqref{prop:hdelta} and the claim after the definition of $\Delta$-(*), it is enough to show that $\chi(H_\Delta)\geq 2\Delta-2$. Towards a contradiction, assume that $c$ is a proper coloring of $H_\Delta$ with $<2\Delta-2$-many colors. Note the vertex $\dagger$ guarantees that  $|c(V(P))|\leq 2\Delta-4$, and also $\Delta-1 \leq |c(V(P))|$.

First we claim that there are no indices $i,j\in [\Delta-1]$ (even with $i=j$) such that $c(i,r)\not=c(i,s)$ and $c(r,j)\not=c(s,j)$ for every $s\not=r$: indeed, otherwise, by the definition of $P$ we had $c(i,r)\not=c(s,j)$ for every $r,s$ unless $(i,r)=(s,j)$, which would the upper bound on the size of $c(V(P))$.

Therefore, without loss of generality, we may assume that for every $i\in [\Delta-1]$ there is a color $\alpha_i$ and two indices $j_i \neq j'_i$ such that $c(i,j_i)=c(i,j'_i)=\alpha_i$. It follows form the definition of $P$ and $j_i \neq j'_i$ that
 $\alpha_i\not=\alpha_{i'}$ whenever $i\not= i'$. 
 
 Moreover, note that any vertex in $V_1$ is connected to at least one of the vertices $(i,j_i)$ and $(i,j'_i)$, hence none of the colors $\{\alpha_i\}_{i \in [\Delta-1]}$ can appear on $V_1$. Consequently, since $V_1$ is isomorphic to $K_{\Delta-1}$ we need to use at least $\Delta-1$ additional colors, a contradiction.

\eqref{prop:hdeltamax} First note that if $G$ admits a homomorphism into $H_\Delta$, then the pullbacks of the sets witnessing  $\Delta$-(*) will witness that $G$ has $\Delta$-(*).

Conversely, let $G$ be a graph that satisfies $\Delta$-(*).
Fix the corresponding sets $S_0,S_1$ together with $(\Delta-1)$-colorings $c_0,c_1$ of their complements.
We construct a homomorphism $\Theta$ from $G$ to $H_\Delta$. Let
\[\Theta(v)=
\begin{cases} \dagger &\mbox{if } v \in S_0 \cap S_1, \\
c_0(v) & \mbox{if } v \in S_1 \setminus S_0, \\
c_1(v) & \mbox{if } v \in S_0 \setminus S_1,
\\
(c_0(v),c_1(v)) & \mbox{if } v \not \in S_0 \cup S_1.
\end{cases} \]

Observe that $S=S_0\cap S_1$ is an independent set such that there is no edge between $S$ and $S_0 \cup S_1$.
Using this observation, one easily checks case-by-case that $\Theta$ is indeed a homomorphism.
%We show that this works as required.
%Let $v,w\in G$ span an edge.
%Then:
%\begin{itemize}
 %   \item If $v\in S$, then $w\not \in S_0\cup S_1$.
  %  Consequently, $(\Theta(v),\Theta(w))$ is an edge in $H_\Delta$.
  %  \item If $v\in S_0\setminus S$ and $c_0(v)=i$, then $w\not\in S_1\cup c_0^{-1}(i)$.
    %Consequently, $\Theta(w)\in V_0$ with index different from $i$, or $\Theta(w)=(i',j)\in P$ where $i'\not =i$, i.e., $(\Theta(v),\Theta(w))$ is an edge in $H_{\Delta}$.
    %Similar argument works for $S_1$.
    %\item If $v,w\in G\setminus (S_0\cup S_1)$, then we must have $c_0(v)\not=c_0(w)$ and %$c_1(v)\not=c_1(w)$.
    %Consequently, %$(\Theta(v),\Theta%(w))$ is an edge %in $H_\Delta$.
%end{itemize}
\end{proof}

Now, combining what we have so far, we can easily prove Theorem \ref{th:homomorphism}.

\begin{proof}[Proof of Theorem \ref{th:homomorphism}]
It follows that $\Pi_{H_\Delta}$ is not playable from \cref{pr:MinimalExample}, \cref{pr:CONdition*}. It is easy to see that if for a graph $G$ the LCL $\Pi_G$ is not playable then $\Pi_{G'}$ is not playable for every subgraph $G'$ of $G$. Since erasing a vertex decreases the chromatic number with at most one, for each $k\leq 2\Delta-2$ there is a subgraph $G_k$ of $H_\Delta$ with $\chi(G_k)=k$, such that $\Pi_{G_k}$ is not playable.

It follows from Theorems \ref{th:MainLOCAL} and \ref{thm:MainBorel} that there is a $\Delta$-regular Borel forest that admits no Borel homomorphism to any graph of $G_k$ and that $\Pi_{G_k} \not \in \local(O(\log^*(n)))$.

Finally, note that if $k \geq \Delta$ then $G_k$ can be chosen so that it contains $K_\Delta$, yielding $\Pi_{G_k} \in \rlocal(O(\log\log(n))$.
\end{proof}

\begin{remark}
\label{r:hedet}Recall that Hedetniemi's conjecture is the statement that if $G,H$ are finite graphs then $\chi(G \times H)=\min\{\chi(G),\chi(H)\}$. This conjecture has been recently disproven by Shitov \cite{shitov}, and strong counterexamples have been constructed later (see, \cite{tardif2019note,zhu2021note}). We claim that these imply for $\varepsilon>0$ the existence of finite graphs $H$ with $\chi(H) \geq (2-\varepsilon)\Delta$ to which $\Delta$-regular Borel forests cannot have a homomorphism in BOREL, for every large enough $\Delta$. Indeed, if a $\Delta$-regular Borel forest admitted a Borel homomorphism to each finite graph of chromatic number at least $(2-\varepsilon)\Delta$, it would have such a homomorphism to their product as well. Thus, we would obtain that the chromatic number of the product of any graphs of chromatic number $(2-\varepsilon)\Delta$ is at least $\Delta+1$. This contradicts Zhu's result \cite{zhu2021note}, which states that the chromatic number of the product of graphs with chromatic number $n$ can drop to $\approx \frac{n}{2}$.

\end{remark}

\begin{remark}
A natural attempt to construct graphs with large girth and not playable homomorphisms problem would be to consider random $d$-regular graphs of size $n$ for a large enough $n$. However, it is not hard to see that setting $\Lambda(A)=\PI$ if and only if $|A|<\frac{n}{d}$ shows that this approach cannot work.
\end{remark}

\subsection{Proof of \cref{th:MainLOCAL}}\label{sec:marksproof}

In this section we prove Theorem \ref{th:MainLOCAL}, by applying Marks' game technique in the LOCAL setting. In order to define our games, we will need certain auxiliary graphs, the so-called \emph{ID graphs}. 
The purpose of these graphs is to define a ``playground'' for the games that we consider.
Namely, vertices in the game are labeled by vertices from the ID graph in such a way that the at the end we obtain a homomorphism from our underlying graph to the ID graph. 

%Let $\Pi$ be an LCL and $n\in \mathbb{N}$.
%We say that an algorithm $\fA_n$ of local complexity $t\in o(\log(n))$ \emph{solves $\Pi$} if whenever we get a rooted $\Delta$-regular edge colored tree of radius $t+1$ with injective labels from $[n]$ and we apply $\fA_n$ on the root and its neighbors, then we obtain a valid solution of $\Pi$ around the root.

%Recall that $\Pi$ is in the class $O(\log^*(n))$, or equivalently in the deterministic class $o(\log(n))$, if there are sequences $\fA=(\fA_n)_{n\in \mathbb{N}}$ and $t_n\in o(\log(n))$, where each $\fA_n$ is an algorithm of local complexity $t_n$ that solves $\Pi$.
%Our aim here is to use ideas from \cite{DetMarks} to show that if $\Pi$ is not playable, then it is not in the class $O(\log^*(n))$.
%To that end we define a two-player game
%$$\mathbb{G}(\fA_n,n,t,\idgraphcolored{n}{t}{r})[\alpha,\sigma,S],$$ where the players will alternatingly label vertices of a neighbourhood of a vertex in the tree, the labels will come from a edge-colored \emph{ID graph} $\idgraphcolored{n}{t}{r}{\Delta}$, with certain special properties (see below). The game will have as a parameter $\alpha\in \Delta$, $n\in \mathbb{N}$, an algorithm $\fA_n$ of local complexity $t\in o(\log(n))$, a vertex $\sigma$ of the ID graph, and $S\in \mathbb{P}_{\alpha}$, i.e., $S\subseteq \Sigma_{out}$.

\begin{definition}
A pair $\idgraphcolored{n}{t}{r}=(\idgraph{n}{t}{r},c)$ is called an \emph{ID graph}, if
\begin{enumerate}
    \item \label{p:idgirth} $\idgraph{n}{t}{r}$ is graph with girth at least $2t+2$,
    \item \label{p:idsize} $|V(\idgraph{n}{t}{r})|\le n$,
    \item \label{p:idlabels} $c$ is a $\Delta$-edge-coloring of $\idgraph{n}{t}{r}$, such that every vertex is adjacent to at least one edge of each color,
    \item \label{p:idratio} for each $\alpha \in [\Delta]$ the ratio of a maximal independent set of $\idgraph{n}{t}{r}^{\alpha}=(V(\idgraph{n}{t}{r}), E(\idgraph{n}{t}{r}) \cap c^{-1}(\alpha))$ is at most $r$ (i.e., $\idgraph{n}{t}{r}^\alpha$ is the graph formed by $\alpha$-colored edges).
\end{enumerate}
\end{definition}

%See \cref{subsec:Discussion} for an explanation why ID graphs are needed.
Before we define the game we show that ID graphs exist.

\begin{proposition}\label{pr:IDgraph}
Let $t_n\in o(\log(n))$, $r>0$, $\Delta\geq 2$.
Then there is an ID graph $\idgraphcolored{n}{t_n}{r}$ for every $n\in \mathbb{N}$ sufficiently large.
\end{proposition}
\begin{proof}
We use the \emph{configuration model} for regular random graphs, see \cite{wormald1999models}.
This model is defined as follows.
Let $n$ be even.
Then a $d$-regular random sample on $n$-vertices is just a union of $d$-many independent uniform random perfect matchings.
Note that in this model we allow parallel edges.

It was proved by Bollob\'as \cite{bollobas} that the independence ratio of a random $d$-regular graph is at most $\frac{2\log(d)}{d}$ a.a.s.
Moreover, this quantity is concentrated by an easy application of McDiarmid's result \cite{mcdiarmid2002concentration}, i.e., 
\begin{equation}
    \P\left(\left|X-\E(X)\right|\ge \sqrt{n}\right)<2\exp\left(-\frac{n}{d}\right),
\end{equation}
where $X$ is the random variable that counts the size of a maximal independent set.
Therefore for fixed $n$ large enough we have that the independence ratio of a random sample is at most $\frac{3\log(d)}{d}$ with probability at least $\left(1-2\exp\left(-\frac{n}{d}\right)\right)$.

%We show the existence of the ID graphs $\idgraph{n}{t}{r}$. The idea is to construct them as the union of $\Delta$-many independently sampled $d$-regulars graphs, and label the edges corresponding to the $\alpha$-th random graph by $\alpha$.

Pick a $d$ large enough such that $\frac{3\log(d)}{d}<r$.
Now for an $n$ large enough take $\Delta$-many independent samples of random $d$-regular graphs according to the configuration model.
Note that this is a random sample from the configuration model for $\Delta d$-regular graphs.
We define $c$ to be equal to $\alpha\in [\Delta]$ on edges of the $\alpha$-th sample.
Then condition (4) is satisfied with probability at least
$$\left(1-2\exp\left(-\frac{n}{d}\right)\right)^\Delta.$$

It remains to show that the girth condition is satisfied.
Recall that we assume $t_n\in o(\log n)$.
Then we have $(\Delta d-1)^{2t_n-1}\in o(n)$.
Using \cite[Corollary~1]{mckayshortcycles} we have that the probability of having girth at least $t_n$ is
\begin{equation}
\begin{split}
    \exp\left(-\sum_{a=3}^{t_n}\frac{(\Delta d-1)^a}{2a}+o(1)\right)\ge & \exp\left(-(\Delta d)^{t_n}\right) \\
    \ge & \exp(-o(n)) \\
    > & 1-\left(1-2\exp\left(-\frac{n}{d}\right)\right)^\Delta
\end{split}
\end{equation}
as $n\to \infty$.
This shows that there exists such a graph $\idgraphcolored{t_n}{n}{r}$ with non-zero probability.
\end{proof}

%ID graphs can be constructed as unions of $\Delta$-many random $d$-regular graphs for some large enough $d$, where the edges of the $\alpha$-th graph get the color $\alpha$. The details of the proof are postponed to Appendix \ref{a:idgraph}. 

Next, we define the games. As mentioned before, the games are going to depend on the following parameters: an algorithm $\fA_n$ of local complexity $t\in o(\log(n))$, an ID graph $\idgraphcolored{n}{t}{r}$, $\alpha\in \Delta$, $\sigma\in V(\idgraph{n}{t}{r})$ and $S \subseteq \Sigma$.
(We will view $\fA_n$, $\idgraphcolored{n}{t}{r}$ as fixed, and the rest of the parameters as variables).

The game
$$\mathbb{G}(\fA_n,n,t,\idgraphcolored{n}{t}{r})[\alpha,\sigma,S]$$
is defined as follows: two players, $\PI$ and $\PK$ assign labels to the vertices of a rooted $\Delta$-regular tree of diameter $t$. The labels are vertices of $\idgraph{n}{t}{r}$ and the root is labeled $\sigma$. In the $k$-th round, where $0<k\le t$, first $\PI$ labels vertices of distance $k$ from the root on the side of the $\alpha$ edge.
After that, $\PK$ labels all remaining vertices of distance $k$, etc (see \cref{fig:game}). We also require the assignment of labels to give rise to an edge-color preserving homomorphism to $\idgraphcolored{n}{t}{r}$. (For example, if it is $\PI$'s turn to label a neighbor of some vertex $v$, that has been assigned a label $\rho\in V(\idgraph{n}{t}{r})$ in the previous round, along an edge that has color $\beta$, then the allowed labels are only those that span a $\beta$ edge with $\rho$ in $\idgraph{n}{t}{r}$).

By property \eqref{p:idsize} of the ID graph, we can fix an injective map from $V(\idgraph{n}{t}{r})$ to $[n]$.
Now, we say that $\PI$ wins an instance of the game iff $\fA_n$ applied to the produced labeling of the rooted tree does {\bf not} produce an element of $S$ on the half edge that starts in the root and has edge color $\alpha$.
Note that this is well defined thanks to our assumption on the girth of $\idgraph{n}{t}{r}$.
Let us define $\Lambda^{\sigma}_{\alpha}(S)$ to be $\PI$ or $\PK$ depending on who has a winning strategy in the game $$\mathbb{G}(\fA_n,n,t,\mathbf{H}_{n,t,r})[\alpha,\sigma,S].$$
Since the game is finite, we have that
$$\Lambda^\sigma_\alpha:\mathbb{P}_{\alpha}\to \{\PI,\PK\}$$
is well-defined for all $\sigma\in V(\idgraph{n}{t}{r})$ and $\alpha\in [\Delta]$.

\begin{figure}
    \centering
    \includegraphics[]{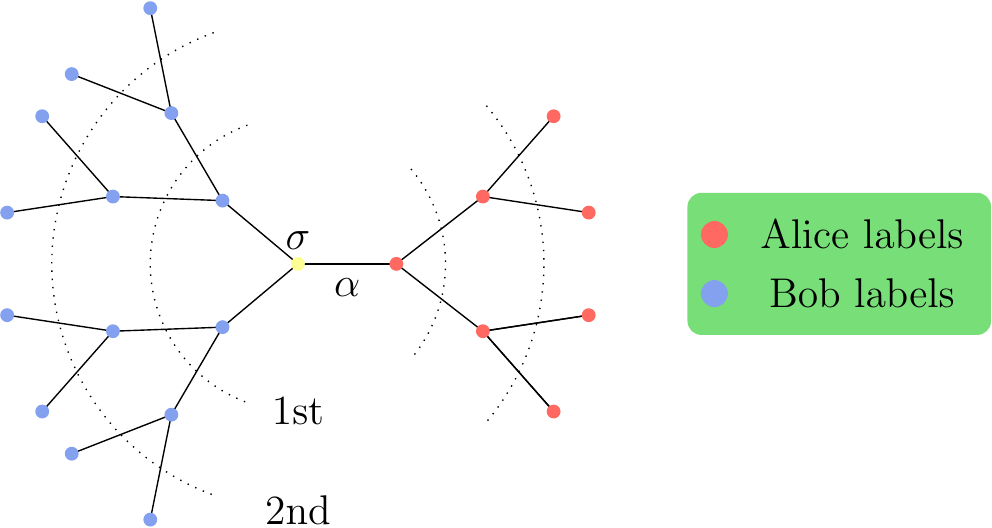}
    \caption{The game $\mathbb{G}(\fA_n,n,t,\idgraphcolored{n}{t}{r})[\alpha,\sigma,S]$}
    \label{fig:game}
\end{figure}

\begin{proposition}\label{pr:(A)satisfied}
Let $\fA_n$ be an algorithm of local complexity $t\in o(\log(n))$ that solves $\Pi$ and $\sigma\in V(\idgraph{n}{t}{r})$.
Then $(\Lambda^{\sigma}_{\alpha})_{\alpha\in \Delta}$ satisfies (A) in \cref{def:Playability}.
\end{proposition}
\begin{proof}
Assume that $(S_\alpha)_{\alpha\in \Delta}$ is such that $\Lambda^\sigma_\alpha(S_\alpha)=\PI$.
This means that $\PI$ has winning strategy in all the games corresponding to $S_\alpha$.
%Let us define $\Lambda^{\sigma}_{\alpha}(S)$ to be $\PI$ or $\PK$ depending on who has a winning strategy in the game $$\mathbb{G}(\fA_n,n,t,\mathbf{H}_{n,t,r})[\alpha,\sigma,S]$$
%for every $[\alpha]\in \Delta$.
Letting these strategies play against each other in the obvious way, produces a labeling of the tree.
Since $\fA_n$ solves $\Pi$ it has to output labeling of half edges $(\ta_{\alpha})_{\alpha\in \Delta}\in \mathcal{N}$, where $\ta_\alpha$ is a label on the half edge that start at the root and has color $\alpha$.
Note that we must have $\ta_\alpha\not\in S_\alpha$ by the definition of winning strategy for $\PI$.
This shows that (A) of Definition \ref{def:Playability} holds.
\end{proof}

We are ready to prove Theorem~\ref{th:MainLOCAL}. 
\begin{proof}[Proof of Theorem \ref{th:MainLOCAL}]
Let $(\fA_n)_{n\in \mathbb{N}}$ be a sequence of algorithms of local complexity $t_n\in o(\log(n))$ that solve $\Pi$ and $N\in \mathbb{N}$ be the number of all possible colorings of vertices of $(\mathbb{P}_{\alpha})_{\alpha\in \Delta}$ with two colors, that is, $N=2^{\sum_\alpha |V(\mathbb{P}_\alpha)|}$.
Set $r:=\frac{1}{N+1}$.
By \cref{pr:IDgraph} there exists an ID graph $\mathbf{H}_{n,t,r}$. Thus, the games above are well-defined and we can construct the functions $(\Lambda^\sigma_\alpha)_{\alpha\in \Delta}$. Since there are only $N$ possibilities for such a sequence, there exists a set $X \subseteq V(\idgraph{n}{t}{r})$ of relative size greater than $r$, such that for all $\sigma \in X$ the sequence of functions $(\Lambda^\sigma_\alpha)_{\alpha\in \Delta}$ is the same.

Since $\Pi$ is not playable, we can find an edge color $\alpha\in \Delta$ and sets $S,T\in \mathbb{P}_\alpha$ such that $\Lambda^\sigma_{\alpha}$ does not satisfy (B) from \cref{def:Playability} with $S,T$ for every $\sigma\in X$.
Note that this is because (A) is always satisfied by \cref{pr:(A)satisfied}.

By property \eqref{p:idratio} of the ID graph, there exist $\sigma_0,\sigma_1\in X$ that span an $\alpha$ edge in $\idgraph{n}{t}{r}$.
We let the winning strategies of $\PK$ in the games
$$\mathbb{G}(\fA_n,n,t,\mathbf{H}_{n,t,r})[\alpha,\sigma_0,S], \ \mathbb{G}(\fA_n,n,t,\mathbf{H}_{n,t,r})[\alpha,\sigma_1,T],$$
play against each other, where we start with an $\alpha$ edge with endpoints labeled by $\sigma_0,\sigma_1$.
This produces a labeling of the vertices that have distance at most $t$ from either of the endpoints of the edge (intuitively, the tree ``rooted'' at this edge of ``diameter'' $t+\frac{1}{2}$).
Now applying $\fA_n$ on the vertex with label $\sigma_0$ produces $\ta_0\in S$ on the the half edge that starts at this vertex and has color $\alpha$.
Similarly, we produce $\ta_1\in T$.
However, $(\ta_0,\ta_1)\not\in {\bf E}_\alpha$ by the definition of an edge in $\mathbb{P}_{\alpha}$.
This shows that $\fA_n$ does not solve $\Pi$.
\end{proof}

\subsection{Proof of \cref{thm:MainBorel}}
\label{subsec:local_to_borel}

We show how to use an infinite analogue of the ID graph to prove our main result in the Borel context, \cref{thm:MainBorel}.
Finding such a graph/graphing is in fact much easier in this context.

\begin{definition}
Let $r>0$. $\idgraphingcolored{r}=(\idgraphing{r},c)$ is an \emph{ID graphing}, if

\begin{enumerate}
    \item $\idgraphing{r}$ is an acyclic locally finite Borel graphing on a standard probability measure space $(X,\mu)$,
    \item \label{p:idgraphinglabels} $c$ is a Borel $\Delta$-edge-coloring of $\idgraphing{r}$, such that every vertex is adjacent to at least one edge of each color,
    \item \label{p:idgraphingratio} for each $\alpha \in [\Delta]$ the $\mu$-measure of a maximal independent set of $\idgraphing{r}^{\alpha}=(V(\idgraphing{r}), E(\idgraphing{r}) \cap c^{-1}(\alpha))$ is at most $r$.
\end{enumerate}
\end{definition}

\begin{proposition}
\label{p:idgraphing}
For each $r>0$ there exists an ID graphing $\idgraphingcolored{r}$.
\end{proposition}
\begin{proof}
Let $\idgraphing{r}$ be a \emph{local-global limit} of the random graphs constructed in \cref{pr:IDgraph} (see, e.g., \cite{hatamilovaszszegedy} for the basic results about local-global convergence).
It is not hard to check that this limit satisfies the required properties.
\end{proof}

%$\idgraphing{r}$ can be constructed by taking the local-global limit of the finite graphs constructed in the previous section, for the details see Appendix \ref{a:idgraph}.

Now we are ready to prove the theorem. The proof will closely follow the argument given in the proof of the LOCAL version, i.e., Theorem \ref{th:MainLOCAL}, but can be understood without reading the latter.

\begin{proof}[Proof of \cref{thm:MainBorel}]
Let $\Pi$ be an LCL that is not playable, and $N\in \mathbb{N}$ be the number of all possible colorings of vertices of $(\mathbb{P}_{\alpha})_{\alpha\in \Delta}$ with two colors, that is, $N=2^{\sum_\alpha |V(\mathbb{P}_\alpha)|}$.
Set $r:=\frac{1}{N+1}$.

We define a Borel acyclic $\Delta$-regular graph $\fG$, with edges properly colored by $\Delta$, that does not admit a Borel solution of $\Pi$.
Vertices of $\fG$ are pairs $(x,A)$, where $x\in X$ is a vertex of $\idgraphing{r}$ and $A$ is a countable subgraph of $\idgraphing{r}$ that is a $\Delta$-regular tree that contains $x$ and the edge coloring of $\idgraphingcolored{r}$ induces a proper edge coloring of $A$.
We say that $(x,A)$ and $(y,B)$ are connected by an $\alpha$-edge in $\fG$ if $A=B$, $x,y$ are adjacent in $A$ and the edge that connects them has color $\alpha\in \Delta$.

Suppose for a contradiction that $\mathcal{A}$ is a Borel function that solves $\Pi$ on $\mathcal{G}$.

Next, we define a family of games parametrized by $\alpha \in [\Delta]$, $x \in V(\idgraphing{r})$ and $S \subseteq \Sigma$. For the reader familiar with Marks' construction, let us point out that for a fixed $x$, the games are analogues to the ones he defines, with the following differences: allowed moves are vertices of the ID graphing $\idgraphing{r}$ and restricted by its edge relation, and the winning condition is defined by a set of labels, not just merely one label.

So, the game
$$\mathbb{G}(\mathcal{A},\idgraphingcolored{r})[\alpha,x,S]$$ 
is defined as follows: $\PI$ and $\PK$ alternatingly label vertices of a $\Delta$-regular rooted tree. The root is labelled by $x$, and the labels come from $V(\idgraphing{r})$. In the $k$-th round, first $\PI$ labels vertices of distance $k$ from the root on the side of the $\alpha$ edge.
After that, $\PK$ labels all remaining vertices of distance $k$, etc (see  \cref{fig:game}). We also require the assignment of labels to give rise to an edge-color preserving homomorphism to $\idgraphingcolored{r}$. 

It follows from the acyclicity of $\idgraphing{r}$ that a play of a game determines a $\Delta$-regular rooted subtree of $\idgraphing{r}$ to which the restriction of the edge-coloring is proper. That is, it determines a vertex $(x,A)$ of $\mathcal{G}$. Let $\PI$ win iff the output of $\mathcal{A}$ on the half-edge determined by $(x,A)$ and the color $\alpha$ is \textbf{not} in $S$.

Define the function $\Lambda^x_{\alpha}:\mathbb{P}_\alpha\to \{\PI,\PK\}$ assigning the player to some $S \in V(\mathbb{P}_\alpha)$ who has a winning strategy in $\mathbb{G}(\mathcal{A},\idgraphingcolored{r})[\alpha,x,S]$.
Note that, since $\idgraphing{r}$ is locally finite, each player has only finitely many choices at each position. Thus, it follows from Borel Determinacy Theorem that $\Lambda^x_{\alpha}$ is well defined.

Now we show the analogue of Proposition \ref{pr:(A)satisfied}.

\begin{proposition}\label{pr:Borel(A)satisfied}
$(\Lambda^{x}_{\alpha})_{\alpha\in \Delta}$ satisfies (A) in \cref{def:Playability}.
\end{proposition}
\begin{proof}
Assume that $(S_\alpha)_{\alpha\in \Delta}$ is such that $\Lambda^x_\alpha(S_\alpha)=\PI$.
This means that $\PI$ has winning strategy in all the games corresponding to $S_\alpha$. Letting these strategies play against each other in the obvious way, produces a vertex $(x,A)$ of $\mathcal{G}$.
Since $\fA$ solves $\Pi$ it has to output labeling of half edges $(\ta_{\alpha})_{\alpha\in \Delta}\in \mathcal{N}$, where $\ta_\alpha$ is a label on the half edge that start at $(x,A)$ and has color $\alpha$.
Note that we must have $\ta_\alpha\not\in S_\alpha$ by the definition of winning strategy for $\PI$.
This shows that (A) of Definition \ref{def:Playability} holds.
\end{proof}

 Let $f$ be 
 defined by $$x\mapsto (\Lambda^x_{\alpha})_{\alpha\in\Delta}.$$ Note that $f$ has a finite range. Using the fact that the allowed moves for each player can be determined in a Borel way, uniformly in $x$, it is not hard to see that for each element $s$ in the range, $f^{-1}(s)$ is in the algebra generated by sets that can be obtained by applying game quantifiers to Borel sets (for the definition see \cite[Section 20.D]{kechrisclassical}). It has been shown by Solovay \cite{Solovay} and independently by Fenstad-Norman \cite{fenstadnorman} that such sets are provably $\Delta^1_2$, and consequently, measurable. Therefore, $f$ is a measurable map.
 
 As the number of sequences of functions $(\Lambda^x_{\alpha})_{\alpha\in\Delta}$ is $\leq N$, by the choice of $r$, there exists a Borel set $Y$ with $\mu(Y) > r$, such that $f$ is constant on $Y$.
 
Since $\Pi$ is not playable Proposition \ref{pr:Borel(A)satisfied} gives that there is an $\alpha\in \Delta$ and $S,T\in \mathbb{P}_\alpha$ that violate condition (B).
Recall that the measure of an independent Borel set of $\idgraphing{r}^\alpha$ is at most $r$.
That means that there are $x,y\in Y$ such that $(x,y)$ is an $\alpha$-edge in $\idgraphing{r}$.
We let the winning strategies of $\PK$ in the games
$$\mathbb{G}(\mathcal{A},\idgraphingcolored{r})[\alpha,x,S], \ \mathbb{G}(\mathcal{A},\idgraphingcolored{r})[\alpha,y,T],$$
play against each other, where we start with an $\alpha$ edge with endpoints labeled by $x,y$.
This produces a labeling of a $\Delta$-regular tree $A$ with labels from $\idgraphing{r}$ such that $(x,A)$ and $(y,A)$ span an $\alpha$-edge in $\idgraphing{r}$.
Applying $\mathcal{A}$ on the half-edge determined by $(x,A)$ and the color $\alpha$, we obtain $\ta_0\in S$.
Similarly, we produce $\ta_1\in T$ for $(y,A)$.
However, $(\ta_0,\ta_1)\not\in {\bf E}_\alpha$ by the definition of an edge in $\mathbb{P}_{\alpha}$.
This shows that $\fA$ does not produce a Borel solution to $\Pi$.
\end{proof}

\begin{remark}
One can give an alternative proof of the existence of a Borel $\Delta$-regular forest that does not admit a Borel homomorphism to $G_k$ that avoids using the graphing $\idgraphing{r}$ and uses the original example described by Marks \cite{DetMarks} instead.
The reason is that the example graphs $G_k$ contain $K_\Delta$, as discussed in the proof of \cref{th:homomorphism}.
This takes care of the ``non-free'' as in the argument of Marks.
Then it is enough to use the pigeonhole principle on $\mathbb{N}$ instead of the independence ratio reasoning.
\end{remark}

\section{Separation Of Various Complexity Classes}
\label{sec:separating_examples}

In this section we provide examples of several problems that separate some classes from \cref{fig:big_picture_trees}.
The examples show two things. First, we have $\localo(\poly\log 1/\eps) \not= \localo(O(\log\log 1/\eps))$. 
This shows that there are problems such that their worst case complexity is at least $\Theta(\log n)$ on finite $\Delta$-regular trees, but their average local complexity is constant. 
Second, we show that there are problems in the class $\borel$ that are not in the class $\rlocal(O(\log\log n)) = \rlocal(o(\log n))$, on $\Delta$-regular trees. This shows that one cannot in general hope that results from (Borel) measurable combinatorics can be turned into very efficient (sublogarithmic) distributed algorithms. 

\subsection{Preliminaries}
\label{subsec:preliminaries_separating}

We first show that there is a strong connection between randomized local complexities and uniform local complexities. Afterwards, we introduce a generic construction that turns any LCL into a new LCL. We later use this generic transformation to construct an LCL that is contained in the set $\borel \setminus \rlocal(O(\log \log n))$.

\paragraph{Uniform vs Local Randomized Complexity} 

We will now discuss the connections between uniform and randomized complexities. 
Note that the easy part of the connection between the two concepts is turning uniform local algorithms into randomized ones, as formalized in the following proposition. 

\begin{proposition}
\label{prop:bee_gees}
We have $\olocal(t(\eps)) \subseteq \rlocal(t(1/n^{O(1)}))$. 
\end{proposition}
\begin{proof}
We claim that an uniform local algorithm $\fA$ with an uniform local complexity of $t(\eps)$ can be turned into a local randomized algorithm $\fA'$ with a local complexity  of $t(1/n^{O(1)})$. 

The algorithm $\fA'$ simulates $\fA$ on an infinite $\Delta$-regular tree -- each vertex $u$ of degree less than $\Delta$ in the original tree pretends that the tree continues past its virtual half-edges and the random bits of $u$ are used to simulate the random bits in this virtual subtree. 
Choosing $\eps = 1/n^{O(1)}$, one gets that the probability of $\fA'$ needing to look further than $t(\eps)$ for any vertex is bounded by $1/n^{O(1)}$, as needed in the definition of the randomized local complexity. 
\end{proof}
%Note, that the constructed algorithm $\fA'$ is in fact uniform, that is, it does not need access the size of the graph $n$.

On the other hand, we will use the following proposition from \cite{grebik_rozhon2021toasts_and_tails}. It informally states that the existence of \emph{any} uniform local algorithm together with the existence of a sufficiently fast randomized local algorithm for a given LCL $\Pi$ directly implies an upper bound on the uniform complexity of $\Pi$.
\begin{proposition}[\cite{grebik_rozhon2021toasts_and_tails}]
\label{prop:local->uniform}
Let $\fA$ be an uniform local algorithm solving an LCL problem $\Pi$ such that its \emph{randomized} local complexity on finite $\Delta$-regular trees is $t(n)$ for $t(n) = o(\log n)$. 
Then, the uniform local complexity of $\fA$ on infinite $\Delta$-regular trees is $O(t(1/\eps))$.
\end{proposition}

\cref{prop:local->uniform} makes our life simpler, since we can take a known randomized distributed local algorithm with local complexity $g(n) = o(\log n)$ and check if it works without the knowledge of $n$. If yes, this automatically implies that the  uniform local complexity of the algorithm is $h(\eps) = O(g(1/\eps))$. 
In particular, combining \cref{prop:local->uniform} with the work of \cite{chang_kopelowitz_pettie2019exp_separation,chang_pettie2019time_hierarchy_trees_rand_speedup} and verifying that the algorithms of 
Fischer and Ghaffari~\cite[Section 3.1.1]{fischer-ghaffari-lll} and  Chang~et~al.~\cite[Section 5.1]{ChangHLPU20} work without the knowledge of $n$, we obtain the following result.

\begin{theorem}
\label{thm:classification_local_uniform_trees}
We have:
\begin{itemize}
    \item $\local(O(1)) = \localo(O(1))$
    \item $\rlocal(O(\log^* n)) = \localo(O(\log^* 1/\eps))$
    \item $\rlocal(O(\log\log n)) = \localo(O(\log\log 1/\eps))$
\end{itemize}
Moreover, there are no other possible uniform local complexities for $t(\eps) = o(\log 1/\eps)$. 
\end{theorem}
We note that the first two items are proven in \cite{grebik_rozhon2021toasts_and_tails}.
For the third item it suffices by known reductions to find an uniform algorithm
%version of the algorithm of Chang et al. \cite{ChangHLPU20} that solves 
solving a version of the distributed Lov\'asz Local Lemma (LLL) on so-called tree-structured dependency graphs considered in~\cite{ChangHLPU20}. %\todo{feels like we say the same thing three times}

\begin{proof}[Proof sketch]
The proof follows from \cref{prop:bee_gees} and the following ideas. 
The first item follows from the fact that any local algorithm with local complexity $O(1)$ can simply be made uniform. More specifically, there exists a constant $n_0$ --- depending only on the problem $\Pi$ --- such that the algorithm, being told the size of the graph is $n_0$, is correct on \emph{any} graph of size $n \ge n_0$. 

For the second item, by the work of \cite{chang_kopelowitz_pettie2019exp_separation,chang_pettie2019time_hierarchy_trees_rand_speedup}, it suffices to check that there is an uniform distributed $(\Delta + 1)$-coloring algorithm with uniform local complexity $O(\log^* 1/\eps)$. Such an algorithm was given in \cite{HolroydSchrammWilson2017FinitaryColoring}, or follows from the work in \cite{Korman_Sereni_Viennot2012Pruning_algorithms_+_oblivious_coloring} and \cref{prop:local->uniform}. 

Similarly, for the third item, by the work of \cite{chang_pettie2019time_hierarchy_trees_rand_speedup} it suffices to check that there is an uniform distributed algorithm for a specific LLL problem on trees
%version of the Lovász Local Lemma (LLL) on so-called tree-structured dependency graphs
with uniform local complexity $O(\log\log 1/\eps)$.
Such an algorithm can be obtained by combining the randomized \emph{pre-shattering} algorithm of Fischer and Ghaffari~\cite[Section 3.1.1]{fischer-ghaffari-lll} and the deterministic \emph{post-shattering} algorithm of Chang~et~al.~\cite[Section 5.1]{ChangHLPU20} in a \emph{graph shattering} framework~\cite{barenboim2016locality}, which solves the LLL problem 
%on bounded-degree trees 
with local complexity $O(\log\log n)$.
%There is a distributed randomized algorithm for this problem on trees with local complexity $O(\log\log n)$ \cite{ChangHLPU20}. 
By \cref{prop:local->uniform}, it suffices to check that this algorithm can be made to work even if it does not know the size of the graph $n$. We defer the details to \cref{sec:LLL}.
This finishes the proof of \cref{thm:classification_local_uniform_trees}. 
\end{proof}

\paragraph{Adding Paths}
Before we proceed to show some separation results, we define a certain construction that turns any LCL problem $\Pi$ into a new LCL problem $\overline{\Pi}$ with the following property. If the original problem $\Pi$ cannot be solved by a fast local algorithm, then the same holds for $\overline{\Pi}$. However, $\overline{\Pi}$ might be strictly easier to solve than $\Pi$ for $\borel$ constructions. 

\begin{definition}
Let $\Pi=(\Sigma,\fV,\fE)$ be an LCL. 
We define an LCL $\overline{\Pi}=(\Sigma',\fV',\fE')$ as follows.
Let $\Sigma'$ be $\Sigma$ together with one new label.
Let $\fV'$ be the union of $\fV$ together with any cardinality-$\Delta$ multiset that contains the new label exactly two times.
Let $\fE'$ be the union of $\fE$ together with the cardinality-$2$ multiset that contains the new label twice.
\end{definition}

In other words, the new label determines doubly infinite lines in infinite $\Delta$-regular trees, or lines that start and end in virtual half-edges in finite $\Delta$-regular trees.
Moreover, a vertex that is on such a line does not have to satisfy any other vertex constraint.
We call these vertices \emph{line-vertices} and each edge on a line a \emph{line-edge}.

%of distance at most $1$ from a vertex with two such edges, then there are no constraints on labels around of this vertex, except for the $\alpha$ label constraint above. In another words, every $\overline{\Pi}$-coloring is the same as $\Pi$-coloring except for the possibility of existence of lines around which there are no ``old'' $\Pi$ constraints. 

\begin{proposition}
\label{pr:AddingLineLLL}
Let $\Pi$ be an LCL problem such that $\overline{\Pi}\in \rlocal(t(n))$ for $t(n)  = o(\log(n))$. Then also $\Pi \in \rlocal(O(t(n)))$.
\end{proposition}
\begin{proof}
Let $\overline{\fA}$ be a randomized $\local$ algorithm that solves $\overline{\Pi}$ in $t(n) = o(\log n)$ rounds with probability at least $1 - 1/n^C$ for some sufficiently large constant $C$. We will construct an algorithm $\fA$ for $\Pi$ with complexity $t(n)$ that is correct with probability $1 - \frac{4}{n^{C/3 - 2}}$. The success probability of $\fA$ can then be boosted by ``lying to it'' that the number of vertices is $n^{O(1)}$ instead of $n$; this increases the running time by at most a constant factor and boosts the success probability back to $1 - 1/n^C$. 

Consider a $\Delta$-regular rooted finite tree $T$ of depth $10t(n)$ and let $u$ be its root vertex. 
Note that $|T| \le \Delta^{10t(n)+1} < n$, for $n$ large enough.

We start by proving that when running $\overline{\fA}$ on the tree $T$, then $u$ is most likely not a line-vertex. This observation then allows us to turn $\overline{\fA}$ into an algorithm $\fA$ that solves $\Pi$ on any $\Delta$-regular input tree.
Let $X$ be the indicator of $\fA$ marking $u$ as a line-vertex. Moreover, for $i \in [\Delta]$, let $Y_i$ be the indicator variable for the following event. The number of line edges in the $i$-th subtree of $u$ with one endpoint at depth $5t(n)$ and the other endpoint at depth $5t(n) + 1$ is odd.

By a simple parity argument we can relate the value of $X$ with the values of $Y_1,Y_2,\ldots,Y_\Delta$ as follows. If $X = 0$, that is, $u$ is not a line-vertex, then all of the $Y_i$'s have to be $0$, as each path in the tree $T$ is completely contained in one of the $\Delta$ subtrees of $u$. On the other hand, if $u$ is a line-vertex, then there exists exactly one path, the one containing $u$, that is not completely contained in one of the $\Delta$ subtrees of $u$. This in turn implies that exactly two of the $\Delta$ variables $Y_1,Y_2,\ldots,Y_\Delta$ are equal to $1$.

The random variables $Y_1,Y_2,\ldots,Y_\Delta$ are identically distributed and mutually independent. Hence, if $\P(Y_i = 1) > \frac{1}{n^{C/3}}$, then the probability that there are at least $3$ $Y_i$'s equal to $1$ is strictly greater than $\left(\frac{1}{n^{C/3}}\right)^3 = \frac{1}{n^C}$. This is a contradiction, as in that case $\overline{\fA}$ does not produce a valid output, which according to our assumption happens with probability at most $\frac{1}{n^C}$. 

Thus, we can conclude that $\P(Y_i = 1) \leq \frac{1}{n^{C/3}}$. By a union bound, this implies that all of the $Y_i$'s are zero with probability at least $\frac{1}{n^{C/3 - 1}}$, and in that case $u$ is not a line-vertex. 

Finally, the algorithm $\fA$ simulates $\overline{\fA}$ as if the neighborhood of each vertex was a $\Delta$-regular branching tree up to depth at least $t(n)$. 

It remains to analyze the probability that $\fA$ produces a valid solution for the LCL problem $\Pi$. To that end, let $v$ denote an arbitrary vertex of the input tree. The probability that the output of $v$'s half edges satisfy the vertex constraint of the LCL problem $\overline{\Pi}$ is at least $1 - 1/n^C$. Moreover, in the case that $v$ is not a line-vertex, which happens with probability at least $1 - \frac{1}{n^{C/3 - 1}}$, the output of $v$'s half edges even satisfy the vertex constraint of the LCL problem $\Pi$. Hence, by a union bound the vertex constraint of the LCL problem $\Pi$ around $v$ is satisfied with probability at least $1 - \frac{2}{n^{C/3 - 1}}$. With exactly the same reasoning, one can argue that the probability that the output of $\fA$ satisfies the edge constraint of the LCL problem $\Pi$ at a given edge is also at least $1 - \frac{2}{n^{C/3 - 1}}$. Finally, doing a union bound over the $n$ vertex constraints and $n-1$ edge constraints it follows that $\fA$ produces a valid solution for $\Pi$ with probability at least $1 - \frac{4}{n^{C/3 - 2}}$.
\end{proof}

\begin{remark}
The ultimate way how to solve LCLs of the form $\overline{\Pi}$ is to find a spanning forest of lines.
This is because once we have a spanning forest of lines, then we might declare every vertex to be a line-vertex and label every half-edge that is contained on a line with the new symbol in $\overline{\Pi}$.
The remaining half-edges can be labeled arbitrarily in such a way that the edge constraints are satisfied.

Put otherwise, once we are able to construct a spanning forest of lines and $\fE\not=\emptyset$, where $\Pi=(\Sigma,\fV,\fE)$, then $\overline{\Pi}$ can be solved.
Moreover, in that case the complexity of $\overline{\Pi}$ is upper bounded by the complexity of finding the spanning forest of lines.

Lyons and Nazarov \cite{lyons2011perfect} showed that $\perfmatch\in \fiid$ and it is discussed in  \cite{LyonsTrees} that the construction can be iterated $(\Delta-2)$-many times. After removing each edge that is contained in one of the $\Delta-2$ perfect matchings each vertex has a degree of two.
Hence, it is possible to construct a spanning forest of lines as fiid.
Consequently, $\overline{\Pi}\in \fiid$ for every $\Pi=(\Sigma,\fV,\fE)$ with $\fE\not=\emptyset$.
\end{remark}

\subsection{$\local(O(\log^* n)) \not= \borel$}

We now formalize the proof that $\local(O(\log^* n)) \not= \borel$ from \cref{sec:introseparation}. 
Recall that $\deltacol$ is the proper vertex $\Delta$-coloring problem.
The following claim directly follows by  \cref{pr:AddingLineLLL} together with the fact that $\deltacol\not\in \local(O(\log^* n))$.

\begin{claim}
We have $\overline{\deltacol}\not \in \local(O(\log^* n))$.
\end{claim}

On the other hand we show that adding lines helps to find a Borel solution.
This already provides a simple example of a problem in the set $\borel \setminus \local(O(\log^* n))$.

\begin{proposition}
We have $\overline{\deltacol} \in \borel$. 
\end{proposition}
\begin{proof}
By \cite{KST}, every Borel graph of finite maximum degree admits a Borel maximal independent set.
Let $\fG$ be a Borel $\Delta$-regular acyclic graph on a standard Borel space $X$.
By an iterative application of the fact above we find Borel sets $A_1,\dots,A_{\Delta-2}$ such that $A_1$ is a maximal independent set in $\fG$ and $A_i$ is a maximal independent set in $\fG\setminus \bigcup_{j<i}A_j$ for every $i>1$.
We think of $A_1,\dots,A_{\Delta-2}$ as color classes for the first $\Delta-2$ colors.
Let $B=X\setminus \bigcup_{i\in [\Delta-2]}A_i$ and let $\fH$ be the subgraph of $\fG$ determined by $B$. %\gtodo{ Also, do we formally introduce what a Borel maximal independent set is. J: I think that it should follow from something that when you say Borel + some standard notion, then it just means the standard notion but the set is moreover a Borel set}
It is easy to see that the maximum degree of $\fH$ is $2$.
In particular, $\fH$ consists of finite paths, one-ended infinite paths or doubly infinite paths.
We use the extra label in $\overline{\deltacol}$ to mark the doubly infinite paths.
It remains to use the remaining $2$ colors in $[\Delta]$ to define a proper vertex $2$-coloring of the finite and one-ended paths.
It is a standard argument that this can be done in a Borel way and it is easy to verify that, altogether, we found a Borel $\overline{\deltacol}$-coloring of $\fG$.
\end{proof}

\subsection{Examples and Lower Bound}

%In this subsection we show that the classes $\rlocal(o(\log n))$ and $\borel$ are incomparable and that $\localo(\poly\log 1/\eps) \not= \localo(O(\log\log 1/\eps))$.
%In particular, $\ffiid\not=\rlocal(o(\log n))$.
%Note that the fact that there are problems in $\rlocal(o(\log n))$ that are not in $\borel$ follows from the discussion above, e.g., $\deltacol$.
In this subsection we define two LCLs and show that they are not in the class $\rlocal(o(\log n))$.
In the next subsection we show that one of them is in the class $\localo(\poly\log 1/\eps)$ and the other in the class $\borel$.
Both examples that we present are based on the following relaxation of the perfect matching problem.

\begin{definition}[Perfect matching in power-$2$ graph]
Let $\twomatch$ be the perfect matching problem in the power-$2$ graph, i.e., in the graph that we obtain from the input graph by adding an edge between any two vertices that have distance at most $2$ in the input graph.%\todo{is it clear how precisely this is defined with respect to ``leaves are unconstrained''}
\end{definition}

We show that $\twomatch$ is not contained in $\rlocal(o(\log n))$. \cref{pr:AddingLineLLL} then directly implies that $\overline{\twomatch}$ is not contained in $\rlocal(o(\log n))$ as well.
On the other hand, we later use a one-ended spanning forest decomposition to show that $\twomatch$ is in $\localo(\poly\log 1/\eps)$ and a one or two-ended spanning forest decomposition to show that $\overline{\twomatch}$ is in $\borel$.

We first show the lower bound result.
The proof is based on a simple parity argument as in the proof of \cref{pr:AddingLineLLL}.

\begin{theorem}
The problems $\twomatch$ and $\overline{\twomatch}$ are not in the class $\rlocal(o(\log n))$.
\end{theorem}

\begin{proof}
Suppose that the theorem statement does not hold for $\twomatch$.
Then, by \cref{thm:basicLOCAL}, there is a distributed randomized algorithm solving $\twomatch$ with probability at least $1 - 1/n$. By telling the algorithm that the size of the input graph is $n^{2\Delta}$ instead of $n$, we can further boost the success probability to $1 - 1/n^{2\Delta}$. The resulting round complexity is at most a constant factor larger, as $\Delta = O(1)$.
We can furthermore assume that the resulting algorithm $\fA$ will always output for each vertex $v$ a vertex $M(v) \neq v$ in $v$'s $2$-hop neighborhood, even if $\fA$ fails to produce a valid solution. The vertex $M(v)$ is the vertex $v$ decides to be matched to, and if $\fA$ produces a valid solution it holds that $M(M(v)) = v$.
%We can assume that $t(n) \geq 2$ (since the existence of an algorithm with runtime $t(n)$ trivially implies the existence of an algorithm with runtime $t(n) + 2$).

Consider a fully branching $\Delta$-regular tree $T$ of depth $10t(n)$. Note that $|T| < n$ for $n$ large enough. Let $u$ be the root of $T$ and $T_i$ denote the $i$-th subtree of $u$ (so that $u \not\in T_i$).
Let $S_i$ denote the set of vertices $v$ in $T_i$ such that both $v$ and $M(v)$ have distance at most $2t(n)$ from $u$.
Note that $M(v)$ is not necessarily a vertex of $T_i$.
Let $Y_i$ be the indicator of whether $S_i$ is even.
%If $\fA$ does not provide a correct output on the given input graph, set $Y_i := \fail$, for all $i \in [\Delta]$.

Observe that, if $\fA$ does not fail on the given input graph, then the definition of the $S_i$ implies that every vertex in $\{ u \} \cup \bigcup_{i \in [\Delta]} S_i$ is matched to a vertex in $\{ u \} \cup \bigcup_{i \in [\Delta]} S_i$.
Hence, the number of vertices in $\{ u \} \cup \bigcup_{i \in [\Delta]} S_i$ is even, which implies that we cannot have $Y_1 = Y_2 = \ldots = Y_\Delta = 1$ unless $\fA$ fails.
Note that $Y_i$ depends only on the output of $\fA$ at the vertices in $T_i$ that have a distance of precisely $2t(n) - 1$ or $2t(n)$ to $u$ (as all other vertices in $T_i$ are guaranteed to be in $S_i$). 
Hence, the events $Y_i = 1$, $i \in [\Delta]$, are independent, and, since $\P(Y_1 = 1) = \ldots = \P(Y_{\Delta} = 1)$ (as all vertices in all $T_i$ see the same topology in their $t(n)$-hop view), we obtain $(\P(Y_1 = 1))^{\Delta} \le 1/n^{2\Delta}$, which implies $\P(Y_i = 1) \le 1/n^2$, for any $i \in [\Delta]$.

%Note that if $Y_1 = Y_2 = \dots = Y_\Delta = 1$, there is no valid solution to $\twomatch$, since, by definition of the $S_i$, every vertex in the total number of vertices that are at distance at most $2t(n)$ from $u$ and they are matched to a vertex at distance at most $2t(n)$ from $u$ is odd because the number if even in all subtrees of $u$ that do not include $u$ itself that is in this set. However, this number should clearly be even. As $Y_1, Y_2, \dots, Y_\Delta$ are independent, we get $\P(Y_i) \le 1/n^{2\Delta / \Delta} = 1/n^2$. 

Hence, with probability at least $1 - \Delta /n^2$ we have $Y_1=Y_2=\ldots = Y_\Delta = 0$, by a union bound.
Let $v_1, v_2, \dots, v_\Delta$ denote the neighbors of $u$ with $v_i$ being the root of subtree $T_i$.
Note that if $\fA$ does not fail and $Y_1=Y_2=\ldots = Y_\Delta = 0$, then it has to be the case that $u$ is matched to a vertex in some subtree $T_i$ (not necessarily $v_i$), while any vertex from $N_{-i}(u) := \{v_1, v_2, \dots, v_{i-1}, v_{i+1}, \dots, v_\Delta\}$ is matched with another vertex from $N_{-i}(u)$ (hence, we already get that $\Delta$ needs to be odd).
This is because each subtree needs to have at least one vertex matched to a different subtree (since $|S_i|$ is odd for each $i \in [\Delta]$).
If $M(u)$ is a vertex in $T_i$ and, for any $v \in N_{-i}(u)$, we have $M(v) \in N_{-i}(u)$, we say that $u$ orients the edge going to $v_i$ outwards.
%If $\fA$ fails, but still $Y_1=Y_2=\ldots = Y_\Delta = 0$, vertex $u$ chooses an incident edge uniformly at random and orients it outwards; if $Y_1=Y_2=\ldots = Y_\Delta = 0$ is not satisfied, $u$ does not orient any incident edge outwards.

Consider a path $(u_0 = u, u_1, \dots, u_k)$, where $k = 2t(n) + 3$.
By the argument provided above, the probability that $u$ orients an edge outwards is at least $1 - \Delta/n^2 - 1/n^{2\Delta}$, and, if $u$ orients an edge outwards, the probability that this edge is not $uu_1$ is $(\Delta - 1)/\Delta$, by symmetry.
Hence, we obtain (for sufficiently large $n$) that $\P(\fE_1) \geq 1/2$, where $\fE_1$ denotes the event that vertex $u$ orients an edge different from $uu_1$ outwards.
With an analogous argument, we obtain that $\P(\fE_2) \geq 1/2$, where $\fE_2$ denotes the event that vertex $u_k$ orients an edge different from $u_ku_{k-1}$ outwards.
Since $\fE_1$ and $\fE_2$ are independent (due to the fact that the distance between any neighbor of $u$ and any neighbor of $u_k$ is at least $2t(n) + 1$), we obtain $\P(\fE_1 \cap \fE_2) \geq 1/4$.
Moreover, the probability that all vertices in $\{ u_1, \dots, u_{k-1} \}$ orient an edge outwards is at least $1 - (2t(n)+2)(\Delta/n^2 + 1/n^{2\Delta})$, which is at least $4/5$ for sufficiently large $n$.
Thus, by a union bound, there is a probability of at least $1 - 3/4 - 1/5 - 1/n^{2\Delta}$ that $\fA$ does not fail, all vertices in $(u, u_1, \dots, u_k)$ orient an edge outwards, and the outward oriented edges chosen by $u$ and $u_k$ are not $uu_1$ and $u_ku_{k-1}$, respectively.
For sufficiently large $n$, the indicated probability is strictly larger than $0$.
We will obtain a contradiction and conclude the proof for $\twomatch$ by showing that there is no correct solution with the described properties.

Assume in the following that the solution provided by $\fA$ is correct and $u, u_1, \dots, u_k$ satisfy the mentioned properties.
Since $u$ orients an edge different from $uu_1$ outwards, $u_1$ must be matched to a neighbor of $u$, which implies that $u_1$ orients edge $u_1u$ outwards.
Using an analogous argumentation, we obtain inductively that $u_j$ orients edge $u_ju_{j-1}$ outwards, for all $2 \leq j \leq k$.
However, this yields a contradiction to the fact that the edge that $u_k$ orients outwards is different from $u_ku_{k-1}$, concluding the proof (for $\twomatch$).

The theorem statement for $\overline{\twomatch}$ follows by applying \cref{pr:AddingLineLLL}.

%Let $T$ be the set of vertices at distance at most $t(n)+1$ from $u$. 
%Observe that the same argument that we did for $u$ we can do for any vertex in $T$. That is, with probability at least $1 - |T|/n^2 \ge 1 - 1/n$ every vertex $v$ in $T$ has the property that it orients outward one edge incident to some neighbor $w$, while the other neighbors of $v$ are matched together by $\fA$. 
%This means that every neighbor $w' \not= w, w' \in T$ of $v$ needs to orient the edge $wv$ outwards, as $w'$ is matched to a vertex in the subtree of $v$. \todo{figure would help?}

%Finally, consider two arbitrary vertices $v_1 \in T_1, v_2 \in T_2$ such their distance to $u$ is $t(n)+1$. Due to the symmetry, each $v_i$ orients outward each incident edge with probability $1/\Delta$ and these two decisions are independent. Hence, with probability at least $1/10$ both $v_1$ and $v_2$ orient outward a different edge than the one in the direction of $u$. This implies, by the previous consideration, that all vertices on the path from $u$ to $v_i$ are oriented in the outward direction. But this is contradiction with $u$ orienting only one edge outwards, with high probability.
\end{proof}

%It follows from \cite{?} that $\deltacol\in \rlocal(o(\log n))$ and we already mentioned that Marks proved that $\deltacol\not \in \borel$.
%We show that there is an LCL

%Next, we use \cref{pr:AddingLineLLL} to show that $\borel\setminus \local(O(\log\log(n)))\not=\emptyset$ and $\localo(\poly\log 1/\eps) \not= \localo(O(\log\log 1/\eps))$. 

\subsection{Upper Bounds Using Forest Decompositions}

We prove an upper bound for the problems $\twomatch$ and $\overline{\twomatch}$ defined in the previous subsection.
We use a technique of decomposing the input graph in a spanning forest with some additional properties, i.e., one or two ended. This technique was used in \cite{BrooksMeas} to prove Brooks' theorem in a measurable context.
Namely, they proved that if $\fG$ is a Borel $\Delta$-regular acyclic graph and $\mu$ is a Borel probability measure, then it is possible to erase some edges of $\fG$ in such a way that the remaining graph is a one-ended spanning forest on a $\mu$-conull set.
It is not hard to see that one can solve $\twomatch$ on one-ended trees in an inductive manner, starting with the leaf vertices. Consequently we have $\twomatch\in \measure$.
We provide a quantitative version of this result and thereby show that the problem of constructing a one-ended forest is contained in $\localo(\poly\log 1/\eps)$.
A variation of the construction shows that it is possible to find a one or two-ended spanning forest in a Borel way.
This allows us to show that $\overline{\twomatch}\in \borel$.
As an application of the decomposition technique we show that Vizing's theorem, i.e., proper $\Delta+1$ edge coloring $\Pi_{\chi',\Delta+1}$, for $\Delta=3$ is in the class $\localo(\poly \log 1/\eps)$.

\subsubsection{Uniform Complexity of the One-Forest Decomposition }
We start with the definition of a one-ended spanning forest.
It can be viewed as a variation of the edge grabbing problem $\edgegrab$.
Namely, if a vertex $v$ grabs an incident edge $e$, then we think of $e$ as being oriented away from $v$ and $v$ selecting the other endpoint of $e$ as its parent.
Suppose that $\fT$ is a solution of $\edgegrab$ on $T_\Delta$.
We denote with $\fT^{\leftarrow}(v)$ the subtree with root $v$.

\begin{definition}[One-ended spanning forest]
We say that a solution $\fT$ of $\edgegrab$ is a \emph{one-ended spanning forest} if $\fT^{\leftarrow}(v)$ is finite for every $v\in T_\Delta$ (see \cref{fig:one_ended}). 

\end{definition}

\begin{figure}
    \centering
    \includegraphics[width=.4\textwidth]{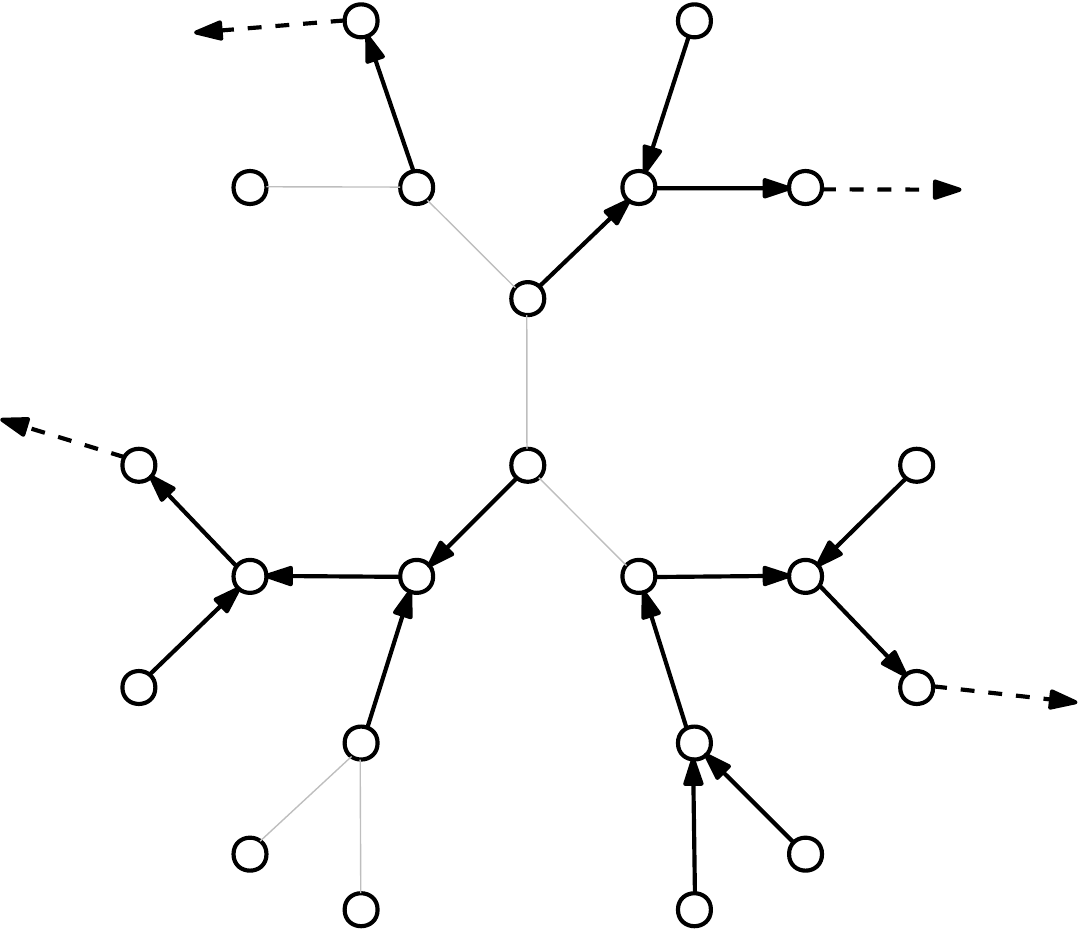}
    \caption{An example of a one-ended forest decomposition of an infinite 3-regular tree. }
    \label{fig:one_ended}
\end{figure}

Note that on an infinite $\Delta$-regular tree every connected component of a one-ended spanning forest must be infinite.
Furthermore, recall that we discussed above that it is possible to construct a one-ended spanning forest in $\measure$ by \cite{BrooksMeas}.
Next we formulate our quantitative version.

\begin{restatable}{theorem}{oneendedtreetail}
\label{thm:one_ended_forest}
The one-ended spanning forest can be constructed by an uniform local algorithm  with an uniform local complexity of $O(\poly\log1/\eps)$. 
More precisely, there is an uniform distributed algorithm $\fA$ that computes a one-ended spanning forest and if we define $R(v)$ to be the smallest coding radius that allows $v$ to compute $\fT^{\leftarrow}(v)$, then
\[
P(R(v) > O(\poly \log 1/\eps)) \le \eps. 
\]
\end{restatable}

The formal proof of \cref{thm:one_ended_forest} can be found in \cref{sec:formalproof}.
%Here we only provide a high-level sketch.
We first show the main applications of the result.
Namely, we show that $\twomatch$ can be solved inductively starting from the leaves of the one-ended trees. This proves that $\twomatch$ is in $\localo(\poly \log 1/\eps)$ and hence $\localo(\poly \log 1/\eps) \setminus \rlocal(o(\log n))$ is non-empty.

\begin{theorem}\label{thm:twomatchFFIID}
The problem $\twomatch$ is in the class $\measure$ and $\localo(O(\poly \log 1/\eps))$.
\end{theorem}
\begin{proof}
First we show that every infinite one-ended directed tree $T$ admits an inductive solution from the leaves to the direction of infinity. More precisely, we define a power-$2$ perfect matching on the infinite one-ended directed tree $T$ such that the following holds. Let $v$ be an arbitrary vertex and $T^{\leftarrow}(v)$ the finite subtree rooted at $v$. Then, all vertices in $T^{\leftarrow}(v)$, with the possible exception of $v$, are matched to a vertex in $T^{\leftarrow}(v)$. Moreover, for each vertex in $T^{\leftarrow}(v) \setminus \{v\}$, it is possible to determine to which vertex it is matched by only considering the subtree $T^{\leftarrow}(v)$. 
We show that it is possible to define a perfect matching in that way by induction on the height of the subtree $T^{\leftarrow}(v)$. 

We first start with the leaves and declare them as unmatched in the first step of the induction.

Now, consider an arbitrary vertex $v$ with $R \geq 1$ children that we denote by $v_1, v_2, \ldots, v_R$.
By the induction hypothesis, all vertices in $T^{\leftarrow}(v_i) \setminus \{v_i\}$ are matched with a vertex in $T^{\leftarrow}(v_i)$ for $i \in [R]$ and it is possible to compute that matching given $T^{\leftarrow}(v)$. However, some of the children of $v$ are possibly still unmatched. If the number of unmatched children of $v$ is even, then we simply pair them up in an arbitrary but fixed way and we leave $v$ unmatched. Otherwise, if the number of unmatched children of $v$ is odd, we pair up $v$ with one of its unmatched children and we pair up all the remaining unmatched children of $v$ in an arbitrary but fixed way. This construction clearly satisfies the criteria stated above. In particular, the resulting matching is a perfect matching, as every vertex gets matched eventually.

By \cref{thm:one_ended_forest} there is an uniform distributed algorithm of complexity $O(\poly \log 1/\eps)$ that computes a one-ended spanning forest $\fT$ on $T_\Delta$.
Moreover, every vertex $v\in T_\Delta$ can compute where it is matched with an uniform local complexity of $O(\poly \log 1/\eps)$.
This follows from the discussion above, as a vertex $v$ can determine where it is matched once it has computed $\fT^{\leftarrow}(w)$ for each $w$ in its neighborhood.
Hence, $\twomatch$ is in the class $\localo(O(\poly \log 1/\eps))$.

Similarly, the one-ended spanning forest $\fT\subseteq \fG$ can be constructed in the class $\measure$ by \cite{BrooksMeas}.
By the discussion above, this immediately implies that $\twomatch$ is in the class $\measure$. 
\end{proof}

\subsubsection{Decomposition in $\borel$}
Next, we show that a similar, but weaker, decomposition can be done in a Borel way.
Namely, we say that a subset of edges of an acyclic infinite graph $G$, denoted by $\fT$, is a \emph{one or two-ended spanning forest} if every vertex $v\in G$ is contained in an infinite connected component of $\fT$ and each infinite connected component of $\fT$ has exactly one or two directions to infinity.
Let $S$ be such a connected component.
Note that if $S$ has one end, then we can solve $\twomatch$ on $S$ as above.
If $S$ has two ends, then $S$ contains a doubly infinite path $P$ with the property that erasing $P$ splits $S\setminus P$ into finite connected components.
In order to solve $\overline{\twomatch}$ we simply solve $\twomatch$ on the finite connected components of $S\setminus P$ (with some of the vertices in $S\setminus P$ being potentially matched with vertices on the path $P$) and declare vertices on $P$ to be line vertices.

The high-level idea to find a one or two-ended spanning forest is to do the same construction as in the measure case and understand what happens with edges that do not disappear after countably many steps.
A slight modification in the construction guarantees that what remains are doubly infinite paths.

\begin{restatable}{theorem}{oneendedforestborel}
\label{thm:OneTwoEndBorel}
Let $\fG$ be a Borel $\Delta$-regular forest. 
Then there is a Borel one or two-ended spanning forest $\fT\subseteq \fG$.
\end{restatable}

The proof of \cref{thm:OneTwoEndBorel} can be found in \cref{sec:oneortwo}.
We remark that the notation used in the proof is close to the notation in the proof that gives a measurable construction of a one-ended spanning forest in \cite{BrooksMeas}. %Therefore, the proof is probably only accessible to people familiar with descriptive combinatorics. \todo{J:erase or formulate differently}
Next, we show more formally how \cref{thm:OneTwoEndBorel} implies that $\overline{\twomatch}\in \borel$.

\begin{theorem}
The problem $\overline{\twomatch}$ is in the class $\borel$.
\end{theorem}
\begin{proof}

Let $\fT$ be the one or two-ended spanning forest given by \cref{thm:OneTwoEndBorel} and $S$ be a connected component of $\fT$.
If $S$ is one-ended, then we use the first part of the proof of \cref{thm:twomatchFFIID} to find a solution of $\twomatch$ on $S$.
Since deciding that $S$ is one-ended as well as computing an orientation towards infinity can be done in a Borel way, this yields a Borel labeling.

Suppose that $S$ is two-ended.
Then, there exists a doubly infinite path $P$ in $S$ with the property that the connected components of $S\setminus P$ are finite.
Moreover, it is possible to detect $P$ in a Borel way.
That is, declaring the vertices on $P$ to be line vertices and using the special symbol in $\overline{\twomatch}$ for the half-edges on $P$ yields a Borel measurable labeling.
Let $C\not=\emptyset$ be one of the finite components in $S\setminus P$ and $v\in C$ be the vertex of distance $1$ from $P$.
Orient the edges in $C$ towards $v$, this can be done in a Borel way since $v$ is uniquely determined for $C$.
Then the first part of the proof of \cref{thm:twomatchFFIID} shows that one can inductively find a solution to $\twomatch$ on $C$ in such a way that all vertices, possibly up to $v$, are matched.
If $v$ is matched, then we are done.
If $v$ is not matched, then we add the edge that connects $v$ with $P$ to the matching. Note that it is possible that multiple vertices are matched with the same path vertex, but this is not a problem according to the definition of $\overline{\twomatch}$. 
It follows that this defines a Borel function on $H(\fG)$ that solves $\overline{\twomatch}$.
\end{proof}

\subsubsection{Vizing's Theorem for $\Delta=3$}
Finding a measurable or local version of Vizing's Theorem, i.e., proper edge $(\Delta+1)$-coloring $\vizing$, was studied recently in \cite{grebik2020measurable,weilacher2021borel,BernshteynVizing}.
It is however not known, even on trees, whether $\vizing$ is in $\rlocal(O(\log\log n))$.
Here we use the one-ended forest construction to show that $\vizing$ is in the class $\localo(\poly\log 1/\eps)$ for $\Delta=3$.

\begin{proposition}
Let $\Delta=3$. We have $\vizing \in\localo(\poly \log 1/\eps)$.
\end{proposition}
\begin{proof}[Proof sketch.]
By \cref{thm:one_ended_forest}, we can compute a one-ended forest decomposition $\fT$ with uniform local complexity $O(\poly \log 1/\eps)$. 
Note that every vertex has at least one edge in $\fT$, hence the edges in $G \setminus \fT$ form paths. 
These paths can be $3$-edge colored by using the uniform version of Linial's coloring algorithm. This algorithm has an uniform local complexity of $O(\log^* 1/\eps)$. By \cref{lem:sequential_composition}, the overall complexity is $O(\poly \log (\Delta^{\log^* 1/\eps} / \eps) + \log^* 1/\eps)) = O(\poly \log 1/\eps)$. 

Finally, we color the edges of $T$. We start from the leaves and color the edges inductively. In particular, whenever we consider a vertex $v$ we color the at most two edges in $T^{\leftarrow}(v)$ below $v$. 
Note that there always is at least one uncolored edge going from $v$ in the direction of $T$. Hence, we can color the at most two edges below $v$ in $T^{\leftarrow}(v)$ greedily -- each one neighbors with at most $3$ colored edges at any time.  
\end{proof}

\subsection{Proof of \cref{thm:one_ended_forest}} 
\label{sec:formalproof}

In this section we formally prove \cref{thm:one_ended_forest}.

\oneendedtreetail*

% \todo{Discuss where to put the proof sketch; J: I still think that having this as a first part of the actual proof is a good idea}
We follow the construction of the one-ended spanning forest in \cite{BrooksMeas}.
The construction proceeds in rounds. Before each round, some subset of the vertices have already decided which incident edge to grab, and we refer to these vertices as settled vertices. All the remaining vertices are called unsettled. The goal in each round is to make a large fraction of the unsettled vertices settled. To achieve this goal, our construction relies on certain properties that the graph induced by all the unsettled vertices satisfies. 

One important such property is that the graph induced by all the unsettled vertices is expanding. That is, the number of vertices contained in the neighborhood around a given vertex grows exponentially with the radius. The intuitive reason why this is a desirable property is the following. Our algorithm will select a subset of the unsettled vertices and clusters each unsettled vertex to the closest selected vertex. As the graph is expanding and any two vertices in the selected subset are sufficiently far away, there will be a lot of edges leaving a given cluster. For most of these inter-cluster edges, both of the endpoints will become settled. This in turn allows one to give a lower bound on the fraction of vertices that become settled in each cluster.

To ensure that the graph induced by all the unsettled vertices is expanding, a first condition we impose on the graph induced by all the unsettled vertices is that it has a minimum degree of at least $2$. While this condition is a first step in the right direction, it does not completely suffice to ensure the desired expansion, as an infinite path has a minimum degree of $2$ but does not expand sufficiently. Hence, our algorithm will keep track of a special subset of the unsettled vertices, the so-called hub vertices. Each hub vertex has a degree of $\Delta$ in the graph induced by all the unsettled vertices. That is, all of the neighbors of a hub vertex are unsettled as well. Moreover, each unsettled vertex has a bounded distance to the closest hub vertex, where the specific upper bound on the distance to the closest hub vertex increases with each round. As we assume $\Delta > 2$, the conditions stated above suffice to show that the graph induced by all the unsettled vertices expands.

Next, we explain in more detail how a vertex decides which edge to grab. Concretely, if a vertex becomes settled, it grabs an incident edge such that the other endpoint of that edge is strictly closer to the closest unsettled vertex as the vertex itself. For example, if a vertex becomes settled and there is exactly one neighbor that is still unsettled, then the vertex will grab the edge that the vertex shares with its unsettled neighbor. Grabbing edges in that way, one can show two things. First, for a settled vertex $v$, the set $T^{\leftarrow}(v)$ of vertices behind $v$ does not change once $v$ becomes settled. The intuitive reason for this is that the directed edges point towards the unsettled vertices and therefore the unsettled vertices lie before $v$ and not after $v$. Moreover, one can also show that $T^{\leftarrow}(v)$ only contains finitely many vertices. The reason for this is that at the moment a vertex $v$ becomes settled, there exists an unsettled vertex that is sufficiently close to $v$, where the exact upper bound on that distance again depends on the specific round in which $v$ becomes settled. 

What remains to be discussed is at which moment a vertex decides to become settled. As written above, in each round the algorithm considers a subset of the unsettled vertices and each unsettled vertex is clustered to the closest vertex in that subset. This subset only contains hub vertices and it corresponds to an MIS on the graph with the vertex set being equal to the set of hub vertices and where two hub vertices are connected by an edge if they are sufficiently close in the graph induced by all the unsettled vertices. Now, each cluster center decides to connect to exactly $\Delta$ different neighboring clusters, one in each of its $\Delta$ subtrees. Remember that as each cluster center is a hub vertex, all of its neighbors are unsettled as well. Now, each unsettled vertex becomes settled except for those that lie on the unique path between two cluster centers such that one of the cluster centers decided to connect to the other one.

\paragraph{Construction}
The algorithm proceeds in rounds. After each round, a vertex is either settled or unsettled and a settled vertex remains settled in subsequent rounds. Moreover, some unsettled vertices are so-called hub vertices. We denote the set of unsettled, settled and hub vertices after the $i$-th round with $U_i$, $S_i$ and $H_i$, respectively. We set $U_0 = H_0 = V$ prior to the first round and we always set $S_i = V \setminus U_i$.  \cref{fig:constr1} illustrates the situation after round $i$.  Moreover, we denote with $O_i$ a partial orientation of the edges of the infinite $\Delta$-regular input tree $T$. In the beginning, $O_0$ corresponds to the partial orientation with no oriented edges. Each vertex is incident to at most $1$ outwards oriented edge in $O_i$. If a vertex is incident to an outwards oriented edge in $O_i$, then the other endpoint of the edge will be its parent in the one-ended forest decomposition. For each vertex $v$ in $S_i$, we denote with $T_{v,i}$ the smallest set that satisfies the following conditions. First, $v \in T_{v,i}$. Moreover, if $u \in T_{v,i}$ and $\{w,u\}$ is an edge that according to $O_i$ is oriented from $w$ to $u$, then $w \in T_{v,i}$. We later show that $T_{v,i}$ contains exactly those vertices that are contained in the subtree $\fT^\leftarrow(v)$.

\begin{figure}[H]
    \centering
    \includegraphics{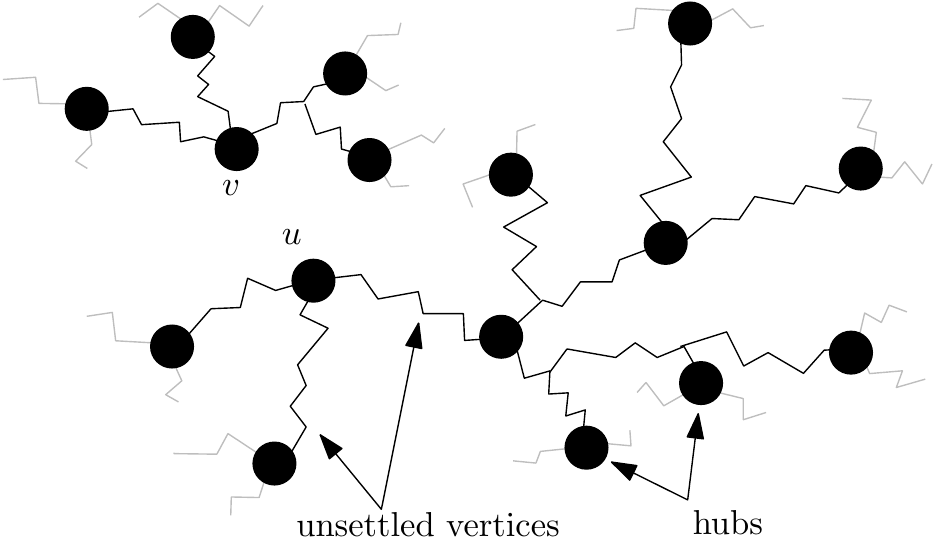}
    \caption{The picture illustrates the situation after some round $i$ of the construction. The vertices on the paths are all contained in $U_i$, while the vertices corresponding to the big black dots are additionally contained in $H_i$. Note that the length of each path is at least $d_i$. However, the distance between $u$ and $v$ in the original graph can be much smaller.}
    \label{fig:constr1}
\end{figure}

After the $i$-th round, the construction satisfies the following invariants.

\begin{enumerate}
    
    \item For each $v \in S_{i-1}$, we have $T_{v,i} = T_{v,i-1}$.
    
    \item Let $v$ be an arbitrary vertex in $S_i$. Then, $T_{v,i}$ contains finitely many vertices and furthermore $T_{v,i} \subseteq S_i$.
    
    \item The minimum degree of the graph $T[U_i]$ is at least $2$ and each vertex in $H_i$ has a degree of $\Delta$ in $T[U_i]$.
    
    \item Each vertex in $U_i$ has a distance of at most $\sum_{j=0}^i d_j$ to the closest vertex in $H_i$ in the graph $T[U_i]$.
\end{enumerate}

We now describe how to compute $U_i, S_i$, $H_i$ and $O_i$. The construction during the $i$-th round is illustrated in $\cref{fig:constr2}$. Note that we can assume that the invariants stated above are satisfied after the $(i-1)$-th round. In the $i$-th round we have a parameter $d_i$ that we later set to $2^{2^i}$.

\begin{enumerate}
    \item $H_i$ is a subset of $H_{i-1}$ that satisfies the following property. No two vertices in $H_i$ have a distance of at most $d_i$ in $T[U_{i-1}]$. Moreover, each vertex in $H_{i-1} \setminus H_i$ has a distance of at most $d_i$ to the closest vertex in $H_i$ in the graph $T[U_{i-1}]$. Note that we can compute $H_i$ by computing an MIS in the graph with vertex set $H_{i-1}$ and where two vertices are connected iff they have a distance of at most $d_i$ in the graph $T[U_{i-1}]$. Hence, we can use Ghaffari's MIS algorithm \cite{ghaffari2016MIS} to compute the set $H_i$.
    \item Next, we describe how to compute $U_i$.  
    We assign each vertex $u \in U_{i-1}$ to the closest vertex in the graph $T[U_{i-1}]$ that is contained in $H_i$, with ties being broken arbitrarily. We note that there exists a node in $H_i$ with a distance of at most $(\sum_{j=0}^{i-1} d_j) + d_i$ to $u$. To see why, note that Invariant $(4)$ from round $i - 1$ implies that there exists a vertex $w$ in $H_{i-1}$ with a distance of at most $\sum_{j=0}^{i-1} d_j$ to $u$. We are done if $w$ is also contained in $H_i$. If not, then it follows from the way we compute $H_i$ that there exists a vertex in $H_i$ with a distance of at most $d_i$ to $w$, but then the triangle inequality implies that the distance from $u$ to that vertex is at most $(\sum_{j=0}^{i-1} d_j) + d_i$, as desired. For each vertex $v \in H_i$, we denote with $C(v)$ the set of all vertices in $U_{i-1}$ that got assigned to $v$. Now, let $E_v$ denote the set of edges that have exactly one endpoint in $C(v)$. We can partition $E_v$ into $E_{v,1} \sqcup E_{v,2} \sqcup \ldots \sqcup E_{v,\Delta}$, where for $\ell \in [\Delta]$, $E_{v,\ell}$ contains all the edges in $E_v$ that are contained in the $\ell$-th subtree of $v$. 
    Invariants $(3)$ and $(4)$ after the $(i-1)$-th round imply that $E_{v,\ell} \neq \emptyset$. Moreover, $E_{v,\ell}$ contains only finitely many edges, as we have shown above that the cluster radius is upper bounded by $\sum_{j=0}^{i} d_j$.
    
    Now, for $\ell \in [\Delta]$, we choose an edge $e_\ell$ uniformly at random from the set $E_{v,\ell}$. Let $u_\ell \neq v$ be the unique vertex in $H_i$ such that one endpoint of $e_\ell$ is contained in $C(u_\ell)$. We denote with $P_{v,\ell}$ the set of vertices that are contained in the unique path between $v$ and $u_\ell$. Finally, we set $U_i = \bigcup_{v \in H_i, \ell \in [\Delta]} P_{v,\ell}$.
    \item It remains to describe how to compute the partial orientation $O_i$. All edges that are oriented in $O_{i-1}$ will be oriented in the same direction in $O_i$. Additionally, we orient for each vertex $u$ that got settled in the $i$-th round, i.e., $u \in S_i \cap U_{i-1}$, exactly one incident edge away from $u$. Let $w$ be the closest vertex to $u$ in $T[U_{i-1}]$ that is contained in $U_i$, with ties being broken arbitrarily. We note that it follows from the discussions above that the distance between $u$ and $w$ in the graph $T[U_{i-1}]$ is at most $\sum_{j=0}^{i} d_j$. We now orient the edge incident to $u$ that is on the unique path between $u$ and $w$ outwards. We note that this orientation is well-defined in the sense that we only orient edges that were not oriented before and that we don't have an edge such that both endpoints of that edge want to orient the edge outwards.
\end{enumerate}

\begin{figure}
    \centering
    \includegraphics{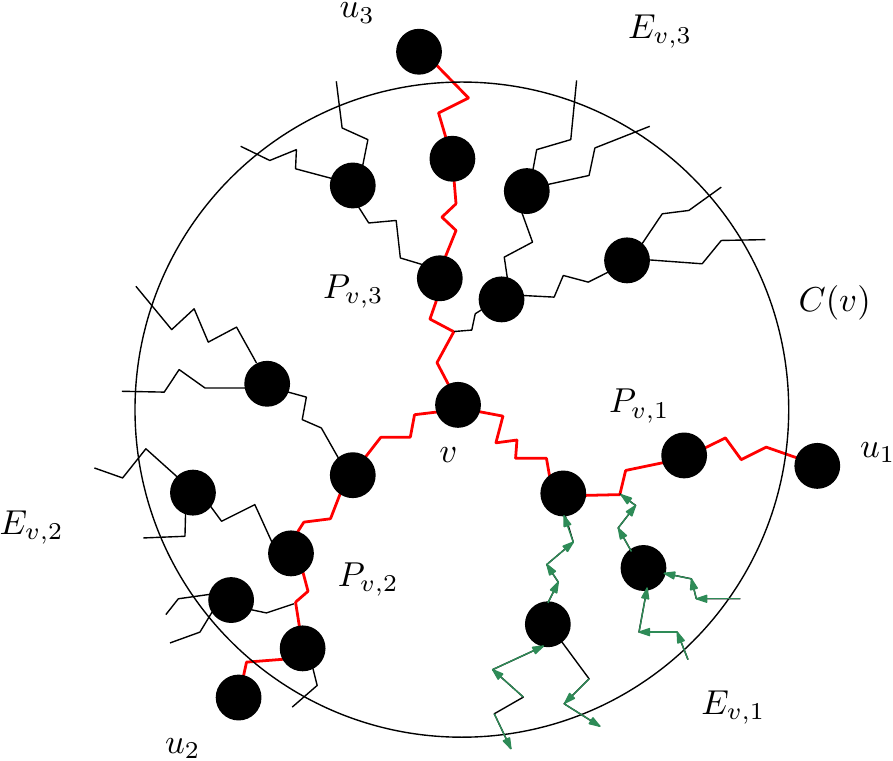}
    \caption{The picture illustrates the situation during some round $i$ of the construction. The vertices on all paths are contained in $U_{i-1}$ and the vertices corresponding to the big black dots are contained in $H_{i-1}$. The vertices $v$, $u_1$, $u_2$ and $u_3$ are the only vertices in the picture that are contained in $H_i$. All the vertices on the red paths remain unsettled during round $i$ while all the vertices on the black paths become settled. Each vertex that becomes settled during the $i$-th round orients one incident edge outwards, illustrated by the green edges. Each green edge is oriented towards the closest vertex that is still unsettled after the $i$-th round. We note that some of the edges in the picture are oriented away from $v$.}
    \label{fig:constr2}
\end{figure}

We now prove by induction that our procedure satisfies all the invariants.
Prior to the first round, all invariants are trivially satisfied. Hence, it remains to show that the invariants hold after the $i$-th round, given that the invariants hold after the $(i-1)$-th round. 

\begin{enumerate}

    \item Let $v \in S_{i-1}$ be arbitrary. As $O_i$ is an extension of $O_{i-1}$, it follows from the definition of $T_{v,i-1}$ and $T_{v,i}$ that $T_{v,i-1} \subseteq T_{v,i}$. Now assume that $T_{v,i} \not \subseteq T_{v,i-1}$. This would imply the existence of an edge $\{u,w\}$ such that $u \in T_{v,i-1}$ and the edge was oriented during the $i$-th round from $w$ to $u$. As $u \in T_{v,i-1}$, Invariant $(2)$ from the $(i-1)$-th round implies $u \in S_{i-1}$. This is a contradiction as during the $i$-th round only edges with both endpoints in $U_{i-1}$ can be oriented. Therefore it holds that $T_{v,i} = T_{v,i-1}$, as desired.

    \item Let $v \in S_i$ be arbitrary. If it also holds that $v \in S_{i-1}$, then the invariant from round $i-1$ implies that $T_{v,i-1} = T_{v,i}$ contains finitely many vertices and $T_{v,i-1} \subseteq S_{i-1} \subseteq S_i$. Thus, it suffices to consider the case that $v \in S_i \cap U_{i-1}$. Note that $T_{v,i} \cap U_i = \emptyset$. Otherwise there would exist an edge that is oriented away from a vertex in $U_i$ according to $O_i$, but this cannot happen according to the algorithm description. Hence, $T_{v,i} \subseteq S_i$. Now, for the sake of contradiction, assume that $T_{v,i}$ contains infinitely many vertices. From the definition of $T_{v,i}$, this implies that there exists a sequence of vertices $(v_k)_{k \geq 1}$ with $v = v_1$ such that for each $k \geq 1$, $\{v_{k+1},v_k\}$ is an edge in $T$ that is oriented from $v_{k+1}$ to $v_k$ according to $O_i$. For each $k \geq 1$ we have $v_k \in S_i$. We furthermore know that $v_k \in U_{i-1}$, as otherwise $T_{v_k,i} = T_{v_k,i-1}$ would contain infinitely many vertices, a contradiction. From the way we orient the edges, a simple induction proof implies that the closest vertex of $v_k$ in the graph $T[U_{i-1}]$ that is contained in $U_i$ has a distance of at least $k$. However, we previously discussed that the distance between $v_k$ and the closest vertex in $U_i$ in the graph $T[U_{i-1}]$ is at most $\sum_{j=0}^{i} d_j$. This is a contradiction. Hence, $T_{v,i}$ contains finitely many vertices.
    
    \item It follows directly from the description of how to compute $U_i$ and $H_i$ that the minimum degree of the graph $T[U_i]$ is at least $2$ and that each vertex in $H_i$ has a degree of $\Delta$ in $T[U_i]$.
    
    \item Let $u \in U_i$ be arbitrary. We need to show that $u$ has a distance of at most $\sum_{j=0}^i d_j$ to the closest vertex in $H_i$ in the graph $T[U_i]$. Let $v$ be the vertex with $u \in C(v)$. We know that there is no vertex in $H_i$ that is closer to $u$ in $T[U_{i-1}]$ than $v$. Hence, the distance between $u$ and $v$ is at most $\sum_{j=0}^i d_j$ in the graph $T[U_{i-1}]$. Moreover, from the description of how we compute $U_i$, it follows that all the vertices on the unique path between $u$ and $v$ are contained in $U_i$. Hence, each vertex in $U_i$ has a distance of at most $\sum_{j=0}^i d_j$ to the closest vertex in $H_i$ in the graph $T[U_i]$, as desired.
    \end{enumerate}

We derive an upper bound on the coding radius of the algorithm in three steps. First, for each $i \in \mathbb{N}$ and $\eps > 0$, we derive an upper bound on the coding radius for computing with probability at least $1 - \eps$ for a given vertex in which of the sets $S_i, U_i$ and $H_i$ it is contained in and for each incident edge whether and how it is oriented in $O_i$, given that each vertex receives as additional input in which of the sets $S_{i-1}, U_{i-1}$ and $H_{i-1}$ it is contained in and for each incident edge whether and how it is oriented in $O_{i-1}$. Given this upper bound on the coding radius, we can use the sequential composition lemma (\cref{lem:sequential_composition}) and a simple induction proof to give an upper bound on the coding radius for computing with probability at least $1 - \eps$ for a given vertex in which of the sets $S_i, U_i$ and $H_i$ it is contained in and for each incident edge whether and how it is oriented in $O_i$, this time without providing any additional input. 

Second, we analyze after how many rounds a given vertex is settled with probability at least $1- \eps$ for a given $\eps > 0$.  

Finally, we combine these two upper bounds to prove  \cref{thm:one_ended_forest}.

\begin{lemma}
\label{lem:tail1}
For each $i \in \mathbb{N}$ and $\eps \in (0,0.01]$, let $g_i(\eps)$ denote the smallest coding radius such that with probability at least $1 - \eps$ one knows for a given vertex in which of the sets $S_i, U_i$ and $H_i$ it is contained in and for each incident edge whether and how it is oriented in $O_i$, given that each vertex receives as additional input in which of the sets $S_{i-1}, U_{i-1}$ and $H_{i-1}$ it is contained in and for each incident edge whether and how it is oriented in $O_{i-1}$. It holds that $g_i(\eps) = O(d^2_i \log (\Delta/\eps))$.
\end{lemma}

\begin{proof}
 When running Ghaffari's uniform MIS algorithm on some graph $G'$ with maximum degree $\Delta'$, each vertex knows with probability at least $1-\eps$ whether it is contained in the MIS or not after $O(\log (\Delta') + \log (1/\eps))$ communication rounds in $G'$ (\cite{ghaffari2016MIS}, Theorem 1.1). For the MIS computed during the $i$-th round, $\Delta' =\Delta^{O(d_i)}$ and each communication round in $G'$ can be simulated with $O(d_i)$ communication rounds in the tree $T$. Hence, to know for a given vertex with probability at least $1-\eps$ whether it is contained in the MIS or not, it suffices consider the $O(d_i^2\log(\Delta) + d_i\log(1/\eps))$-hop neighborhood around that vertex. Now, let $u$ be an arbitrary vertex. In order to compute in which of the sets $S_i, U_i$ and $H_i$ $u$ is contained in and for each incident edge of $u$ whether and how it is oriented in $O_i$, we not only need to know whether $u$ is in the MIS or not. However, it suffices if we know for all the vertices in the $O(d_i)$-hop neighborhood of $u$ whether they are contained in the MIS or not (on top of knowing for each vertex in the neighborhood in which of the sets $S_{i-1}, U_{i-1}$ and $H_{i-1}$ it is contained in and for each incident edge whether and how it is oriented in $O_{i-1}$). Hence, by a simple union bound over the at most $\Delta^{O(d_i)}$ vertices in the $O(d_i)$-hop neighborhood around $u$, we obtain $g_i(\eps) = O(d_i) + O(d_i^2\log(\Delta) + d_i\log(\Delta^{O(d_i)}/\eps)) = O(d_i^2 \log(\Delta/\eps))$.
\end{proof}

\begin{lemma}
\label{lem:tail2}
For each $i \in \mathbb{N}$ and $\eps \in (0,0.01]$, let $h_i(\eps)$ denote the smallest coding radius such that with probability at least $1 - \eps$ one knows for a given vertex in which of the sets $S_i, U_i$ and $H_i$ it is contained in and for each incident edge whether and how it is oriented in $O_i$. Then, there exists a constant $c$ independent of $\Delta$ such that $h_i(\eps) \leq 2^{2^{i+2}} \cdot(c\log(\Delta))^i \log(\Delta/\eps)$.
\end{lemma}

\begin{proof}
By prove the statement by induction on $i$. For a large enough constant $c$, it holds that $h_1(\eps) \leq 2^{2^{3}} \cdot(c\log(\Delta)) \log(\Delta/\eps)$. Now, consider some arbitrary $i$ and assume that $h_i(\eps) \leq 2^{2^{i+2}} \cdot(c\log(\Delta))^i \log(\Delta/\eps)$ for some large enough constant $c$. We show that this implies $h_{i+1}(\eps) \leq  2^{2^{(i+1)+2}} \cdot(c\log(\Delta))^{i+1} \log(\Delta/\eps)$.
By the sequential composition lemma (\cref{lem:sequential_composition}) and assuming that $c$ is large enough, we have

\begin{align*}
  h_{i+1}(\eps) &\leq h_{i}((\eps/2)/\Delta^{g_{i+1}(\eps/2) + 1}) + g_{i+1}(\eps/2)  \\
                &\leq 2^{2^{i+2}} \cdot(c\log(\Delta))^i \log(\Delta \cdot \Delta^{g_{i+1}(\eps/2) + 1}/(\eps/2)) + g_{i+1}(\eps/2)  \\
                &\leq 2 \cdot 2^{2^{i+2}} \cdot(c\log(\Delta))^i \log(\Delta \cdot \Delta^{g_{i+1}(\eps/2) + 1}/(\eps/2)) \\
                &\leq 2 \cdot 2^{2^{i+2}} \cdot(c\log(\Delta))^i \log(\Delta^{(c/10) \cdot 2^{2^{i+2}}\log(\Delta/\eps)}/(\eps/2)) \\
                &\leq 4 \cdot 2^{2^{i+2}} \cdot(c\log(\Delta))^i \log(\Delta^{(c/10) \cdot 2^{2^{i+2}}\log(\Delta/\eps)}) \\
                &\leq (4/10) \cdot 2^{2^{i+2}} \cdot 2^{2^{i+2}}  (c \log(\Delta))^{i+1} \log(\Delta/\eps) \\
                &\leq 2^{2^{(i+1)+2}} \cdot(c\log(\Delta))^{i+1} \log(\Delta/\eps),
\end{align*}

as desired.

\end{proof}

\begin{lemma}\label{lm:tail0}
Let $u$ be an arbitrary vertex. For each $\eps \in (0,0.01]$, let $f(\eps)$ denote the smallest $i \in \mathbb{N}$ such that $u$ is settled after the $i$-th round with probability at least $1 - \eps$. There exists a fixed $c \in \mathbb{R}$ independent of $\Delta$ such that $f(\eps) \leq \lceil 1 + \log \log \frac{1}{c} \log_\Delta(1/\eps)\rceil$.
\end{lemma}

\begin{proof}

Let $i \in \mathbb{N}$ be arbitrary. We show that a given vertex is settled after the $i$-th round with probability at least $1 - O(1/\Delta^{\Omega(d_i/d_{i-1})})$. For the sake of analysis, we run the algorithm on a finite $\Delta$-regular high-girth graph instead of an infinite $\Delta$-regular tree. For now, we additionally assume that no vertex realizes that we don't run the algorithm on an infinite $\Delta$-regular tree. That is, we assume that for each vertex the coding radius to compute all its local information after the $i$-th round is much smaller than the girth of the graph.  

With this assumption, we give a deterministic upper bound on the fraction of vertices that are not settled after the $i$-th round. On the one hand, the number of vertices that are not settled after the $i$-th round can be upper bounded by $|H_i| \cdot \Delta \cdot O(d_i)$. On the other hand, we will show that the fraction of vertices that are contained in $H_i$ is upper bounded by $1/\Delta^{\Omega(d_i/d_{i-1})}$. We do this by showing that $C(v)$ contains $\Delta^{\Omega(d_i/d_{i-1})}$ many vertices for a given $v \in H_i$. Combining these two bounds directly implies that the fraction of unsettled vertices is smaller than 

\[\frac{\Delta \cdot O(d_i)}{\Delta^{\Omega(d_i/d_{i-1})}} = 1/\Delta^{\Omega(d_i/d_{i-1})}.\]

Let $v \in H_i$ be arbitrary. We show that $|C(v)| = \Delta^{\Omega(d_i/d_{i-1})}$. Using Invariants (3) and (4) together with a simple induction argument, one can show that there are at least $(\Delta-1)^{\lfloor  D / ((2\sum_{j=0}^{i-1}d_{j}) + 1)\rfloor}$ vertices contained in $H_{i-1}$ and whose distance to $v$ is at most $D$ in the graph induced by the vertices in $U_{i-1}$. Furthermore, from the way we defined the clustering $C(v)$ and the fact that two vertices in $H_i$ have a distance of at least $d_i$ in the graph $T[U_{i-1}]$, it follows that all vertices in $U_{i-1}$ having a distance of at most $d_i/2 -1$ to $v$  are contained in the cluster $C(v)$. Hence, the total number of vertices in $C(v)$ is at least $(\Delta-1)^{\lfloor  (d_i/2 - 1) / ((2\sum_{j=0}^{i-1}d_{j}) + 1)\rfloor} = \Delta^{\Omega(d_i/d_{i-1})}$, as promised.

Now we remove the assumption that no vertex realizes that we don't run the algorithm on an infinite $\Delta$-regular tree. If a vertex does realize that we don't run the algorithm on an infinite $\Delta$-regular tree, then we consider the vertex as being unsettled after the $i$-th round. By considering graphs with increasing girth, we can make the expected fraction of vertices that realize that we are not on an infinite $\Delta$-regular tree arbitrarily small.  Combining this observation with the previous discussion, this implies that the expected fraction of vertices that are not settled after the $i$-th round is at most $1/\Delta^{\Omega(d_i/d_{i-1})}$. By symmetry, each vertex has the same probability of being settled after the $i$-th round. Hence, the probability that a given vertex is settled after the $i$-th round when run on a graph with sufficiently large girth is $1 - 1/\Delta^{\Omega(d_i/d_{i-1})}$, and the same holds for each vertex when we run the algorithm on an infinite $\Delta$-regular tree. Thus, there exists a constant $c$ such that the probability that a given vertex is unsettled after the $i$-th round is at most $1/\Delta^{c \cdot d_i/d_{i-1}} = 1/\Delta^{c \cdot 2^{2^{i-1}}}$. Setting $i = \lceil 1 + \log \log \frac{1}{c} \log_\Delta(1/\eps)\rceil$ finishes the proof.
\end{proof}

We are now finally ready to finish the proof of \cref{thm:one_ended_forest}. Let $u$ be an arbitrary vertex and $\eps \in (0,0.01]$. We need to compute an upper bound on the coding radius that is necessary for $u$ to know all the vertices in $\mathcal{T}^{\leftarrow}(u)$ with probability at least $1-\eps$. To compute such an upper bound for the  required coding radius it suffices to find an $i^*$ and an $R^*$ such that the following holds. First, $u$ is settled after $i^*$ rounds with probability at least $1 - \eps/2$. Second, $u$ knows for each edge in its $O(d_{i^*})$-hop neighborhood whether it is oriented in the partial orientation $O_{i^*}$, and if yes, in which direction, by only considering its $R^*$-hop neighborhood with probability at least $1 - \eps/2$.  By a union bound, both of these events occur with probability at least $1-\eps$. Moreover, if both events occur $u$ knows all the vertices in the set $\mathcal{T}^{\leftarrow}(u)$. The reason is as follows. If $u$ is settled after round $i^*$, then it follows from the previous analysis that only vertices in the $O(d_{i^*})$-hop neighborhood of $u$ can be contained in $\mathcal{T}^{\leftarrow}(u)$. Moreover, for each vertex in its $O(d_{i^*})$-hop neighborhood, $u$ can determine if the vertex is contained in $\mathcal{T}^{\leftarrow}(u)$ if it knows all the edge orientations on the unique path between itself and that vertex after the $i^*$-th round. Hence, it remains to find concrete values for $i^*$ and $R^*$. According to \cref{lm:tail0}, we can choose $i^* = \lceil 1 + \log \log \frac{1}{c} \log_\Delta (2/\eps)\rceil$ for some large enough constant $c$. Moreover, it follows from a union bound that all vertices in the $O(d_{i^*})$-hop neighborhood around $u$ know with probability at least $1-\eps/2$ the orientation of all its incident edges according to $O_{i^*}$ by only considering their $h_{i^*}((\eps/2)/\Delta^{O(d_{i^*})})$-hop neighborhood. Hence, we can set 

\begin{align*}
R^* &= O(d_{i^*}) + h_{i^*}((\eps/2)/\Delta^{O(d_{i^*})}) \\
    &\leq O(d_{i^*}) + 2^{2^{i^* + 2}} \cdot (c \log(\Delta))^{i^*}\log(\Delta^{O(d_{i^*})}/\eps) \\
    &= O(2^{2^{i^*}}\cdot 2^{2^{i^* + 2}} (c \log(\Delta))^{i^*}\log(\Delta/\eps)) \\
    &= O(2^{2^{i^* + 3}} (c \log(\Delta))^{i^*}\log(\Delta/\eps)) \\
    &= O(2^{2^{\log \log \frac{1}{c} \log (2/\eps) + 5}} (c \log(\Delta))^{\log \log \frac{1}{c} \log (2/\eps) + 2}\log(\Delta/\eps)) \\
    &= \log(1/\eps)^{32} \cdot \log(\Delta/\eps) \cdot \log\log(1/\eps)^{O(\log\log(\Delta))}.
\end{align*}

As $\Delta = O(1)$, it therefore holds that

\[P(R(v)) = \poly(\log(1/\eps)) \geq 1 - \eps, \]

as desired.

\section{$\baire = \local(O(\log n))$}
\label{sec:baire}

In this section, we show that on $\Delta$-regular trees the classes $\baire$ and $\local(O(\log(n)))$ are the same.
At first glance, this result looks rather counter-intuitive.
This is because in finite $\Delta$-regular trees every vertex can see a leaf of distance $O(\log(n))$, while there are no leaves at all in an infinite $\Delta$-regular tree.
However, there is an intuitive reasons why these classes are the same: in both setups there is a technique to decompose an input graph into a hierarchy of subsets. Furthermore, the existence of a solution that is defined inductively with respect to these decompositions can be characterized by the same combinatorial condition of Bernshteyn \cite{Bernshteyn_work_in_progress}.
We start with a high-level overview of the decomposition techniques used in both contexts.

\paragraph{\rake\ and \compress}
The hierarchical decomposition in the context of distributed computing is based on a variant of a decomposition algorithm of
Miller and Reif~\cite{MillerR89}.
Their original decomposition algorithm works as follows.
Start with a tree $T$, and repeatedly apply the following two operations alternately: \rake\ (remove all degree-1 vertices) and \compress\ (remove all degree-2 vertices).
Then $O(\log n)$ iterations suffice to remove all vertices in $T$~\cite{MillerR89}. To view it another way, this produces a decomposition of the vertex set $V$ into $2L - 1$ layers \[V = \VR{1} \cup \VC{1} \cup \VR{2} \cup \VC{2} \cup \VR{3} \cup \VC{3} \cup \cdots \cup \VR{L},\] 
with $L= O(\log n)$, where $\VR{i}$ is the set of vertices removed during the $i$-th \rake\ operation and $\VC{i}$ is the set of vertices removed during the $i$-th \compress\ operation. We will use a variant~\cite{chang_pettie2019time_hierarchy_trees_rand_speedup} of this decomposition in the proof of \cref{thm:ellfull_to_logn}.

%\todo{This is just a note: the decomposition presented here is the original one, not the one we use in \cref{thm:ellfull_to_logn}. Presenting the simple one might be helpful to the readers in understanding why such a decomposition is called rake and compress. J: I would make this comment in to a proper sentence :)}

Variants of this decomposition turned out to be useful in designing  $\LOCAL$ algorithms~\cite{chang_pettie2019time_hierarchy_trees_rand_speedup,ChangHLPU20,chang2020n1k_speedups}. 
In our context, we assume that the given LCL satisfies a certain combinatorial condition and then find a solution inductively, in the reversed order of the construction of the decomposition.
Namely, in the \rake\ step we want to be able to existentially extend the inductive partial solution to all relative degree $1$-vertices (each $v \in \VR{i}$ has degree at most 1 in the subgraph induced by $\VR{i}  \cup \cdots \cup \VR{L}$) and in the \compress\ step we want to extend the inductive partial solution to paths with endpoints labeled from the induction (the vertices in $\VC{i}$ form degree-2 paths in the subgraph induced by $\VC{i} \cup \cdots \cup \VR{L}$).

\paragraph{$\toast$}
Finding a hierarchical decomposition in the context of descriptive combinatorics is tightly connected with the notion of \emph{Borel hyperfiniteness}.
Understanding what Borel graphs are Borel hyperfinite is a major theme in descriptive set theory \cite{doughertyjacksonkechris,gaojackson,conley2020borel}.
It is known that grids, and generally polynomial growth graphs are hyperfinite, while, e.g., acyclic graphs are not in general hyperfinite \cite{jackson2002countable}.
A strengthening of hyperfiniteness that is of interest to us is called a \emph{toast} \cite{gao2015forcing,conleymillerbound}.
A $q$-toast, where $q\in \mathbb{N}$, of a graph $G$ is a collection $\fD$ of fintie subsets of $G$ with the property that (i) every pair of vertices is covered by an element of $\fD$ and (ii) the boundaries of every $D\not=E\in \fD$ are at least $q$ apart.
The idea to use a toast structure to solve LCLs appears in \cite{conleymillerbound} and has many applications since then \cite{gao2015forcing,marksunger}.
This approach has been formalized in \cite{grebik_rozhon2021toasts_and_tails}, where the authors introduce $\toast$ algorithms.
Roughly speaking, an LCL $\Pi$ admits a $\toast$ algorithm if there is $q\in \mathbb{N}$ and a partial extending function (the function is given a finite subset of a tree that is partially colored and outputs an extension of this coloring on the whole finite subset) that has the property that whenever it is applied inductively to a $q$-toast, then it produces a $\Pi$-coloring.
%Player II has a winning strategy in the game, where Player I chooses partially colored finite subsets of the given graph $G$ and Player II tries to extend the coloring (where the extension depends only on the isomorphism type of what Player I played, not on the position in the graph). \todo{players feel unnatural to me. it is after marks and next to rakeandcompress!}
%The players are alternating, Player I has to produce a $q$-toast and Player II a $\Pi$-coloring.
An advantage of this approach is that once we know that a given Borel graph admits, e.g., a Borel toast structure and a given LCL $\Pi$ admits a $\toast$ algorithm, then we may conclude that $\Pi$ is in the class $\borel$.
Similarly for $\measure$, $\baire$ or $\olocal$, we refer the reader to \cite{grebik_rozhon2021toasts_and_tails} for more details and results concerning grids.

In the case of trees there is no way of constructing a Borel toast in general, however, it is a result of Hjorth and Kechris \cite{hjorth1996borel} that every Borel graph is hyperfinite on a comeager set for every compatible Polish topology.
A direct consequence of \cite[Lemma~3.1]{marks2016baire} together with a standard construction of toast via Voronoi cells gives the following strengthening to toast.
We include a sketch of the proof for completeness.

\begin{restatable}{proposition}{BaireToast}
\label{pr:BaireToast}
Let $\fG$ be a Borel graph on a Polish space $(X,\tau)$ with degree bounded by $\Delta\in \mathbb{N}$.
Then for every $q>0$ there is a Borel $\fG$-invariant $\tau$-comeager set $C$ on which $\fG$ admits a Borel $q$-toast.
\end{restatable}
\begin{proof}[Proof sketch]
Let $\{A_n\}_{n\in \mathbb{N}}$ be a sequence of MIS with parameter $f(n)$ as in \cite[Lemma~3.1]{marks2016baire} for a sufficiently fast growing function $f(n)$, e.g., $f(n)=(2q)^{n^2}$.
Then $C=\bigcup_{n\in \mathbb{N}}A_n$ is a Borel $\tau$-comeager set that is $\fG$-invariant.
We produce a toast structure in a standard way, e.g., see~\cite[Appendix~A]{grebik_rozhon2021toasts_and_tails}.

Let $B_n(x)$ denote the ball of radius $f(n)/3$ and $R_n(x)$ the ball of radius $f(n)/4$ around $x\in A_n$.
Iteratively, define cells $\fD_n=\{C_{n}(x)\}_{x\in A_n}$ as follows.
Set $C_1(x)=R_1(x)$.
Suppose that $\fD_n$ has been defined and set
\begin{itemize}
	\item $H^{n+1}(x,n+1):=R_{n+1}(x)$ for every $x\in A_{n+1}$,
	\item if $1\le i\le n$ and $\left\{H^{n+1}(x,i+1)\right\}_{x\in A_{k+1}}$ has been defined, then we put
	$$H^{n+1}(x,i)=\bigcup \left\{B_i(y):H^{n+1}(x,i+1)\cap B_i(y)\not=\emptyset\right\}$$
	for every $x\in A_{n+1}$,
	\item set $C_{n+1}(x):=H^{n+1}(x,1)$ for every $x\in A_{n+1}$, this defines $\mathcal{D}_{n+1}$.
\end{itemize}
The fact that $\fD=\bigcup_{n\in\mathbb{N}} \fD_n$ is a $q$-toast on $C$ follows from the fact that $C=\bigcup_{n\in \mathbb{N}}A_n$ together with the fact that the boundaries are $q$ separated which can be shown as in \cite[Appendix~A]{grebik_rozhon2021toasts_and_tails}.
\end{proof}

Therefore to understand LCLs in the class $\baire$ we need understand what LCLs on trees admit $\toast$ algorithm.
It turns out that these notions are equivalent, again by using the combinatorial characterization of Bernshteyn \cite{Bernshteyn_work_in_progress} that we now discuss.

\paragraph{Combinatorial Condition -- $\ell$-full set}
In both decompositions, described above, we need to extend a partial coloring along paths that have their endpoints colored from the inductive step.
The precise formulation of the combinatorial condition that captures this demand was extracted by Bernshteyn \cite{Bernshteyn_work_in_progress}.
He proved that it characterizes the class $\baire$ for Cayley graphs of virtually free groups.
Note that this class contains, e.g., $\Delta$-regular trees with a proper edge $\Delta$-coloring.

\begin{definition}[Combinatorial condition -- an $\ell$-full set]
\label{def:ellfull}
Let $\Pi = (\Sigma, \fV, \fE)$ be an LCL and $\ell\geq 2$.
A set $\fV' \subseteq \fV$ is \emph{$\ell$-full} whenever the following is satisfied.
Take a path with at least $\ell$ vertices, and add half-edges to it so that each vertex has degree $\Delta$.
Take any $c_1, c_2 \in \fV'$ and label arbitrarily the half-edges around the endpoints with $c_1$ and $c_2$, respectively.
Then there is a way to label the half-edges around the remaining $\ell-2$ vertices with configurations from $\fV'$ such that all the $\ell-1$ edges on the path have valid edge configuration on them. 
\end{definition}

Now we are ready to formulate the result that combines Bernshteyn's result \cite{Bernshteyn_work_in_progress} (equivalence between (1.) and (2.), and the moreover part) with the main results of this section.
This also shows the remaining implications in \cref{fig:big_picture_trees}.

\begin{theorem}\label{thm:MainBaireLog}
Let $\Pi$ be an LCL on regular trees. 
Then the following are equivalent:
\begin{enumerate}
    \item $\Pi\in \baire$,
    \item $\Pi$ admits a $\toast$ algorithm,
    \item $\Pi$ admits an $\ell$-full set,
    \item $\Pi\in \local(O(\log(n)))$.
\end{enumerate}
Moreover, any of the equivalent conditions is necessary for $\Pi\in \fiid$.
\end{theorem}

Next we discuss the proof of \cref{thm:MainBaireLog}.
We refer the reader to Bernshteyn's paper \cite{Bernshteyn_work_in_progress} for full proofs in the case of $\baire$ and $\fiid$, here we only sketch the argument for completeness.
We also note that instead of using the toast construction, he used a path decomposition of acyclic graphs of Conley, Marks and Unger \cite{conley2020measurable}.
%The proof of the difficult part of the equivalence between (2.) and (4.) is deferred to the appendix.

\subsection{Sufficiency}

We start by showing that the combinatorial condition is sufficient for $\baire$ and $\local(O(\log(n)))$.
Namely, it follows from the next results together with \cref{pr:BaireToast} that (2.) implies all the other conditions in \cref{thm:MainBaireLog}.
As discussed above the main idea is to color inductively along the decompositions.

\begin{proposition}\label{pr:toastable}
Let $\Pi=(\Sigma, \fV, \fE)$ be an LCL that admits $\ell$-full set $\fV'\subseteq \fV$ for some $\ell>0$.
Then $\Pi$ admits a $\toast$ algorithm that produces a $\Pi$-coloring for every $(2\ell+2)$-toast $\fD$.
\end{proposition}
\begin{proof}[Proof sketch]
Our aim is to build a partial extending function.
Set $q:=2\ell+2$.
Let $E$ be a piece in a $q$-toast $\fD$ and suppose that $D_1,\dots,D_k\in \fD$ are subsets of $E$ such that the boundaries are separated.
Suppose, moreover, that we have defined inductively a coloring of half-edges of vertices in $D=\bigcup D_i$ using only vertex configurations from $\fV'$ such that every edge configuration $\fE$ is satisfied for every edge in $D$.

We handle each connected component of $E\setminus D$ separately.
Let $A$ be one of them.
Let $u\in A$ be a boundary vertex of $E$.
Such an vertex exists since every vertex in $E$ has degree $\Delta$.
The distance of $u$ and any $D_i$ is at least $2\ell+2$ for every $i\in [k]$.
We orient all the edges from $A$ towards $u$.
Moreover if $v_i\in A$ is a boundary vertex of some $D_i$ we assign to $v_i$ a path $V_i$ of length $\ell$ towards $u$.
Note that $V_i$ and $V_j$ have distance at least $1$, in particular, are disjoint for $i\not=j\in [k]$ .
Now,  until you encounter some path $V_i$, color any in manner half-edges of vertices in $A$ inductively starting at $u$ in such a way that edge configurations $\fE$ are satisfied on every edge and only vertex configurations from $\fV'$ are used.
Use the definition of $\ell$-full set to find a coloring of any such $V_i$ and continue in a similar manner until the whole $A$ is colored.
\end{proof}

\begin{proposition}[$\ell$-full $\Rightarrow$ $\local(O(\log n))$]\label{thm:ellfull_to_logn}
Let $\Pi = (\Sigma, \fV, \fE)$ be an LCL with an $\ell$-full set $\fV' \subseteq \fV$. Then $\Pi$ can be solved in $O(\log n)$ rounds in  $\local$.
\end{proposition}
\begin{proof}
The proof  uses a variant of the rake-and-compress decomposition considered  in~\cite{chang_pettie2019time_hierarchy_trees_rand_speedup}.
%, and we follow the notations in~\cite{chang2020n1k_speedups}.
%The $O(\log n)$-round algorithm solving $\Pi$ is based on a variant of a decomposition algorithm of
%Miller and Reif~\cite{MillerR89}. Their decomposition algorithm works as follows. Start with a tree $T$, and repeatedly apply the following two operations alternately: \rake\ (remove all degree-1 vertices) and \compress\ (remove all degree-2 vertices). Then $O(\log n)$ iterations suffice to remove all vertices in $T$~\cite{MillerR89}.
%Variants of this decomposition  turned out to be useful in designing  $\LOCAL$ algorithms~\cite{chang_pettie2019time_hierarchy_trees_rand_speedup,ChangHLPU20,chang2020n1k_speedups}. 

\paragraph{The Decomposition}
%Here we  use a variant of the decomposition algorithm considered in~\cite{chang_pettie2019time_hierarchy_trees_rand_speedup}, and we follow the notations in~\cite{chang2020n1k_speedups}.
The decomposition is parameterized an integer $\ell' \geq 1$, and it decomposes the vertices of $T$ into $2L - 1$ layers \[V = \VR{1} \cup \VC{1} \cup \VR{2} \cup \VC{2} \cup \VR{3} \cup \VC{3} \cup \cdots \cup \VR{L},\] 
with $L= O(\log n)$. We write $\GC{i}$ to denote the subtree induced by the vertices  $\left(\bigcup_{j=i+1}^{L} \VR{j}\right) \cup \left( \bigcup_{j=i}^{L-1} \VC{j}\right)$. Similarly, $\GR{i}$ is the subtree induced by the vertices  
$\left(\bigcup_{j=i}^L \VR{j}\right) \cup \left( \bigcup_{j=i}^{L-1} \VC{j}\right)$.
The sets $\VR{i}$ and $\VC{i}$ are required to satisfy the following requirements. 
\begin{itemize}
    \item Each $v \in \VR{i}$ has degree at most one in the graph $\GR{i}$.
    \item Each $v \in \VC{i}$ has degree exactly two in the graph $\GC{i}$. Moreover, the $\VC{i}$-vertices in $\GC{i}$ form paths with $s$ vertices, with $\ell' \leq s \leq 2 \ell'$.
\end{itemize}
For any given constant $\ell' \geq 1$, it was shown in~\cite{chang_pettie2019time_hierarchy_trees_rand_speedup} that such a decomposition of a tree $T$ can be computed in $O(\log n)$ rounds. See \cref{fig:rake_and_compress} for  an example of such a decomposition with $\ell' = 4$.

\paragraph{The Algorithm}
Given such a decomposition with $\ell' = \max\{1, \ell - 2\}$, $\Pi$ can be solved in $O(\log n)$ rounds by labeling the vertices in this order: $\VR{L}$, $\VC{L-1}$, $\VR{L-1}$, $\ldots$, $\VR{1}$, as follows. The algorithm only uses the vertex configurations in the $\ell$-full set $\fV'$.

\paragraph{Labeling $\VR{i}$}
Suppose  all vertices in $\VR{L}, \VC{L-1}, \VR{L-1}, \ldots, \VC{i}$ have been labeled using $\fV'$.  Recall that each $v \in \VR{i}$ has degree at most one in the graph $\GR{i}$.
If $v \in \VR{i}$ has no neighbor in $\VR{L} \cup \VC{L-1} \cup \VR{L-1} \cup \cdots \cup \VC{i}$, then we can label the half edges surrounding $v$ by any $c \in \fV'$.
Otherwise, $v \in \VR{i}$ has exactly one neighbor $u$ in $\VR{L} \cup \VC{L-1} \cup \VR{L-1} \cup \cdots \cup \VC{i}$. Suppose the vertex configuration of $u$ is $c$, where the half-edge label on $\{u,v\}$ is $\ta \in c$.  A simple observation from the definition of $\ell$-full sets is that for any $c \in \fV'$ and any $\ta \in c$, there exist $c' \in \fV'$ and $\ta' \in c'$ in such a way that $\{\ta, \ta'\} \in \fE$.
Hence we can label the half edges surrounding $v$ by  $c' \in \fV'$  where the half-edge label on $\{u,v\}$ is $\ta' \in c'$.

\paragraph{Labeling $\VC{i}$} Suppose  all vertices in $\VR{L}, \VC{L-1}, \VR{L-1}, \ldots, \VR{i+1}$ have been labeled using $\fV'$. Recall that the   $\VC{i}$-vertices in $\GC{i}$ form degree-2 paths $P=(v_1, v_2, \ldots, v_s)$, with $\ell' \leq s \leq 2 \ell'$.
Let $P'= (x, v_1, v_2, \ldots, v_s, y)$ be the path resulting from appending to $P$ the neighbors of the two end-points of $P$ in $\GC{i}$. The two vertices  $x$ and $y$ are in $\VR{L} \cup \VC{L-1} \cup \VR{L-1} \cup \cdots \cup \VR{i+1}$, so they have been assigned half-edge labels using $\fV'$. Since  $P'$ contains at least $\ell'+2 \geq \ell$ vertices, the definition of $\ell$-full sets  ensures that we can label $v_1, v_2, \ldots, v_s$ using vertex configurations in $\fV'$ in such a way that the half-edge labels on $\{x, v_1\}, \{v_1, v_2\}, \ldots, \{v_s, y\}$ are all in $\fE$.
\end{proof}

\begin{figure}[h!]
    \centering
    \includegraphics[width= .8\textwidth]{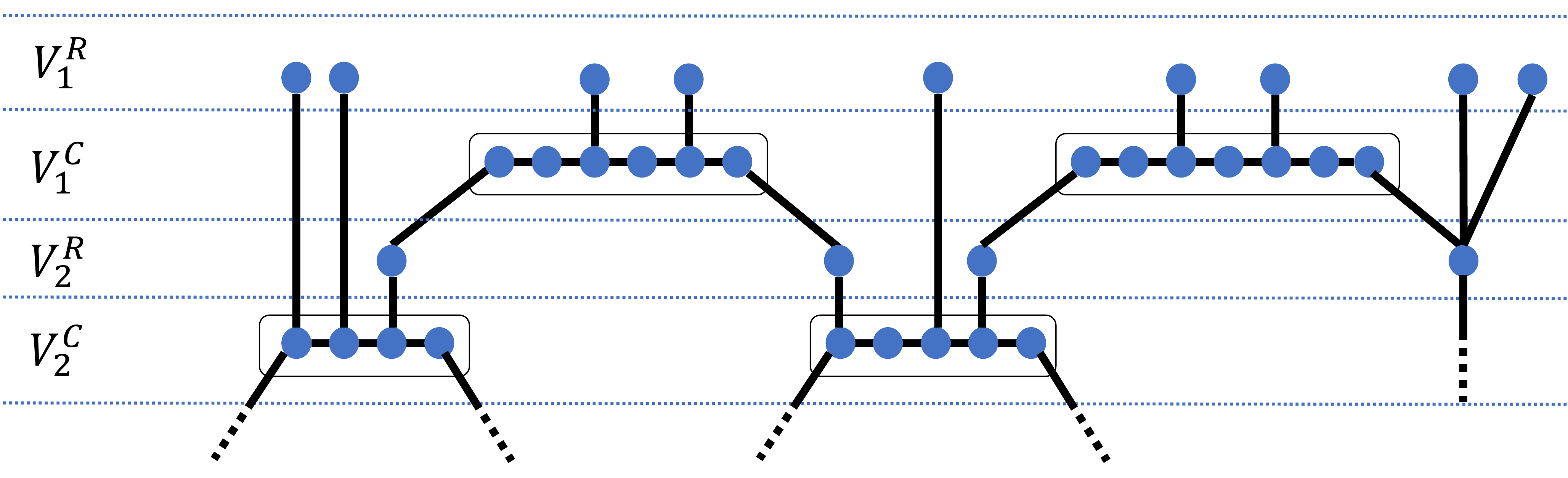}
    \caption{The variant of the rake-and-compress  decomposition used in the proof of \cref{thm:ellfull_to_logn}. }
    \label{fig:rake_and_compress}
\end{figure}

\subsection{Necessity}

We start by sketching that (2.) in \cref{thm:MainBaireLog} is necessary for $\baire$ and $\fiid$.

\begin{theorem}[Bernshteyn \cite{Bernshteyn_work_in_progress}]
Let $\Pi=(\Sigma,\fV,\fE)$ be an LCL and suppose that $\Pi\in \baire$ or $\Pi\in \fiid$.
Then $\Pi$ admits an $\ell$-full set $\fV'\subseteq\fV$ for some $\ell>0$.
\end{theorem}
\begin{proof}[Proof Sketch]

We start with $\baire$.
Suppose that every Borel acyclic $\Delta$-regular graph admits a Borel solution on a $\tau$-comeager set for every compatible Polish topology $\tau$.
In particular, this holds for the Borel graph induced by the standard generators of the free product of $\Delta$-copies of $\mathbb{Z}_2$ on the free part of the shift action on the alphabet $\{0,1\}$ endowed with the product topology.
Let $F$ be such a solution.
Write $\fV'\subseteq \fV$ for the configurations of half-edge labels around vertices that $F$ outputs on a non-meager set.
Let $C$ be a comeager set on which $F$ is continuous. Then, every element of $\fV'$ is encoded by some finite window in the shift on $C$, that is, for each element there are a $k \in \mathbb{N}$ and function $s:B(1,k) \to \{0,1\}$ such that $F$ is constant on the set $N_s \cap C$ (where $N_s$ is the basic open neighbourhood determined by $s$, and $B(1,k)$ is the $k$-neighbourhood of the identity in the Cayley graph of the group).
Since $\fV'$ is finite, we can take $t>0$ to be the maximum of such $k$'s.
It follows by standard arguments that $\fV'$ is $\ell$-full for $\ell>2t+1$.

A similar argument works for the $\fiid$, however, for the sake of brevity, we sketch a shorter argument that uses the fact that there must be a correlation decay for factors of iid's. Let $\Pi\in \fiid$.
That is, there is an $\operatorname{Aut}(T)$-equivariant measurable function from iid's on $T$ (without colored edges this time) into the space of $\Pi$-colorings.
Let $\fV'$ be the set of half-edges configurations around vertices that have non-zero probability to appear.
Let $u,v\in T$ be vertices of distance $k_0\in \mathbb{N}$.
By \cite{BGHV} the correlation between the configurations around $u$ and $v$ tends to $0$ as $k_0\to \infty$.
This means that if the distance is big enough, then all possible pairs of $\fV'$ configurations need to appear.
\end{proof}

To finish the proof of \cref{thm:MainBaireLog} we need to demonstrate the following theorem.  Note that  $\local(n^{o(1)}) =  \local(O(\log n))$ according to the $\omega(\log n)$ -- $n^{o(1)}$ complexity gap~\cite{chang_pettie2019time_hierarchy_trees_rand_speedup}.

\begin{restatable}{theorem}{logcomb}\label{thm:logn_to_ellfull}
Let $\Pi = (\Sigma, \fV, \fE)$ be an LCL solvable in  $\local(n^{o(1)})$ rounds.
Then there exists an $\ell$-full set  $\fV' \subseteq \fV$ for some $\ell \geq 2$.
\end{restatable}

%The formal proof is based on a machinery developed in~\cite{chang_pettie2019time_hierarchy_trees_rand_speedup} and it is deferred to \cref{app:LOG=BAIRE}.

The rest of the section is devoted to the proof of \cref{thm:logn_to_ellfull}. We start with the high-level idea of the proof.
%$\local(n^{o(1)})$ $\Rightarrow$ $\ell$-full. 
%First, we sketch a rough idea. 
%\todo{we have two sketches now!}
A natural attempt for showing $\local(n^{o(1)})$ $\Rightarrow$ $\ell$-full is to simply take any $\local(n^{o(1)})$ algorithm $\fA$ solving $\Pi$, and then take $\fV'$ to be all vertex configurations that can possibly occur in an output of $\fA$. It is not hard to see that this approach does not work in general, because the algorithm might use a special strategy to label vertices with degree smaller than $\Delta$. Specifically, there might be some vertex configuration $c$ used by $\fA$ so that some $\ta \in c$ will only be used to label \emph{virtual} half edges. It will be problematic to include $c$ in $\fV'$.

To cope with this issue, we do not deal with general bounded-degree trees. Instead, we construct recursively a sequence $(W^\ast_1, W^\ast_2, \ldots, W^\ast_L)$ of sets of rooted, layered, and \emph{partially labeled} tree in a special manner. A tree $T$ is included in $W^\ast_i$ if it can be constructed by gluing a multiset of rooted trees in  $W^\ast_{i-1}$ and a new root vertex $r$ in a certain fixed manner. A vertex is said to be in layer $i$ if it is introduced during the $i$-th step of the construction, i.e., it is introduced as the root $r$ during the construction of  $W^\ast_i$ from $W^\ast_{i-1}$. All remaining vertices are said to be in layer 0.

We  show that each $T \in W^\ast_L$ admits a correct labeling that extends the given partial labeling, as these partial labelings are computed by a simulation of $\fA$.
Moreover, in these correct labelings, the variety of possible configurations of half-edge labels around vertices in different non-zero layers is the same for each layer.
This includes vertices of non-zero layer whose half-edges are labeled by the given partial labeling.
We simply pick $\fV'$ to be the set of all configurations of half-edge labels around vertices that can appear in a non-zero layer in a correct labeling of a tree $T \in W^\ast_L$.
Our construction ensures that each $c \in \fV'$ appears as the labeling of some degree-$\Delta$ vertex in some tree that we consider.

The proof that $\fV'$ is an $\ell$-full set is based on finding paths  using vertices of non-zero layers connecting  two vertices with any two  vertex configurations in  $\fV'$ in different lengths. These paths exist because the way rooted trees in $W^\ast_{i-1}$ are glued together in the construction of $W^\ast_{i}$ is sufficiently flexible.
The reason that we need $\fA$ to have complexity $\local(n^{o(1)})$ is that the construction of the trees can be parameterized by a number $w$ so that all the trees have size polynomial in $w$ and the vertices needed to be assigned labeling are at least distance $w$ apart from each other. Since the number of rounds of $\fA$ executed on trees of size $w^{O(1)}$ is much less than $w$, each labeling assignment can be calculated locally and independently. 
The construction of the trees as well as the analysis are based on a machinery developed in~\cite{chang_pettie2019time_hierarchy_trees_rand_speedup}.
 Specifically, we will consider the equivalence relation $\simm$ defined in~\cite{chang_pettie2019time_hierarchy_trees_rand_speedup} and prove some of its properties, including a pumping lemma for bipolar trees. The exact definition of  $\simm$  in this paper is  different from the one in~\cite{chang_pettie2019time_hierarchy_trees_rand_speedup} because the mathematical formalism describing LCL problems in this paper is different from the one in~\cite{chang_pettie2019time_hierarchy_trees_rand_speedup}.
After that, we will consider a procedure for gluing trees parameterized by a labeling function $f$ similar to the one used in~\cite{chang_pettie2019time_hierarchy_trees_rand_speedup}. 
We will apply this procedure iteratively to generate a set of trees. We will show that the desired  $\ell$-full set $\fV' \subseteq \fV$ can be constructed by considering the set of all possible correct labeling of these trees.

\paragraph{The Equivalence Relation $\simm$} 
We consider trees with a list of designated vertices $v_1, v_2, \ldots, v_k$ called \emph{poles}. A \emph{rooted tree} is a tree with one pole $r$, and a \emph{bipolar tree} is a tree with two poles $s$ and $t$. 
For a tree $T$ with its poles $S=(v_1, v_2, \ldots, v_k)$ with $\deg(v_i) = d_i < \Delta$, we denote by $h_{(T,S)}$ the function that maps each choice of the \emph{virtual} half-edge labeling surrounding the poles of $T$ to $\yes$ or $\no$, indicating whether such a partial labeling can be completed into a correct complete labeling of $T$. 
More specifically, consider 
\[X=(\fI_1, \fI_2, \ldots,  \fI_k),\]
where $\fI_i$ is a size-$(\Delta - d_i)$ multiset of labels in $\Sigma$, for each $1 \leq i \leq k$. Then $h_{(T,S)}(X) = \yes$ if there is a correct labeling of $T$ such that the $\Delta - d_i$ virtual half-edge labels surrounding $v_i$ are labeled by $\fI_i$,  for each $1 \leq i \leq k$. This definition can be generalized to the case $T$ is already partially labeled in the sense that some of the half-edge labels have been fixed. In this case, $h_{(T,S)}(X) = \no$ whenever $X$ is incompatible with the given partial labeling.

Let $T$ be a tree with poles $S = (v_1, v_2, \ldots, v_k)$ and let $T'$ be another tree with poles $S' = (v_1', v_2', \ldots, v_k')$ such that $\deg(v_i) = \deg(v_i') = d_i$ for each $1 \leq i \leq k$.
Then we write $T_1 \simm T_2$ if $h_{(T,S)} = h_{(T',S')}$.

Given an LCL problem $\Pi$, it is clear that the number of equivalence classes of rooted trees and bipolar trees w.r.t.~$\simm$ is finite.
For a rooted tree $T$,  denote by $\class(T)$ the equivalence class of $T$. For a bipolar tree $H$, denote by $\type(H)$ the equivalence class of $H$.

\paragraph{Subtree Replacement}
The following lemma provides a sufficient condition that the equivalence class of a tree $T$ is invariant of the equivalence class of its subtree $T'$.
We note that a real half edge in $T$ might become virtual in its subtree $T'$.
Consider a vertex $v$ in $T'$ and its neighbor $u$ that is in $T$ but not in $T'$. Then the half edge $(v, \{u,v\})$  is real in $T$ and virtual in $T'$.

\begin{lemma}[Replacing subtrees]\label{lem:replace}
Let $T$ be a tree with poles $S$.
Let $T'$ be a connected subtree of $T$ induced by $U \subseteq V$, where $V$ is the set of vertices in $T$.
We identify a list of designated vertices $S' = (v_1', v_2', \ldots, v_k')$ in $U$ satisfying the following two conditions to be the poles of $T'$.
\begin{itemize}
    \item $S \cap U \subseteq S'$.
    \item Each edge $e=\{u,v\}$ connecting $u \in U$ and $ v \in V \setminus U$ must satisfy $u \in S'$.
\end{itemize}
Let $T''$ be another tree with poles $S'' = (v_1'', v_2'', \ldots, v_k'')$ that is in the same equivalence class as $T'$.
Let $T^\ast$ be the result of replacing $T'$ by $T''$ in $T$ by identifying $v_i' = v_i''$ for each $1 \leq i \leq k$.
Then $T^\ast$ is in the same equivalence class as $T$. 
\end{lemma}
\begin{proof}
To prove the lemma, by symmetry, it suffices to show that starting from any correct labeling $\mathcal{L}$ of $T$, it is possible to find a correct labeling  $\mathcal{L}^\ast$ of $T^\ast$ in such a way that the multiset of the % real half-edge labels and  the multiset of the 
virtual half-edge labels surrounding each pole in $S$ remain the same.

Such a correct labeling $\mathcal{L}^\ast$  of $T^\ast$ is constructed as follows. If $v \in V \setminus U$, then we simply adopt the given labeling $\mathcal{L}$ of $v$ in $T$.
Next, consider the vertices in $U$.
Set $X=(\fI_1, \fI_2, \ldots, \fI_k)$ to be the one compatible with the labeling $\mathcal{L}$ of $T$ restricted to the subtree $T'$ in the sense that $\fI_i$ is the multiset of the virtual half-edge labels surrounding the pole $v_i'$ of $T'$, for each $1 \leq i \leq k$. We must have $h_{T',S'}(X) = \yes$.
Since $T'$ and $T''$ are in the same equivalence class, we have $h_{T'',S''}(X) = \yes$, and so we can find a correct labeling $\mathcal{L}''$ of $T''$ that is also compatible with $X$. Combining this labeling $\mathcal{L}''$ of the vertices $U$ in $T''$ with the  labeling  $\mathcal{L}$ of the vertices  $V \setminus U$ in $T$, we obtain a desired  labeling  $\mathcal{L}^\ast$ of $T^\ast$.

We verify that $\mathcal{L}^\ast$  gives a correct labeling of $T^\ast$. Clearly, the size-$\Delta$ multiset that labels each vertex $v \in V$ is in $\fV$, by the correctness of $\mathcal{L}$  and $\mathcal{L}''$. Consider any edge $e=\{u,v\}$ in  $T^\ast$. Similarly, if $\{u,v\} \subseteq V \setminus U$ or $\{u,v\} \cap (V \setminus U) = \emptyset$, then  the size-$2$ multiset that labels $e$ is in $\fE$, by the correctness of $\mathcal{L}$  and $\mathcal{L}''$. For the case that $e=\{u,v\}$ connects a vertex $u \notin V$ and a vertex $v \in V \setminus U$, we must have $u \in S'$ by the lemma statement. Therefore, the label of the half edge $(u,e)$ is the same in both  $\mathcal{L}$  and $\mathcal{L}^\ast$  by our choice of $\mathcal{L}''$. Thus, the size-$2$ multiset that labels $e$ is in $\fE$, by the correctness of $\mathcal{L}$.

We verify that  the 
virtual half-edge labels surrounding each pole in $S$ are the same in both  $\mathcal{L}$  and $\mathcal{L}^\ast$. Consider a pole $v \in S$. If $v \in V \setminus U$, then the labeling of $v$ is clearly the same in both $\mathcal{L}$  and $\mathcal{L}^\ast$.
If $v \notin V\setminus U$, then the condition $S \cap U \subseteq S'$ in the statement implies that $v \in S'$. In this case, the way we pick $\mathcal{L}''$ ensures that the 
virtual half-edge labels surrounding $v \in S'$ are the same in both  $\mathcal{L}$  and $\mathcal{L}^\ast$.
\end{proof}

In view of the proof of \cref{lem:replace}, as long as the conditions in \cref{lem:replace} are met, we are able to abstract out a subtree by its equivalence class when reasoning about correct labelings of a tree. This observation will be applied repeatedly in the subsequent discussion.

\paragraph{A Pumping Lemma} We will prove a pumping lemma of bipolar trees using \cref{lem:replace}. 
Suppose $T_i$ is a tree with a root $r_i$ for each $1 \leq i \leq k$, then $H=(T_1, T_2, \ldots, T_k)$ denotes the bipolar tree resulting from concatenating the roots $r_1, r_2, \ldots, r_k$ into a path $(r_1, r_2, \ldots, r_k)$ and setting the two poles of $H$ by $s = r_1$ and $t = r_k$.
A simple consequence of \cref{lem:replace} is that $\type(H)$ is determined by  $\class(T_1), \class(T_2), \ldots, \class(T_k)$.
We have the following pumping lemma. 

\begin{lemma}[Pumping lemma]\label{lem:pump}
There exists a finite number $\Lpump > 0$ such that as long as $k \geq \Lpump$, any  bipolar tree $H=(T_1, T_2, \ldots, T_k)$ can be decomposed into $H = X \circ Y \circ Z$ with $0 < |Y| < k$ so that $X \circ Y^i \circ Z$ is in the same equivalence class as $H$ for each $i \geq 0$. 
\end{lemma}
\begin{proof}
Set $\Lpump$ to be the number of equivalence classes for bipolar trees plus one. By the pigeon hole principle, there exist $1 \leq a < b \leq k$ such that $(T_1, T_2, \ldots, T_a)$ and $(T_1, T_2, \ldots, T_b)$ are in the same equivalence class. Set $X = (T_1, T_2, \ldots, T_a)$, $Y = (T_{a+1}, T_{a+2}, \ldots, T_b)$, and $Z = (T_{b+1}, T_{b+2}, \ldots, T_k)$. As we already know that $\type(X) = \type(X\circ Y)$, \cref{lem:replace} implies that 
$\type(X\circ Y) = \type(X^2\circ Y)$ by replacing $X$ by $X\circ Y$ in the bipolar tree $X\circ Y$.
Similarly, $\type(X \circ Y^i)$ is the same for for each $i \geq 0$. Applying \cref{lem:replace} again to replace $X \circ Y$ by $X \circ Y^i$ in $H = X \circ Y \circ Z$, we conclude that  $X \circ Y^i \circ Z$ is in the same equivalence class as $H$ for each $i \geq 0$. 
\end{proof}

%We can prove the following statements.
%\begin{itemize}
%    \item Given a fixed number of poles, the number $x$ of equivalence classes is finite. We will only consider the cases of $x=1$ and $x=2$.
%    \item Write a bipolar tree $T=(T_1, T_2, \ldots, T_k)$ if $T$ is the result of concatenating the roots $v_1, v_2, \ldots, v_k$ of $T_1, T_2, \ldots, T_k$ into a path and setting $s = v_1$ and $t = v_k$ as the two poles of $T$. Then $\type(T)$ is a function of  $\class(T_1), \class(T_2), \ldots, \class(T_k)$.
%    \item Similarly, $\type(T)$ is a function of $\type(T')$ and $\class(T_k)$, where    $T' = (T_1, T_2, \ldots, T_{k-1})$ and $T=(T_1, T_2, \ldots, T_k)$. 
%    \item Therefore, we have the following pumping lemma. There exists a finite number $\Lpump > 0$ such that as long as $k \geq \Lpump$, any  bipolar tree $T=(T_1, T_2, \ldots, T_k)$ can be decomposed into $T = X \circ Y \circ Z$ with $0 < |Y| < k$ so that $X \circ Y^i \circ Z$ is in the same equivalence class as $T$ for each $i \geq 0$.
%    \item Clearly, $\class(T)$ depends only on $\type(T)$ if we set the root $r$ of $T$ to be any one of the two poles $s$ an $t$ of the bipolar tree $T$.
%    \item Let $T'$ be the result of fixing some of the half edge labels surrounding the poles of $T$, then the equivalence class of $T'$ depends only on the above partial labeling and the equivalence class of $T$.
%\end{itemize}

\paragraph{A Procedure for Gluing Trees} Suppose that we have a set of rooted trees $W$. We  devise a procedure that generates a new set of rooted trees by gluing the rooted trees in $W$ together. 
This procedure is parameterized by a labeling function $f$. Consider a bipolar tree \[H = (T_1^l, T_2^l, \ldots, T_{\Lpump}^l, T^m, T_1^r, T_2^r, \ldots, T_{\Lpump}^r)\] 
where $T^m$ is formed by attaching the roots $r_1, r_2, \ldots, r_{\Delta - 2}$ of the rooted trees $T_1^m, T_2^m, \ldots, T_{\Delta -2}^m$ to the root $r^m$ of $T^m$.

The labeling function $f$ assigns the half-edge labels surrounding $r^m$ based on 
\[\type(H^l),  \class(T_{1}^m), \class(T_{2}^m), \ldots, \class(T_{\Delta - 2}^m), \type(H^r)\]
where $H^l = (T_1^l, T_2^l, \ldots, T_{\Lpump}^l)$ and $H^r = (T_1^r, T_2^r, \ldots, T_{\Lpump}^r)$.

We write  $T^{m}_\ast$ to denote the result of applying $f$ to label the root $r^m$ of $T^{m}$ in the bipolar tree $H$, and we write \[H_\ast = (T_1^l, T_2^l, \ldots, T_{\Lpump}^l, T_\ast^m, T_1^r, T_2^r, \ldots, T_{\Lpump}^r)\] to denote the result of applying $f$ to $H$.  We make the following observation.

\begin{lemma}[Property of $f$]\label{lem:labeling_function}
The two equivalence classes
$\class(T_\ast^m)$ and $\type(H_\ast)$ are determined by \[\type(H^l),  \class(T_{1}^m), \class(T_{2}^m), \ldots, \class(T_{\Delta - 2}^m), \type(H^r),\] and the labeling function $f$.
\end{lemma}
\begin{proof}
By \cref{lem:replace}, once the half-edge labelings of the root $r^m$ of $T^{m}$ is fixed, $\class(T_\ast^m)$ is determined by $\class(T_{1}^m), \class(T_{2}^m), \ldots, \class(T_{\Delta - 2}^m)$. 
Therefore, indeed $\class(T_\ast^m)$ is determined by
\[\type(H^l),  \class(T_{1}^m), \class(T_{2}^m), \ldots, \class(T_{\Delta - 2}^m), \type(H^r),\] and the labeling function $f$.
Similarly, applying  \cref{lem:replace} to the decomposition of $H_\ast$ into $H^l$, $T_\ast^m$, and $H^r$, we infer that  $\type(H_\ast)$ depends only on $\type(H^l)$, $\class(T_\ast^m)$, and $\type(H^r)$.
%This is a simple observation. %\cref{lem:replace}.
\end{proof}

The three sets of trees $X_f(W)$, $Y_f(W)$, and $Z_f(W)$ are constructed as follows.

\begin{itemize}
    \item $X_f(W)$ is the set of all rooted trees  resulting from appending $\Delta-2$ arbitrary rooted trees $T_1, T_2, \ldots, T_{\Delta-2}$ in $W$ to a new root vertex $r$.
    \item  $Y_f(W)$ is the set of bipolar trees constructed as follows. For each choice of $2\Lpump +1$ rooted trees $T_1^l, T_2^l, \ldots, T_{\Lpump}^l, T^m, T_1^r, T_2^r, \ldots, T_{\Lpump}^r$ from $X_f(W)$, concatenate them into a bipolar tree 
    \[H = (T_1^l, T_2^l, \ldots, T_{\Lpump}^l, T^m, T_1^r, T_2^r, \ldots, T_{\Lpump}^r),\]  let $H_\ast$ be the result of applying the labeling function $f$ to $H$, and then add $H_\ast$ to $Y_f(W)$.
    \item $Z_f(W)$ is the set of rooted trees  constructed as follows. For each \[H_\ast = (T_1^l, T_2^l, \ldots, T_{\Lpump}^l, T_\ast^m, T_1^r, T_2^r, \ldots, T_{\Lpump}^r) \in Y_f(W),\] add $(T_1^l, T_2^l, \ldots, T_{\Lpump}^l, T_\ast^m)$ to $Z_f(W)$,  where we set the root of $T_1^l$ as the root, and  add $(T_\ast^m, T_1^r, T_2^r, \ldots, T_{\Lpump}^r)$ to $Z_f(W)$, where we set the root  of ${T_{\Lpump}^r}$ as the root. 
\end{itemize}

We write $\class(S) = \bigcup_{T \in S} \{\class(T)\}$ and  $\type(S) = \bigcup_{H \in S} \{\type(H)\}$, and 
we make the following observation.

\begin{lemma}[Property of  $X_f(W)$, $Y_f(W)$, and $Z_f(W)$]\label{lem:sets_of_trees}
The three sets of equivalence classes $\class(X_f(W))$, $\type(Y_f(W))$, and $\class(Z_f(W))$ depend only on $\class(W)$ and the labeling function $f$.
\end{lemma}
\begin{proof}
This is a simple consequence of \cref{lem:replace,lem:labeling_function}.
\end{proof}

\paragraph{A Fixed Point $W^\ast$}
Given a fixed labeling function $f$,
we want to find a set of rooted trees $W^\ast$ that is a \emph{fixed point} for the procedure $Z_f$ in the sense that \[\class(Z_f(W^\ast)) = \class(W^\ast).\] 
To find such a set $W^\ast$, we construct a two-dimensional array of rooted trees $\{W_{i,j}\}$, as follows.

\begin{itemize}
    \item For the base case, $W_{1,1}$ consists of only the one-vertex rooted tree.
    \item Given that $W_{i,j}$ as been constructed, we define $W_{i,j+1} = Z_f(W_{i,j})$.
    \item Given that $W_{i,j}$ for all positive integers $j$ have been constructed,  $W_{i+1,1}$ is defined as follows. Pick $b_i$ as the smallest index such that $\class(W_{i, b_i}) = \class(W_{i, a_i})$ for some $1 \leq a_i < b_i$. By the pigeon hole principle, the index $b_i$ exists, and it is upper bounded by $2^C$, where $C$ is the number of equivalence classes for rooted trees. We set 
    \[W_{i+1,1} = W_{i, a_i} \cup W_{i, a_i + 1} \cup \cdots \cup W_{i, b_i - 1}.\]
\end{itemize}

We show that the  sequence $\class(W_{i,a_i}), \class(W_{i,a_i+1}), \class(W_{i,a_i+2}), \ldots$ is periodic with a period $c_i = b_i - a_i$.

\begin{lemma}\label{lem:W-aux-1}
For any $i \geq 1$ and for any $j \geq a_i$, we have $\class(W_{i,j}) = \class(W_{i,j+c_i})$, where $c_i = b_i - a_i$.
\end{lemma}
\begin{proof}
By \cref{lem:sets_of_trees}, $\class(W_{i,j})$  depends only on $\class(W_{i,j-1})$. Hence the lemma follows from the fact that $\class(W_{i, b_i}) = \class(W_{i, a_i})$.
\end{proof}

Next, we show that $\class(W_{2,1}) \subseteq \class(W_{3, 1}) \subseteq \class(W_{4, 1}) \subseteq  \cdots$.

\begin{lemma}\label{lem:W-aux-2}
For any $i \geq 2$, we have $\class(W_{i,1}) \subseteq \class(W_{i,j})$ for each $j > 1$, and so $\class(W_{i,1}) \subseteq \class(W_{i+1, 1})$.
\end{lemma}
\begin{proof}
Since $W_{i,1} = \bigcup_{a_{i-1} \leq l \leq b_{i-1} - 1} W_{i-1, l}$, we have \[\bigcup_{a_{i-1}+j-1 \leq l \leq b_{i-1} +j} W_{i-1, l} \subseteq W_{i,j}\] according to the procedure of constructing $W_{i,j}$. By \cref{lem:W-aux-1}, \[  \class\left(\bigcup_{a_{i-1} \leq l \leq b_{i-1} - 1} W_{i-1, l}\right) = \class\left(\bigcup_{a_{i-1}+j-1 \leq l \leq b_{i-1} +j} W_{i-1, l}\right)\]  for all $j \geq 1$, and so we have $\class(W_{i,1}) \subseteq \class(W_{i,j})$ for each $j > 1$. The claim $\class(W_{i,1}) \subseteq \class(W_{i+1, 1})$ follows from the fact that $W_{i+1,1} = \bigcup_{a_{i} \leq l \leq b_{i} - 1} W_{i, l}$.
\end{proof}

Set $i^\ast$ to be the smallest index $i \geq 2$ such that $\class(W_{i,1}) = \class(W_{i+1, 1})$. By the pigeon hole principle and \cref{lem:W-aux-2}, the index $i^\ast$ exists, and it is upper bounded by $C$, the number  of equivalence classes for rooted trees. We set \[W^\ast = W_{i^\ast, 1}.\] 

The following lemma shows that $\class(Z_f(W^\ast)) = \class(W^\ast)$, as needed.

\begin{lemma}\label{lem:W-aux-3}
For any $i \geq 2$, if $\class(W_{i,1}) = \class(W_{i+1, 1})$, then $\class(W_{i,j})$ is the same for all $j \geq 1$. 
\end{lemma}
\begin{proof}
By \cref{lem:W-aux-1} and the way we construct $W_{i+1, 1}$, we have \[\class(W_{i, j}) \subseteq \class(W_{i+1, 1}) \ \ \text{for each} \ \ j \geq a_i.\]
By \cref{lem:W-aux-2}, we have  \[\class(W_{i, 1}) \subseteq \class(W_{i, j}) \ \ \text{for each} \ \ j > 1.\]
Therefore, $\class(W_{i,1}) = \class(W_{i+1, 1})$ implies that $\class(W_{i, j}) = \class(W_{i,1})$ for each $j \geq a_i$. Hence we must have $\class(Z_f(W_{i,1})) = \class(W_{i,1})$, and so $\class(W_{i,j})$ is the same for all $j \geq 1$. 
\end{proof}

We remark that simply selecting $W^\ast$ to be any set such that $\class(Z_f(W^\ast)) = \class(W^\ast)$ is not enough for our purpose. As we will later see, it is crucial that the set $W^\ast$ is constructed by iteratively applying the function $Z_f$ and taking the union of previously constructed sets.

%start with $W_0$ containing only the trivial 1-vertex tree, and then we construct an infinite sequence of sets of rooted trees $W_0, W_1, W_2, \ldots$ using the above procedure. Because $\class(W_i)$ depends only on $\class(W_{i-1})$, the infinite sequence $\type(W_0), \class(W_1), \class(W_2), \ldots$ must eventually be periodic, that is, there exist $i$ and $j$ such that $\class(W_{i+kj})$ is the same for all non-negative integers $k$.

%Now, we consider the set $W_0^1 = \bicup_{0 \leq l \leq j-1} W_{i+l}$, and then similarly we find an infinite sequence $W_0^1, W_1^1, W_2^1, \ldots$, detect a loop, and identify a new set of rooted trees $W_0^2$ by taking the union of all trees in a loop. It is straightforward that $\class(W_0^1) \subseteq \class(W_0^2)$. 

%Repeating the above process, we find an infinite sequence of trees $W_0^1, W_0^2, W_0^3, \ldots$. Because $\class(W_0^1) \subseteq \class(W_0^2) \subseteq \class(W_0^3)$, we must reach a fixed point eventually. That is, there is a number $i$ such that  $\class(W_0^{i+j})$ is the same for all non-negative integers $j$. It is clear that setting $W^\ast = W_0^{i}$ satisfies the needed requirement.

%Note that we can assume that $W^\ast$ contains at most one rooted tree from each equivalence class. If there are more than one, then we can remove those extra trees from the set, and the resulting set still satisfies the desired property.

\paragraph{A Sequence of Sets of Trees}
%We are ready to construct a set $\fV'$ of size-$\Delta$ multi-sets of labels in $\Sigma$ that meets the combinatorial condition for an $\ell$-full set.
We define $W_1^\ast = W^\ast$ and $W_i^\ast = Z_f(W_{i-1}^\ast)$ for each $1 <  i \leq L$, where $L$ is some sufficiently large number to be determined.

The way we choose $W^\ast$ guarantees that $\class(W_i^\ast) = \class(W^\ast)$  for all $1 \leq i \leq L$. 
For convenience, we write $X_i^\ast = X_f(W_i^\ast)$ and $Y_i^\ast = Y_f(W_i^\ast)$.
Similarly, $\class(X_i^\ast)$ is the same for all $1 \leq i \leq L$ and $\type(Y_i^\ast)$ is the same for all $1 \leq i \leq L$.

Our analysis will rely on the assumption that all rooted trees in $W^\ast$ admit correct labelings. Whether this is true depends only on $\class(W^\ast)$, which depends only on the labeling function $f$. We say that $f$ is \emph{feasible} if it leads to a set $W^\ast$ where all the rooted trees therein admit correct labelings. The proof that a feasible labeling function $f$ exists is deferred.

The assumption that $f$ is feasible implies that all trees in $W_i^\ast$, $X_i^\ast$, and $Y_i^\ast$, for all $1 \leq i \leq L$, admit correct labelings. All rooted trees in $X_i^\ast$ admit correct labelings because they are subtrees of the rooted trees in $W_{i+1}^\ast$. A correct labeling of any bipolar tree $H \in Y_i^\ast$ can be obtained by combining any correct labelings of the two rooted trees in $W_{i+1}^\ast$ resulting from $H$.

%Similarly, we will start with $W_0 = W^\ast$, and then we find a sequence of sets of rooted trees $W_0, W_1, W_2, \ldots, W_k$ using the previous procedure. Here $k$ is a sufficiently large number.

%Now we make an assumption that the partial labeling of all rooted trees in $W_k$ can be completed into complete legal labeling. Note that whether a rooted tree $T$ admits a legal labeling depends only on $\class(T)$. Because the way we choose $W^\ast$, the set of equivalence classes in each set $W_i$ is the same, for all $0 \leq i \leq k$. 

\paragraph{Layers of Vertices}
We assign a layer number $\lambda(v)$ to each vertex $v$ in a tree  based on the step that $v$ is introduced in the construction \[W_1^\ast \rightarrow W_2^\ast \rightarrow \cdots \rightarrow W_L^\ast.\]

If a vertex $v$ is introduced as the root vertex of a tree in $X_i^\ast = X_f(W_i^\ast)$, then we say that the layer number of $v$ is $\lambda(v) = i \in \{1,2, \ldots, L\}$. A vertex $v$ has  $\lambda(v) = 0$ if it belongs to a tree in $W_1^\ast$.

For any vertex $v$ with $\lambda(v) = i$ in a tree $T \in X_j^\ast$ with $i \leq j$, we write $T_v$ to denote the subtree of $T$ such that $T_v \in X_i^\ast$ where $v$ is the root of $T_v$.

We construct a sequence of sets $R_1, R_2, \ldots, R_L$ as follows.
We go over all rooted trees $T \in X_L^\ast$, all possible correct labeling $\mathcal{L}$ of $T$, and all vertices $v$ in $T$ with $\lambda(v) = i \in \{1,2, \ldots, L\}$. Suppose that the $\Delta-2$ rooted trees in the construction of $T_v \in X_i^\ast = X_f(W_i^\ast)$ are $T_1$, $T_2$, $\ldots$, $T_{\Delta-2}$, and let $r_i$ be the root of $T_i$. Consider the following parameters.
\begin{itemize}
    \item $c_i = \class(T_i)$.
    \item $\ta_i$ is the real half-edge label of $v$ in $T_v$ for the edge $\{v, r_i\}$, under the correct labeling $\mathcal{L}$ of $T$ restricted to $T_v$.
    \item $\fI$ is the size-$2$ multiset of virtual half-edge labels of  $v$ in $T_v$, under the correct labeling $\mathcal{L}$ of $T$ restricted to $T_v$.
\end{itemize}
Then we add 
$(c_1, c_2, \ldots, c_{\Delta-2}, \ta_1, \ta_2, \ldots, \ta_{\Delta-2}, \fI)$ to $R_i$.

%$\fR$ is the size-$(\Delta-2)$ multiset of labels in $\Sigma$ corresponding to the real half-edge labels of $v$ in $T_v$ under $\mathcal{L}$, and $\fI$ is the size-$2$ multiset of labels in $\Sigma$ corresponding to the virtual half-edge labels of $v$ in $T_v$ under $\mathcal{L}$. Note that the degree of $v$ in $T_v$ must be $\Delta - 2$ in view of the definition of $Z_f$.

\begin{lemma}\label{lem:Rsets1}
For each $1 \leq i < L$, $R_i$ is determined by $R_{i+1}$.
\end{lemma}
\begin{proof}
We consider the following alternative way of constructing $R_i$ from $R_{i+1}$.
Each rooted tree $T'$ in $W_{i+1}^\ast = Z_f(W_i^\ast)$ can be described as follows.
\begin{itemize}
    \item Start with a path $(r_1, r_2, \ldots, r_{\Lpump + 1})$, where $r_1$ is the root of $T'$.
    \item For each $1 \leq j \leq \Lpump+1$, append $\Delta-2$ rooted trees $T_{j,1}, T_{j,2}, \ldots, T_{j, \Delta-2} \in W_i^\ast$ to $r_j$.
    \item Assign the labels to the half edges surrounding $r_{\Lpump + 1}$ according to the labeling function $f$.
\end{itemize}

 Now, consider the function $\phi$ that maps each equivalence class $c$ for rooted trees to a subset of $\Sigma$ defined as follows:
$\ta \in \phi(c)$ if there exist \[(c_1, c_2, \ldots, c_{\Delta-2}, \ta_1, \ta_2, \ldots, \ta_{\Delta-2}, \fI) \in R_{i+1}\] and $1 \leq j \leq \Delta -2$ such that $c = c_j$ and $\ta = \ta_j$.

We go over all possible $T' \in W_{i+1}^\ast = Z_f(W_i^\ast)$.
Note that the root $r$ of $T'$ has exactly one virtual half edge. For each $\tb \in \Sigma$ such that $\{\ta, \tb\} \in \fE$ for some $\ta \in \phi(\class(T'))$, we go over all  possible correct labelings $\mathcal{L}$ of $T'$ where the virtual half edge of $r$ is labeled $\tb$. For each $1 \leq j \leq \Lpump+1$, consider the following parameters.
\begin{itemize}
    \item $c_l = \class(T_{j,l})$.
    \item $\ta_l$ is the half-edge label of $r_j$ for the edge $\{r_j, r_{j,l}\}$.
    \item $\fI$ is the size-$2$ multiset of the remaining two half-edge labels of  $v$.
\end{itemize}
Then we add 
$(c_1, c_2, \ldots, c_{\Delta-2}, \ta_1, \ta_2, \ldots, \ta_{\Delta-2}, \fI)$ to $R_i$. 

This construction of $R_i$ is equivalent to the original construction of $R_i$ because a correct labeling of $T' \in W_i^\ast$ can be extended to a correct labeling of a tree $T \in X_L^\ast$ that contains $T'$ as a subtree if and only if the virtual half-edge label of the root $r$ of $T'$ is $\tb \in \Sigma$ such that $\{\ta, \tb\} \in \fE$ for some $\ta \in \phi(\class(T'))$.

It is clear that this construction of $R_i$ only depends on $\class(W_i^\ast)$, the labeling function $f$, and the function $\phi$, which depends only on $R_{i+1}$. Since the labeling function $f$ is fixed and $\class(W_i^\ast)$ is the same for all $i$, we conclude that $R_{i}$ depends only on $R_{i+1}$.
%This follows from the fact that the way $W_{i+1}^\ast$ is constructed from  $W_{i}^\ast$ is the same  $W_{i+1}^\ast = Z_f(W_{i}^\ast)$ for all $i$ and the fact that we can abstract out subtrees by their equivalence classes.
%More specifically, given that $R_{i+1}$ has been defined, the set $R_i$ can be alternatively defined as follows. Add $(c, \fR, \fI)$ to $R_i$ if there exist a rooted tree $T \in W_{i}^\ast$ with $\class(T) = c$ and a rooted tree $T' \in  W_{i+1}^\ast$ in such a way that the following conditions are met.
%\begin{itemize}
%    \item 
%The construction of  $T'$ in  $W_{i+1}^\ast = Z_f(W_{i}^\ast)$ involves $T$.
%\item There exists a correct labeling of $T'$ such that $(c', \fR', \fI') \in R_{i+1}$ with $c' = \class(T')$ and $(\fR', \fI')$ is compatible with the labeling of the root of $T'$, and $(\fR, \fI)$ compatible with the labeling of the root of the subtree $T$.
%\end{itemize}
%Since $\class(W_i^\ast)$ is the same for all $i$, the above construction of $R_i$  from $R_{i+1}$ depends only on $R_{i+1}$.
\end{proof}

\begin{lemma}\label{lem:Rsets2}
We have $R_1 \subseteq R_2 \subseteq \cdots \subseteq R_L$.
\end{lemma}
\begin{proof}
For the base case, we show that $R_{L-1} \subseteq R_L$. In fact, our proof will show that $R_{i} \subseteq R_L$ for each $1 \leq i < L$. Consider any 
\[(c_1, c_2, \ldots, c_{\Delta-2}, \ta_1, \ta_2, \ldots, \ta_{\Delta-2}, \fI) \in R_i.\] 
Then there is a rooted tree $T \in X_i^\ast$ that is formed by attaching $\Delta-2$ rooted trees of equivalence classes $c_1, c_2, \ldots, c_{\Delta-2}$ to the root vertex so that if we label the half edges surrounding the root vertex according to  $\ta_1, \ta_2, \ldots, \ta_{\Delta-2}$, and  $\fI$, then this partial labeling can be completed into a correct labeling of $T$.

Because $\class(W_i^\ast) = \class(W_L^\ast)$, there is also a rooted tree $T' \in X_L^\ast$ that is formed by attaching $\Delta-2$ rooted trees of equivalence classes $c_1, c_2, \ldots, c_{\Delta-2}$ to the root vertex. Therefore, if we label the root vertex of $T'$ in the same way as we do for $T$, then this partial labeling can also be completed into a correct labeling of $T'$. Hence we must have 
\[(c_1, c_2, \ldots, c_{\Delta-2}, \ta_1, \ta_2, \ldots, \ta_{\Delta-2}, \fI) \in R_L.\] 

%Observe that $(c, \fR, \fI) \in R_L$ if and only if the function $h$ associated with the equivalence class $c$ satisfies $h(X)=\yes$ for $X = (\fR, \fI)$. That is,  any rooted tree with $\class(T) = c$ admits a correct labeling where the vertex configuration of its root $r$ is compatible with $X = (\fR, \fI)$. Therefore,  $(c, \fR, \fI) \in R_L$ is a necessary condition for   $(c, \fR, \fI) \in R_i$, for any $1 \leq i \leq L-1$.

Now, suppose that we already have $R_i \subseteq R_{i+1}$ for some $1 < i < L$. We will show that  $R_{i-1} \subseteq R_{i}$. 
Denote by $\phi_{i}$ and $\phi_{i+1}$ the function $\phi$ in \cref{lem:Rsets1} constructed from  $R_{i}$ and $R_{i+1}$. 
We have $\phi_i(c) \subseteq \phi_{i+1}(c)$ for each equivalence class $c$, because $R_i \subseteq R_{i+1}$. Therefore, in view of the alternative construction described in the proof of \cref{lem:Rsets1}, we have $R_{i-1} \subseteq R_{i}$.
\end{proof}

\paragraph{The Set of Vertex Configurations $\fV'$}
By \cref{lem:Rsets1,lem:Rsets2}, if we pick $L$ to be sufficient large, we can have  $R_1 = R_2 = R_3$.
More specifically, if we pick 
\[L \geq C^{\Delta-2} \cdot \binom{|\Sigma|+1}{|\Sigma|-1} + 3,\]
then there exists an index $3 \leq i \leq L$ such that $R_i = R_{i-1}$, implying that $R_1 = R_2 = R_3$. Here $C$ is the number of equivalence classes for rooted trees and $\binom{|\Sigma|+1}{|\Sigma|-1}$ is the number of size-$2$ multisets of elements from $\Sigma$.

%Observe that we must have $R_1 \subseteq R_2 \subseteq \cdots \subseteq R_k$. If we pick $k$ to be sufficiently large, then there exists a number $i^\ast > 0$ such that $R_1 = R_2 = \cdots = R_{i^\ast}$, since $R_1 \neq \emptyset$.  This implies that it is possible to find a legal labeling of each $T \in W_k$ in such a way that for each $v$ in $T$ with label $i > 0$, the label $c$ of $v$ in $T$ satisfies that $(\class(T_v),c) \in R_1$.

The set $\fV'$ is defined by including all size-$\Delta$ multisets $\{\ta_1, \ta_2, \ldots, \ta_{\Delta-2}\} \cup \fI$ such that  \[(c_1, c_2, \ldots, c_{\Delta-2}, \ta_1, \ta_2, \ldots, \ta_{\Delta-2}, \fI) \in R_{1}\] for some $c_1, c_2, \ldots, c_{\Delta-2}$.

%included in the set $R_1$. Note that for any tree $T \in W_k$, it is possible to find a legal labeling of $T$ that uses only the labels in $S$ for all vertices with label $i > 0$. This is true independent of $k$. Also, for each level $0 < i \leq k$ and for each $c \in S$, it is possible to find a tree in $W_k$ with a legal labeling that uses $c$ at least once for some level-$i$ vertex.

\paragraph{The Set $\fV'$  is $\ell$-full}
To show that the set $\fV'$  is $\ell$-full, we  consider the subset $\fV^\ast \subseteq \fV'$ defined by the set of vertex configurations used by the labeling function $f$ in the construction of $Y_f(W_i^\ast)$. The definition of $\fV^\ast$ is invariant of $i$ as $\class(W_i^\ast)$ is the same for all $i$. Clearly, for each $x \in \fV^\ast$, there is a rooted tree $T \in W_2^\ast$ where the root of a subtree $T' \in X_1^\ast$ of $T$ has its size-$\Delta$ multiset of half-edge labels fixed to be $x$ by $f$, and so  $\fV^\ast \subseteq \fV'$.

For notational simplicity, we write $x \overset{k}{\leftrightarrow} x'$ if there is a correct labeling of a $k$-vertex path $(v_1, v_2, \ldots, v_k)$ using only vertex configurations in $\fV'$ so that the vertex configuration of $v_1$ is $x$ and the vertex configuration of $v_k$ is $x'$. If it is further required that the half-edge label of $v_1$ for the edge $\{v_1, v_2\}$ is $\ta \in x$, then we write $(x,\ta) \overset{k}{\leftrightarrow} x'$. The notation $(x,\ta) \overset{k}{\leftrightarrow} (x', \ta')$ is defined similarly.

\begin{lemma}\label{lem:find_path_1}
For any $x \in \fV' \setminus \fV^\ast$ and $\ta \in x$, there exist a vertex configuration $x' \in \fV^\ast$ and a number $2 \leq k \leq 2 \Lpump + 1$ such that $(x,\ta) \overset{k}{\leftrightarrow} x'$.
\end{lemma}
\begin{proof}
Let $T \in W_L^\ast$ be chosen so that there is a correct labeling $\mathcal{L}$ where $x$ is a vertex configuration of some vertex $v$ with $\lambda(v) = 2$.

To prove the lemma, it suffices to show that for each of the $\Delta$ neighbors $u$ of $v$, it is possible to find a path $P = (v,u, \ldots, w)$ meeting the following conditions.
\begin{itemize}
    \item $w$ is a vertex whose half-edge labels have been fixed by $f$. This ensures that the vertex configuration of $w$ is in $\fV^\ast$.
    \item All vertices in $P$ are within layers 1,2, and 3. This ensures that the vertex configuration of all vertices in $P$ are in $\fV'$.
    \item The number of vertices $k$ in $P$ satisfies $2 \leq k \leq 2 \Lpump + 1$. 
\end{itemize}

We divide the proof into three cases. Refer to \cref{fig:ell_full_construction1} for an illustration, where squares are vertices whose half-edge labels have been fixed by $f$.

\paragraph{Case 1} Consider the subtree $T_v \in X_2^\ast$ of $T$ whose root is $v$. In view of the construction of the set $X_f(W_2^\ast)$, $v$ has $\Delta-2$ children $u_1, u_2, \ldots, u_{\Delta-2}$ in $T_v$, where the subtree $T_i$ rooted at $u_i$ is a rooted tree in $W_2^\ast$. 

For each $1 \leq i \leq \Delta-2$, according to the structure of the trees in the set $W_2^\ast = Z_f(W_1^\ast)$, there is a path $(u_i = w_1, w_2, \ldots, w_{\Lpump+1})$ in $T_i$ containing only layer-1 vertices, where the half-edge labels of $w_{\Lpump+1}$ have been fixed by $f$. Hence $P = (v, u_i = w_1, w_2, \ldots, w_{\Lpump+1})$ is a desired path with $k = \Lpump + 2$ vertices.
    
    %Each tree $T_i$ can be viewed as a bipolar tree $(T_1', T_2', \ldots, T_{\Lpump+1}')$, where $u_i$ is the root of $T_1'$. The roots $r_1, r_2, \ldots, r_{\Lpump + 1}$ of these $\Lpump +1$ rooted trees all have $\lambda(r_i) = 1$. Moreover, the half-edge labels of $r_{\Lpump + 1}$ have been fixed by $f$. 
    
    %Therefore,  setting $w = r_{\Lpump + 1}$, we can find a desired path $P = (v,u, \ldots, w)$ for each $u \in \{u_1, u_2, \ldots, u_{\Delta-2}\}$. The number of vertices in $P$ is $\Lpump + 2$.
    
\paragraph{Case 2}  Consider the subtree $T' \in W_3^\ast = Z_f(W_2^\ast)$ that contains $v$ in $T$.
Similarly, according to the structure of the trees in the set $W_3^\ast = Z_f(W_2^\ast)$, there is a path $(r = w_1', w_2', \ldots, w_{\Lpump+1}')$ in $T'$ containing only layer-2 vertices so that $r$ is the root of $T'$, $v = w_i'$ for some $1 \leq i' \leq \Lpump+1$, and the  half-edge labels of $w_{\Lpump+1}'$ have been fixed by $f$.
Since  $x \in \fV' \setminus \fV^\ast$, we have $v \neq w_{\Lpump+1}'$. Hence $P = (v = w_i', w_{i+1}', \ldots, w_{\Lpump+1}')$ is a desired path with $2 \leq k \leq \Lpump + 1$ vertices.

%so that if we view $T'$ as a bipolar tree $(T_1', T_2', \ldots, T_{\Lpump+1}')$, then $T_v = T_i'$ for some $i$. Again, we write $r_i$ to denote the root of $T_i'$, and we assume that $r_1$ is the root of $T'$ and so the half-edge labels of $r_{\Lpump+1}$ have been fixed by $f$. By the assumption $x \in \fV' \setminus \fV^\ast$, we must have $v \neq r_{\Lpump+1}$. 

  %  Let $v = r_i$. Then the path $P = (r_i, r_{i+1}, \ldots, r_{\Lpump+1})$ is a path satisfying the needed requirement, as we observe that $\lambda(r_i) = 2$  for all $1 \leq i \leq \Lpump+1$, and the number $k$ of vertices in $P$ satisfies $2 \leq k \leq \Lpump + 1$.

\paragraph{Case 3} There is only one remaining neighbor of $v$ to consider. In view of the construction of $X_f(W_3^\ast)$, there is a layer-3 vertex $v'$ adjacent to the vertex $r$, the root of the tree $T' \in W_3^\ast$ considered in the previous case. If the half-edge labels of $v'$ have been fixed by $f$, then $P = (v = w_i', w_{i-1}', \ldots, w_{1}'=r, v')$ is a desired path.
Otherwise, similar to the analysis in the previous case, we can find a path $P' = (v', \ldots, w)$ connecting $v'$ to a vertex $w$ whose half-edge labels  have been fixed by $f$. All vertices in $P'$ are of layer-3, and the number of vertices in $P'$ is within $[2, \Lpump + 1]$. 
Combining $P'$ with the path $(v = w_i', w_{i-1}', \ldots, w_{1}'=r, v')$, we obtain the desired path $P$ whose number of vertices satisfies $2 \leq k  \leq  2\Lpump + 1$.
%That is, the neighbor of $v = r_i$ in the direction towards the root $r'$ of $T'$ considered in the previous case. In view of the construction of $Z_f(W_2^\ast)$, $r'$ has a neighbor $r''$ in $T$ that is the root of a tree $T'' \in Z_f(W_2^\ast)$. Note that $\lambda(r'') = 3$, and $v$ is connected to $r''$ via the path $(v = r_i, r_{i-1}, \ldots r_1, r'')$. The number of vertices in this path is $i+1 \in [2, \Lpump + 1]$. 
    %Similarly, by considering a tree in $Y_f(W_3^\ast)$, we can show that there is a path $(r'', \ldots, w)$ for some $w$ whose half-edge labels have been fixed by $f$. Moreover, all the vertices in this path  have  $\lambda$-label equals $3$, and the number of vertices in this path is within $[1, \Lpump + 1]$.
    %Combining these two paths, we obtain a desired path $P$ whose number of vertices is in the range $[2, 2\Lpump + 1]$.
\end{proof}

\begin{figure}[h!]
    \centering
    \includegraphics[width= .75\textwidth]{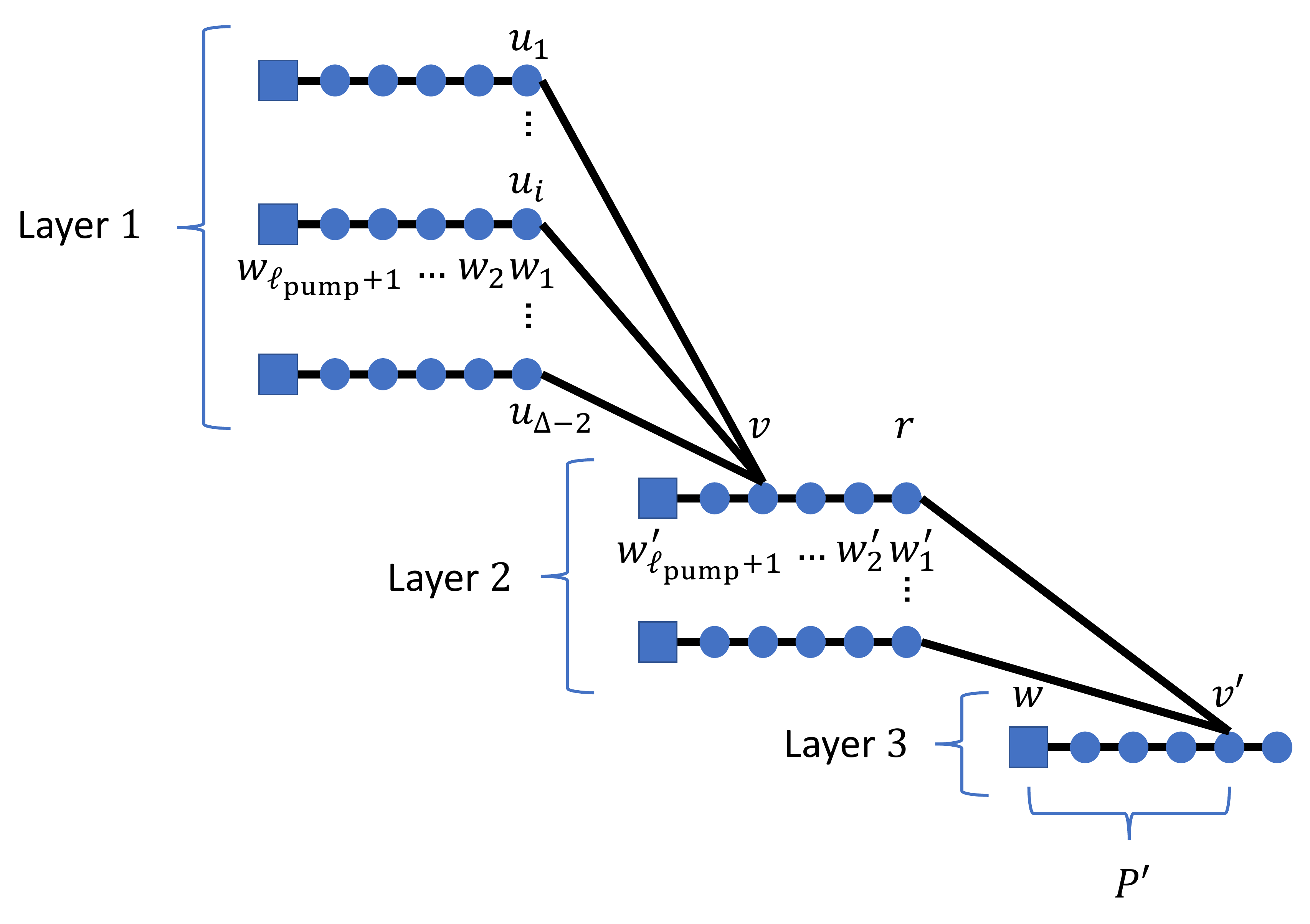}
    \caption{An illustration for the proof of \cref{lem:find_path_1}.}
    \label{fig:ell_full_construction1}
\end{figure}

For \cref{lem:find_path_2}, note that if $\ta$ appears more than once in the multiset $x$, then we still have $\ta \in x \setminus \{\ta\}$.

\begin{lemma}\label{lem:find_path_2}
For any $\ta \in x \in \fV^\ast$, $\ta' \in x' \in \fV^\ast$, and $0 \leq t \leq \Lpump - 1$, we have  $(x,\tb) \overset{k}{\leftrightarrow} (x', \tb')$ for some $\tb \in x \setminus \{\ta\}$ and $\tb' \in x' \setminus \{\ta'\}$
 with $k = 2 \Lpump + 4 + t$.
%we can find a correct labeling of a $k$-vertex path $(v_1, v_2, \ldots, v_k)$  with $k = 2 \Lpump + 3 + t$ using only vertex configurations in $\fV'$ so that the vertex configuration of $v_1$ is $x_1$, the vertex configuration of $v_k$ is $x_2$, $\ta_1$ is used in a virtual half-edge label for $v_1$, and $\ta_2$ is used in a virtual half-edge label for $v_k$.
\end{lemma}
\begin{proof}
For any $\ta \in x \in \fV^\ast$, we can find a rooted tree $T_{x, \ta} \in W_2^\ast$ such that the path $(v_1$, $v_2$, $\ldots$, $v_{\Lpump +1})$ of the layer-1 vertices in  $T_{x, \ta}$ where $v_1 = r$ is the root of   $T_{x, \ta}$ satisfies the property that the half-edge labels of $v_{\Lpump +1}$ have been fixed to $x$ by $f$ where the half-edge label for $\{v_{\Lpump}, v_{\Lpump+1}\}$ is in $x \setminus \{\ta\}$.

Now, observe that for any choice of $(\Delta-2)(\Lpump+1)$ rooted trees \[\{T_{i,j}\}_{1 \leq i \leq \Lpump+1, 1 \leq j \leq \Delta-2}\] in $W_2^\ast$, there is a rooted tree $T' \in W_3^\ast = Z_f(W_2^\ast)$  formed by appending $T_{i, 1}$, $T_{i,2}$, $\ldots$, $T_{i, \Delta-2}$ to $u_i$ in the path $(u_1$, $u_2$, $\ldots$, $u_{\Lpump+1})$ and fixing the half-edge labeling of $u_{\Lpump+1}$ by $f$. All vertices in $(u_1, u_2, \ldots, u_{\Lpump+1})$ are of layer-2.

We choose any $T' \in W_2^\ast$ with $T_{i,j} = T_{x, \ta}$ and $T_{i',j'} = T_{x', \ta'}$ such that $i' - i = t + 1$.
The possible range of $t$ is $[0, \Lpump - 1]$.
Consider any $T \in W_L^\ast$ that contains $T'$ as its subtree, and consider any correct labeling of $T$.
Then the path $P$ resulting from concatenating the path $(v_{\Lpump +1}, v_{\Lpump}, \ldots, v_1)$ in $T_{x, \ta}$, the path $(u_i, u_{i+1}, \ldots, u_{i'})$, and the path $(v_1', v_2', \ldots, v_{\Lpump +1}')$ in $T_{x', \ta'}$
shows that $(x,\tb) \overset{k}{\leftrightarrow} (x', \tb')$ 
for $k = (\Lpump+1) + (t+2) + (\Lpump+1) = 2\Lpump + 4 + t$. See \cref{fig:ell_full_construction2} for an illustration.

For the rest of the proof, we show that the desired rooted tree $T_{x, \ta} \in W_2^\ast$ exists for any $\ta \in x \in \fV^\ast$.
For any $x \in \fV^\ast$, we can find a bipolar tree
\[H_\ast = (T_1^l, T_2^l, \ldots, T_{\Lpump}^l, T_\ast^m, T_1^r, T_2^r, \ldots, T_{\Lpump}^r) \in Y_1^\ast = Y_f(W_1^\ast)\]
such that the vertex configuration of the root $r^m$ of  $T_\ast^m$ is fixed to be $x$ by the labeling function $f$.
Then, for any $\ta \in x$, at least one of the two rooted trees in $W_2^\ast = Z_f(W_1^\ast)$ resulting from cutting $H_\ast$ satisfies the desired requirement. 
%we can find a bipolar tree $H_1 \in X_f(W_1^\ast)$ where the root of the middle tree has its half-edge labels fixed to be $x_1$ by $f$. Denote by $v_1$ this vertex. Recall the construction of $Y_f(W_1^\ast)$ from $X_f(W_1^\ast)$, $H_1$ leads to two rooted trees in $Y_f(W_1^\ast)$, and there is at least one $T_1'$  of them where $\ta_1$ is used in a virtual half-edge label of $v_1$.
%We find a tree $T_2'$ similarly. Now, in view of the construction of $Y_f(W_2^\ast)$, we can find a rooted tree in  $Y_f(W_2^\ast)$ such that if we view it as a bipolar tree, $T_1'$ and $T_2'$ are two rooted trees hanging at the $s$-$t$ path, and the positions of $T_1'$ and $T_2'$ can be arbitrary.  
%Denote by $r_1$ and $r_2$ the roots of $T_1'$ and $T_2'$. Then we can find a desired path $P = (v_1, \ldots, r_1, \ldots, r_2, \ldots, v_2)$, where the number of vertices  ($\lambda$-labels are 1) in $(v_1, \ldots, r_1)$ is $\Lpump +1$, the number of vertices ($\lambda$-labels are 2) in between $r_1$ and $r_2$ can be any number in $[1, \Lpump + 1]$, and the number of vertices ($\lambda$-labels are 1) in $(v_2, \ldots, r_2)$ is $\Lpump +1$. Therefore, the number of vertices in $P$ can be $2 \Lpump + 3 + t$ for any $0 \leq t \leq \Lpump$.
\end{proof}

\begin{figure}[h!]
    \centering
    \includegraphics[width= .65\textwidth]{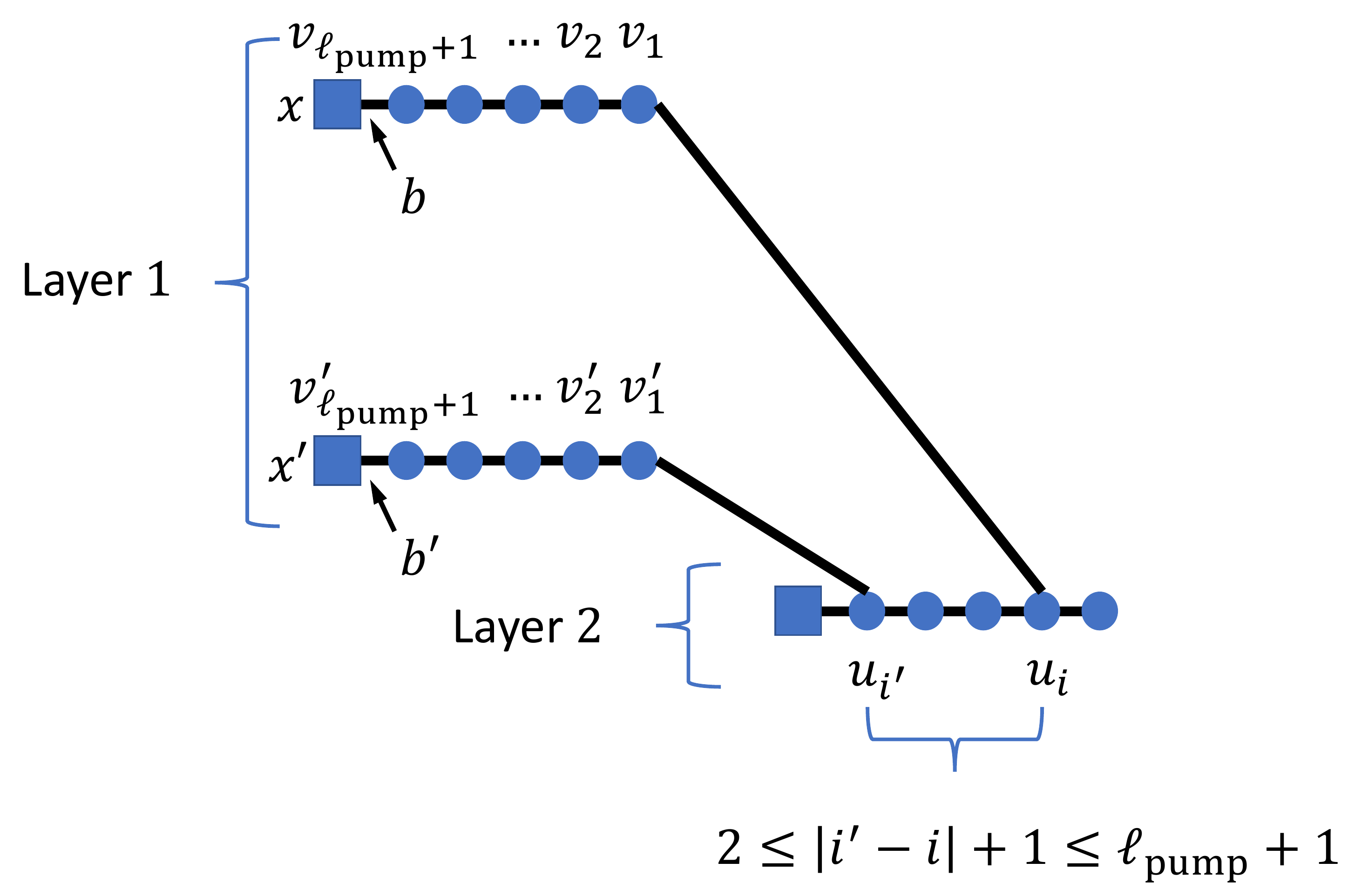}
    \caption{An illustration for the proof of \cref{lem:find_path_2}.}
    \label{fig:ell_full_construction2}
\end{figure} 

In the subsequent discussion, the length of a path refers to the number of edges in a path. In particular, the length of a $k$-vertex path is $k-1$.

The following lemma is proved by iteratively applying \cref{lem:find_path_2} via intermediate vertex configurations $\tilde{x} \in \fV^\ast$. 

\begin{lemma}\label{lem:find_path_3}
For any $\ta \in x \in \fV^\ast$ and $\ta' \in x' \in \fV^\ast$,  we have  $(x,\tb) \overset{k}{\leftrightarrow} (x', \tb')$ for some $\tb \in x \setminus \{\ta\}$ and $\tb' \in x' \setminus \{\ta'\}$
 for all $k \geq 6 \Lpump + 10$.
\end{lemma}
\begin{proof}
By applying \cref{lem:find_path_2} for three times,
we infer that \cref{lem:find_path_2}  also works for any path length $k - 1 = (k_1 - 1) + (k_2 - 1) + (k_3 - 1)$, where $k_i - 1 = 2 \Lpump + 3 + t_i$ for $0 \leq t_i \leq \Lpump-1$. 

Specifically, we first apply \cref{lem:find_path_2} to find a path  realizing $(x,\tb) \overset{k_1}{\leftrightarrow} \tilde{x}$ for some  $\tilde{x} \in \fV^\ast$, and for some $\tb \in x \setminus \{\ta\}$.
Let $\tilde{\ta} \in \tilde{x}$ be the real half-edge label used in the last vertex of the path. 
We extend the length of this path by $k_2 - 1$ by attaching to it the path realizing $(\tilde{x},\tilde{\tb}) \overset{k_2}{\leftrightarrow} \tilde{x}$, for some $\tilde{\tb} \in \tilde{x} \setminus \{\tilde{\ta}\}$.
Finally, let $\tilde{\tc} \in \tilde{x}$ be the real half-edge label used in the last vertex of the current path. 
We extend the length of the current path by $k_3$ by attaching to it the path realizing $(\tilde{x},\tilde{\td}) \overset{k_3}{\leftrightarrow} (x', \tb')$, for some $\tilde{\td} \in \tilde{x} \setminus \{\tilde{\tx}\}$, and for some $\tb' \in x' \setminus \{\ta'\}$.
The resulting path realizes  $(x,\tb) \overset{k}{\leftrightarrow} (x', \tb')$ for some $\tb \in x \setminus \{\ta\}$ and $\tb' \in x' \setminus \{\ta'\}$.

Therefore,  \cref{lem:find_path_2} also works with path length of the form $k-1 = 6\Lpump + 9  + t$, for any $t \in [0, 3\Lpump - 3]$. 
%More specifically, we pick any $x' \in \fV^\ast$ and $x'' \in \fV^\ast$, and then we connect $x_1$ and $x_2$ by connecting $x_1$ and $x'$, connecting $x'$ and $x''$, and connecting $x''$ and $x_2$, using \cref{lem:find_path_2}.

Any $k-1 \geq 6 \Lpump + 9$ can be written as $k-1 = b(2\Lpump +3) + (6\Lpump + 9 + t)$, for some $t \in [0, 3\Lpump - 3]$ and some integer $b \geq 0$.
Therefore,  similar to the above, we can find a path showing $(x,\tb) \overset{k}{\leftrightarrow} (x', \tb')$ by first   applying \cref{lem:find_path_2} with  path length
$2\Lpump +3$ for $b$ times, and then applying the above variant of \cref{lem:find_path_2} to extend the path length by $6\Lpump + 9  + t$.
\end{proof}

We show that \cref{lem:find_path_1,lem:find_path_3} imply that $\fV'$  is $\ell$-full for some $\ell$.  

\begin{lemma}[$\fV'$  is $\ell$-full] \label{lem:find_path_4}
The set $\fV'$  is $\ell$-full for $\ell = 10 \Lpump + 10$. 
\end{lemma}
\begin{proof}
We show that for any target path length $k-1 \geq 10 \Lpump + 9$, and for any $\ta \in x \in \fV^\ast$ and $\ta' \in x' \in \fV^\ast$, we have $(x,\ta) \overset{k}{\leftrightarrow} (x',\ta')$.

By \cref{lem:find_path_1}, there exists a vertex configuration $\tilde{x} \in \fV^\ast$ so that we can find a path $P$ realizing $(x,\ta) \overset{\ell_1}{\leftrightarrow} \tilde{x}$ for some path length  $1 \leq (\ell_1-1) \leq 2 \Lpump$.
Similarly, there exists a vertex configuration $\tilde{x}' \in \fV^\ast$ so that we can find a path $P'$ realizing $(x',\ta') \overset{\ell_2}{\leftrightarrow} \tilde{x}'$ for some path length   $1 \leq  (\ell_2-1) \leq 2 \Lpump$.

Let $\tilde{\ta}$ be the real half-edge label for $\tilde{x}$ in $P$, and let $\tilde{\ta}'$ be the real half-edge label for $\tilde{x}'$ in $P'$.
We apply \cref{lem:find_path_3} to find a path $\tilde{P}$ realizing  $(x,\tb) \overset{\tilde{\ell}}{\leftrightarrow} (x', \tb')$ for some $\tb \in \tilde{x} \setminus \{\tilde{\ta}\}$ and $\tb' \in \tilde{x}' \setminus \{\tilde{\ta}'\}$ with path length  
\[\tilde{\ell}-1 = (k-1) - (\ell_1-1) - (\ell_2-1) \geq  6 \Lpump + 9.\] 
The path formed by concatenating $P_1$, $\tilde{P}$, and $P_2$ shows that $(x,\ta) \overset{k}{\leftrightarrow} (x',\ta')$.
\end{proof}

\paragraph{A Feasible Labeling Function $f$ exists} We show that a feasible labeling function $f$ exists given that $\Pi$ can be solved in $\local(n^{o(1)})$ rounds. We will construct a labeling function $f$ in such a way  each equivalence class in $\class(W^\ast)$ contains only rooted trees that admit legal labeling. That is, for each  $c \in \class(W^\ast)$, the mapping $h$ associated with $c$ satisfies $h(X) = \yes$ for some $X$.

We will consider a series of modifications in the construction \[W  \rightarrow  X_f(W) \rightarrow Y_f(W) \rightarrow Z_f(W)\] that do not alter the sets of equivalence classes in these sets, conditioning on the assumption that all trees that we have processed admit correct labelings. That is, whether $f$ is feasible is invariant of the modifications.

\paragraph{Applying the Pumping Lemma} Let $w > 0$ be some target length for the pumping lemma. In the construction of $Y_f(W)$, when we process a bipolar tree 
\[H = (T_1^l, T_2^l, \ldots, T_{\Lpump}^l, T^m, T_1^r, T_2^r, \ldots, T_{\Lpump}^r),\] we apply  \cref{lem:pump} to the two subtrees $H^l = (T_1^l, T_2^l, \ldots, T_{\Lpump}^l)$ and $H^r = (T_1^r, T_2^r, \ldots, T_{\Lpump}^r)$ to obtain two new bipolar trees $H_+^l$ and $H_+^r$.
The $s$-$t$ path in the new trees $H_+^l$ and $H_+^r$ contains $w + x$ vertices, for some $0 \leq x < \Lpump$. The equivalence classes do not change, that is, $\type(H^l) = \type(H_+^l)$ and $\type(H^r) = \type(H_+^r)$.

We replace $H^l$ by $H_+^l$ and replace $H^r$ by $H_+^r$ in the bipolar tree $H$.
Recall that the outcome of applying the labeling function $f$ to the root  $r$ of the rooted tree $T_{\Lpump+1}$ depends only on \[\type(H^l),  \class(T_{1}^m), \class(T_{2}^m), \ldots, \class(T_{\Delta - 2}^m), \type(H^r),\] so applying the pumping lemma to $H^l$ and $H^r$ during the construction of $Y_f(W)$ does not alter $\class(T_\ast^m)$ and  $\type(H^\ast)$ for the resulting bipolar tree $H^\ast$  and its middle rooted tree $T_\ast^m$, by \cref{lem:labeling_function}.

\paragraph{Reusing Previous Trees} During the construction of the fixed point $W^\ast$, we remember all bipolar trees to which we have applied the feasible function $f$.

During the construction of $Y_f(W)$. Suppose that we are about to process a bipolar tree  $H$, and there is already some  other bipolar tree $\tilde{H}$  to which we have applied $f$ before so that 
$\type(H^l)=\type(\tilde{H}^l)$,  $\class(T_{i}^m) = \class(\tilde{T}_{i}^m)$ for each $1 \leq i \leq \Delta-2$, and $\type(H^r)=\type(\tilde{H}^r)$.
Then we replace $H$ by $\tilde{H}$, and then we process $\tilde{H}$ instead.

%Again, 
By \cref{lem:labeling_function}, this modification does not alter $\class(T_\ast^m)$ and  $\type(H^\ast)$ for the resulting bipolar tree $H^\ast$.

\paragraph{Not Cutting Bipolar Trees} We consider the following different construction of $Z_f(W)$ from $Y_f(W)$. For each $H^\ast \in Y_f(W)$, we simply add two copies of $H^\ast$ to $Z_f(W)$, one of them has $r = s$ and the other one has $r = t$.
That is, we do not cut the bipolar tree $H^\ast$, as in the original construction of $Z_f(W)$.

In general, this modification might alter $\class(Z_f(W))$.
However, it is not hard to see that if all  trees in $Y_f(W)$ already admit correct labelings, then  $\class(Z_f(W))$ does not alter after the modification.

Specifically, let $T$ be a rooted tree in  $Z_f(W)$ with the above modification. 
Then $T$ is identical to a bipolar tree  $H^\ast=(T_1, T_2, \ldots, T_k) \in Y_f(W)$, where there exists some $1 < i < k$ such that the root $r_i$ of $T_i$ has its half-edge labels fixed by $f$.
Suppose that the root of $T$ is the root $r_1$ of $T_1$.
If we do not have the above modification, then the rooted tree $T'$ added to $Z_f(W)$ corresponding to $T$ is $(T_1, T_2, \ldots, T_i)$, where the root of $T'$ is also $r_1$. 

Given the assumption that all bipolar trees in $Y_f(W)$ admit correct labelings, it is not hard to see that $\class(T') = \class(T)$. For any correct labeling $\mathcal{L}'$ of $T'$, we can extend the labeling to a correct labeling $\mathcal{L}$ of $T$ by labeling the remaining vertices according to any arbitrary correct labeling of $H^\ast$. This is possible because the labeling of $r_i$ has been fixed. For any correct labeling $\mathcal{L}$ of $T$, we can obtain a correct labeling $\mathcal{L}'$ of $T'$ by restricting $\mathcal{L}$ to  $T'$.

%At the moment we consider a bipolar tree $H = H_1 \circ T \circ H_2$ when building $X$ from $W$ before applying the labeling function $f$, we will do the following things.
%\begin{itemize}
%    \item First check if we have processed some bipolar tree $H' =  H_1' \circ T' \circ H_2'$ with $\type(H_1) = \type (H_1')$, $\class(T) = \class(T')$, and $\type(H_2) = \type(H_2')$. If so, then we replace $H$ by $H'$.
%    \item Then we apply the labeling function $f$ to the middle vertex of $H$, i.e., the root of $T$, and then we apply the pumping lemma to both $H_1$ and $H_2$ to replace them with much longer bipolar trees of the same equivalence classes. Denote the target length by $w$.
%\end{itemize}
%Although the above modifications change the resulting trees, the equivalence classes remain unchanged. 

\paragraph{Simulating a $\LOCAL$ Algorithm} We will show that there is a labeling function $f$ that makes all the rooted trees in $W^\ast$ to admit correct solutions, where we apply the above three modifications in the construction of $W^\ast$. In view of the above discussion, such a function $f$ is feasible.

The construction of $W^\ast$ involves only a finite number $k$ of iterations of $Z_f$ applied to some previously constructed rooted trees. If we view the target length $w$ of the pumping lemma as a variable, then the size of a tree in $W^\ast$ can be upper bounded by $O(w^{k})$.

Suppose that we are given an arbitrary $n^{o(1)}$-round $\LOCAL$ algorithm $\mathcal{A}$ that solves $\Pi$. It is known~\cite{chang_pettie2019time_hierarchy_trees_rand_speedup,chang2020n1k_speedups} that randomness does not help for LCL problems on bounded-degree trees with round complexity $\Omega(\log n)$, so we assume that 
$\mathcal{A}$ is deterministic.

The runtime of $\mathcal{A}$ on a tree of size 
$O(w^{k})$ can be made to be at most $t = w/10$ if $w$ is chosen to be sufficiently large. 

We pick the labeling function $f$ as follows. 
Whenever we encounter a new bipolar tree $H$ to which we need to apply $f$, we simulate  $\mathcal{A}$ locally at the root $r^m$ of the middle rooted tree $T^m$ in $H$. Here we assume that the pumping lemma was applied before simulating $\mathcal{A}$. Moreover, when we do the simulation, we assume that the number of vertices is $n = O(w^{k})$. 

To make the simulation possible, we just locally generate arbitrary distinct identifiers in the radius-$t$ neighborhood $S$ of $r^m$. This is possible because $t = w/10$  is so small that for any vertex $v$ in any tree $T$ constructed by recursively applying $Z_f$ for at most $k$ iterations, the radius-$(t+1)$ neighborhood of $v$ intersects at most one such $S$-set. Therefore, it is possible to complete the identifier assignment to the rest of the vertices in $T$ so that for any edge $\{u,v\}$, the set of identifiers seen by $u$ and $v$ are distinct in an execution of a $t$-round algorithm. 

Thus, by the correctness of $\mathcal{A}$, the labeling of all edges $\{u,v\}$ are configurations in $\fE$ and the labeling of all vertices $v$ are configurations in $\fV$.
Our choice of $f$ implies that all trees in $W^\ast$ admit correct labeling, and so $f$ is feasible. Combined with \cref{lem:find_path_4}, we have proved that  $\local(n^{o(1)})$  $\Rightarrow$ $\ell$-full.

%  why using section* ? %

% perhaps expand this section --- e.g., explain why these questions are interesting. %

\section{Concluding Remarks}\label{sec:concludingRemarksTrees}

Here we collect some open problems we find interesting.

\begin{enumerate}
    \item What is the relation of $\ffiid, \measure, \fiid$? In particular is perfect matching in the class $\ffiid$ or $\measure$? 
    \item Are there LCL problems whose uniform complexity is nontrivial and slower than $\poly\log 1/\eps$?
    %\item What is the complexity of $\Delta+1$ edge coloring?
    %\item Complexity of 1-defect 2-coloring of 3-regular trees? \todo{drop?}
    %\item Let $\Pi$ be LCL.
    %Then $\Pi\in \tail$ if and only if $\overline{\Pi}\in \tail$.
    %Weaker question is if perfect matching is in $\tail$. \todo{maybe drop since we are already asking about perf match?}
    %\item Is it possible to find a subgraph that is a forest of lines and intersects every connected component in $\MEASURE$, or $\BOREL$? \todo{do we use it?}
    \item Suppose that $\Pi_G$ is a homomorphism LCL, is it true that $\Pi_G\in\borel$ if and only if $\Pi_G\in O(\log^*(n))$?
    \item Are there finite graphs of chromatic number bigger than $2\Delta-2$ such that the corresponding homomorphism LCL is not in $\borel$? See also Remark \ref{r:hedet}.
\end{enumerate}

\subsection*{Acknowledgements} 
We would like to thank Anton Bernshteyn, Endre Csóka, Mohsen Ghaffari, Jan Hladký, Steve Jackson, Alexander Kechris, Edward Krohne, Oleg Pikhurko, Brandon Seward, Jukka Suomela, and Yufan Zheng for insightful discussions. 
We especially thank Yufan Zheng for   suggesting potentially hard instances of homomorphism LCL problems.
%Yi-Jun Chang would like to thank Yufan Zheng for a discussion about homomorphism LCL problems, in particular for suggesting potentially hard instances.

% also put other people's funding information here? %

\bibliographystyle{alpha}
\bibliography{ref}

\appendix 

\section{Missing Proof of \cref{thm:FIIDinMEASURE}}
\label{app:ktulu}

\begin{proof}
We define an acyclic $\Delta$ regular graph $\fT$ on some standard probability space $(X,\mu)$ and then show that any measurable solution of $\Pi$ on $\fT$ implies that $\Pi\in\fiid$.

Suppose that there is a fixed distinguished vertex $*$ in $T_\Delta$, the root.
Let $\Aut^*(T_\Delta)$ be the subgroup that fixes the root.
Then it is easy to see that $\Aut^*(T_\Delta)$ is a compact group.
It is a basic fact that $Y=[0,1]^{T_\Delta}$ with the product Borel structure is a standard Borel space and the product Lebesgue measure $\lambda$ is a Borel probability measure on $Y$.
Consider the shift action $\Aut^*(T_\Delta)\curvearrowright Y$ defined as
$$\alpha\cdot x(v)=x(\alpha^{-1}(v)),$$
where $v\in T_\Delta$ and $\alpha\in \Aut^*(T_\Delta)$.
Since $\Aut^*(T_\Delta)$ is compact, we have that the quotient space
$$X:=[0,1]^{T_\Delta}/\Aut^*(T_\Delta)$$
is a standard Borel space, see \cite{kechrisclassical}.
Moreover, since the shift action preserves $\lambda$, we have that
$$\mu:=\lambda/\Aut^*(T_\Delta)$$
is a Borel probability measure on $X$.

We define $\fT$ on $(X,\mu)$ as follows.
Let $[x],[y]\in X$, where $x,y\in Y$ and $[{-}]$ denotes the $\Aut^*(T_\Delta)$-equivalence class.
We put $([x],[y])\in \fT$ if and only if there is a neighbor $v$ of the root in $x$ such that $y$ is isomorphic to $x$ with $v$ as the root.
Note that this definition is independent of the choice of representatives of $[x]$ and $[y]$.
Moreover, $\fT$ is acyclic and $\Delta$-regular with probability $1$.
A standard argument shows that every measurable solution to an LCL $\Pi$ on $\fT$ yields a fiid solution by composition with the quotient map.
\end{proof}

\section{Missing Proof of \cref{thm:classification_local_uniform_trees}} 
\label{sec:LLL}

\begin{proof}
%In this section
It remains to verify that the LLL algorithms of  Fischer and Ghaffari~\cite[Section 3.1.1]{fischer-ghaffari-lll} and   Chang~et~al.~\cite[Section 5.1]{ChangHLPU20} 
%in a \emph{graph shattering} framework~\cite{barenboim2016locality}
%Chang et al.~\cite[Section~5]{ChangHLPU20}  that solves the distributed LLL in $O(\log \log n)$ rounds (for a specific LLL criterion) on so-called tree structured dependency graphs 
can be turned into an algorithm that does not rely on the knowledge of $n$. As discussed in~\cite[Section 5.1]{ChangHLPU20}, these two algorithms can be combined together to solve the tree-structured LLL problem in $O(\log \log n)$ rounds when the underlying tree has bounded degree.

Before we explain the main ideas in these algorithms, consider the following setup. Let $\mathcal{A}_1$ and $\mathcal{A}_2$ be two distributed algorithms that do not rely  on knowing $n$ and that have a round complexity of $g_1(n)$ and $g_2(n)$, respectively. We now want to compose these two algorithms in the following sense. We are interested in the output of $\mathcal{A}_2$ when the input of $\mathcal{A}_2$ is equal to the output of $\mathcal{A}_1$. Then, we can compute this composition with a distributed algorithm that works without the knowledge of $n$ and which has a round complexity of $O(g_1(n) + g_2(n))$. This statement might seem obvious at first sight, but one needs to be careful. In particular, as the resulting algorithm needs to work without the knowledge of $n$ it cannot compute the value $g_1(n)$. Therefore, it cannot explicitly tell the vertices in which round they should start to run the second algorithm. However, this is not a problem as vertices can start to run the second algorithm as soon as their neighbors are ready. In subsequent rounds, vertices might need to temporarily stall as some neighbors might not have received all the necessary information from all of their neighbors and so on, but this is not a problem. From this discussion, it should be easy to see that the algorithm indeed finishes after $O(g_1(n) + g_2(n))$ rounds. In case the algorithms are randomized the failure probability of the resulting algorithm might increase up to a factor of $2$.

This composition result makes it easier to verify that the algorithms of~\cite{ChangHLPU20,fischer-ghaffari-lll} indeed works without the knowledge of $n$, as we can verify that this is the case for all its subroutines in isolation.

%The algorithm follows the \emph{graph shattering} framework~\cite{barenboim2016locality}, which consists of a pre-shattering phase and a post-shattering phase.

The pre-shattering phase of the LLL algorithm~\cite[Section 3.1.1]{fischer-ghaffari-lll} consists of first computing a $\poly(\Delta)$-coloring of $T^k$ --- the graph obtained from $T$ by connecting any two vertices of distance at most $k$ in $T$ by an edge --- for some $k = O(1)$. 
It directly follows from the results of  \cite{Korman_Sereni_Viennot2012Pruning_algorithms_+_oblivious_coloring,HolroydSchrammWilson2017FinitaryColoring} --- they adapt Linial's coloring algorithm in such a way that it works without the knowledge of $n$ --- that the $\poly(\Delta)$-coloring of $T^k$ can be computed without the knowledge of $n$. Afterwards, an $O(1)$-round routine follows that only uses the computed coloring as its input. Hence, the pre-shattering phase works without the knowledge of $n$. 
Once the pre-shattering phase is complete a subset of the vertices ``survive" and the post-shattering phase is executed on the graph induced by these vertices. Importantly, with high probability all the connected components of the induced graph have size $O(\log n)$. See~\cite[Lemma~6]{fischer-ghaffari-lll}.

The post-shattering phase of the LLL algorithm~\cite[Section 5.1]{ChangHLPU20} starts by decomposing each connected component with the following variant of the rake-and-compress process, which consists of a repeated application of the following two operations.

\begin{description}
    \item[\rake:] Remove all leaves and isolated vertices.
    \item[\compress:]  Remove all vertices that belong to some path $P$ such that (i) all vertices in $P$ have degree at most $2$ and (ii) the number of vertices in $P$ is at least a fixed constant $\ell$. 
\end{description}

Alternating these two operations on a tree with $N$ vertices decomposes the whole tree in $O(\log N)$ steps. In our case $N = O(\log n)$ with high probability and therefore the procedure finishes after $O(\log \log n)$ rounds with high probability. Importantly, the rake-and-compress algorithm works without the knowledge of $n$. At the end of the rake-and-compress procedure, the only important information each vertex needs to know is in which iteration it got removed.

After this information is computed, the algorithm of~\cite{ChangHLPU20} computes in $O(\log^* n)$ rounds a certain coloring variant on a certain subgraph of the input tree. This subgraph can be locally constructed given the rake-and-compress decomposition.
Again, it is a simple corollary of the results of \cite{Korman_Sereni_Viennot2012Pruning_algorithms_+_oblivious_coloring,HolroydSchrammWilson2017FinitaryColoring}  that this coloring variant can be computed without the knowledge of $n$ in $O(\log^* n)$ rounds.

The coloring variant together with the rake-and-compress decomposition is then used to compute a so-called $(2, O(\log N))$ network decomposition of the graph $T^k$ for some $k = O(1)$ in $O(\log N)$ rounds~\cite[Section~6.1]{ChangHLPU20}. We do not explain how this network decomposition is computed in detail, but the idea is to first compute the local output of all the vertices that got removed in the very last rake-and-compress iteration in $O(1)$ rounds. Directly afterwards, the local output of all the vertices that got removed one iteration before gets computed in additional $O(1)$ rounds and so on. Hence the local information of all  vertices is computed within $O(\log N) \cdot O(1) = O(\log N)$ rounds.  From this description, it is clear that this procedure works without the knowledge of $n$.

Finally, once the network decomposition of $T^k$ is computed, it directly follows from the algorithm description of~\cite{ChangHLPU20} that the final computation can be performed by a distributed algorithm without the knowledge of $n$ in $O(\log N)$ rounds.

Hence we have shown that the $O(\log \log n)$-round LLL algorithm~\cite{ChangHLPU20,fischer-ghaffari-lll} works without the knowledge of $n$.
%, as desired.
\end{proof}

\section{Missing Proof of \cref{thm:OneTwoEndBorel}} 
\label{sec:oneortwo}

\begin{proof}[Proof of \cref{thm:OneTwoEndBorel}]
Let $x$ be a vertex of degree strictly bigger than $2$.
A \emph{star} $S(x)$ around $x$ consists of $x$ and  vertices of degree $2$ that are connected with $x$ by a path with all inner vertices of degree $2$.
It is clear that $S(x)$ is connected in $\fG$.

Recall that if $F$ is a function we define the iterated preimage of a vertex $x$ as $F^{\leftarrow}(x)=\bigcup_{n\in \mathbb{N}} F^{-n}(x)$.
We first specify a canonical one ended orientation on $S(x)$, or, equivalently, a function $F$ with finite iterated preimages. (In what follows we interchange freely the notion of one ended orientation and function with finite iterated preimages.)
That is to say, if $S(x)$ is infinite we orient things towards infinity in a one ended fashion (that is always possible), otherwise we fix an orientation towards any of the boundary points, i.e., there is exactly one point that is directed outside of $S(x)$.
We refer to this orientation as \emph{the canonical orientation}.

The inductive construction produces an orientation of some vertices together with doubly infinite lines.
The orientation points either to infinity (is one-ended) or towards these doubly infinite lines, this takes $(\aleph_0+1)$-many steps.

For a graph $\fG'$ we define a graph $\fH'$ on the vertices of degree strictly bigger than $2$, where $(x,y)\in \fH'$ if there is a path from $x$ to $y$ in $\fG'$ that has at most one inner vertex of degree strictly bigger than $2$.
Since, in our situation, $\fH'$ is always Borel and has degree bounded by $\Delta^2$, we can pick a Borel maximal independent set $\fM'$ of $\fH'$ by \cite{KST}.

Inductively along $\mathbb{N}$ do the following:
Suppose that we are in stage $k\in \mathbb{N}$ and we have a Borel set $\fO_k$ and $\fG_k:=\fG\upharpoonright \fO_k$ with the property that every vertex has degree at least $2$.
We start with $\fG_0=\fG$ and $\fO_0=X$.
Pick a Borel maximal independent set $\fM_k$ in $\fH_k$ that is defined from $\fG_k$.
Add to the domain of $F$ every vertex from the star $S_k(x)$ in $\fG_k$, where $x\in \fM_k$, and set
$$\fO_{k+1}=\fO_k\setminus \bigcup_{x\in \fM_k}S(x).$$
Define $F$ so that it corresponds to the canonical orientation on each $S_k(x)$.
By the definition we have that $F(y)\in S_k(x)$ for every $y\in S_k(x)$, possibly up to one point $z\in S_k(x)$ that satisfies $F(z)\in \fO_{k+1}$.
It is easy to see that every $x\in \fO_{k+1}$ has degree at least $2$ in $\fG_{k+1}=\fG\upharpoonright \fO_{k+1}$.
This follows from the definition of $\fM_k$.

Set $\fO_{\infty}=\bigcap_{k\in \mathbb{N}} \fO_k$ and write $\fG_\infty:=\fG\upharpoonright\fO_\infty$.
It follows from the inductive (finitary) construction that every vertex $x\in \fO_\infty$ has degree at least $2$ in $\fG_\infty$.
Moreover, the function $F$ satisfies $\dom(F)=X\setminus \fO_\infty$ and  has finite iterated preimages.
This is because, first, for every $x\in \dom(F)$ there is $k\in \mathbb{N}$ such that $x\in \fO_k\setminus \fO_{k+1}$ and $F^{-1}(x)\subseteq X\setminus \fO_{k+1}$.
Second, $F^{\leftarrow}(x)\cap \fO_{k}$ is finite by the definition of $F$ on stars.
It follows inductively that $F^{\leftarrow}(x)\cap \fO_{l}$ is finite for every $l\le k$, hence, $F^\leftarrow(x)$ is finite, whenever $x\in \dom(F)$.
If $x\not\in \dom(F)$, then $F^{-1}(x)\subseteq \dom(F)$ is finite and the claim follows.

Let $x\in \fO_\infty$ have degree strictly bigger than $2$ and write $C_1,\dots,C_\ell$ for the connected components of $\fG_\infty\setminus \{x\}$ in the connected component of $x$ in $\fG_\infty$.
Note that $\ell\le \Delta$.
We claim that at least one of the sets $C_i$ is a one ended line.
Suppose not, and write $z_1,\dots,z_\ell$ for the closest splitting points in $C_1,\dots, C_\ell$ and $p_i$ for the paths that connect $x$ with $z_i$ for every $i\le \ell$.
There is $k\in \mathbb{N}$ large enough such that the degree of $x,z_1,\dots,z_\ell$ is the same in $\fG_\infty$ as in $\fG_k$ and $p_i$ is a path in $\fG_k$ whose inner vertices have degree $2$.
Note that $\fG_\infty\subseteq \fG_k$ because $\fO_\infty\subseteq \fO_k$.
Since $\fM_k$ was maximal in $\fH_k$, there is $y\in \fM_k$ such that $(x,y)\in \fH_k$.
This is because $x\not\in \fM_k$.
Similar reasoning implies that that $y\not=z_i$ for any $i\le \ell$.
Let $q$ be the path that connects $x$ and $y$ in $\fG_k$.
By the choice of $k$ we have that $q$ extends one of the paths $p_i$.
Let $y'$ be the last point on $q$ such that $y'\in \fO_\infty$.
Since the degree of $z_i$ is the same in $\fG_k$ as in $\fG_\infty$ we have that $y'\not=z_i$.
The degree of $y'$ in $\fG_\infty$ is at least $2$.
Suppose it were $3$ in $\fG_k$, then $y'=y$ because we must have $(x,y)\in \fH_k$.
In that case removing $S(y)$ would decrease the degree of $z_i$ in $\fG_{k+1}$ and consequently in $\fG_\infty$.
Therefore $y'$ has degree $2$ in $\fG_k$.
But then $y'$ is not the last vertex on $q$ such that $y'\in \fO_\infty$, i.e., the other neighbor of $y'$, that is the same in $\fG_k$ and in $\fG_\infty$ must be a vertex on $q$, a contradiction.

Consider now the graph $\fH_\infty$, where $(x,y)\in \fH_\infty$ if and only if $x,y\in \fO_\infty$ have degree strictly bigger than $2$ and there is a path from $x$ to $y$ in $\fG_\infty$ with all inner points having degree $2$ in $\fG_\infty$.
Consider any Borel $(\Delta+1)$-coloring of $\mathcal{H}_\infty$ a decomposition of all vertices of degree strictly bigger than $2$ in $\fG_\infty$ into $\fH_\infty$-independent Borel sets $D_0,\dots,D_\Delta$.
Define a one ended orientation of stars of the form $S(x)$, where $x\in D_i$, inductively according to $i\le \Delta$ using the canonical orientation, i.e., in step $i\le \Delta$ we work with the graph
$$\fG\upharpoonright \left(\fO_{\infty}\setminus \bigcup_{j<i}\bigcup_{x\in D_j} S(x)\right).$$
Note that by the previous argument we have that every such star $S(x)$ is infinite and consequently, this defines a valid extension of $F$, i.e., iterated preimages are still finite.
It remains to realize that complement of $\dom(F)$ consists of doubly infinite lines.
This finishes the proof.
\end{proof}

\end{document}